\documentclass[10pt,a4paper,oneside,reqno]{amsart}

\usepackage{setspace}
\usepackage[totalwidth=14cm,totalheight=22cm]{geometry}
\usepackage[T1]{fontenc}
\usepackage[dvips]{graphicx}
\usepackage{epsfig}
\usepackage{amsmath,amssymb,amscd,amsthm}
\usepackage{subfigure}
\usepackage[latin1]{inputenc}
\usepackage{latexsym}
\usepackage{graphics}
\usepackage{colortbl} 
\usepackage{color}
\usepackage{ae}
\usepackage{enumitem}
\usepackage[all]{xy}
\usepackage{scalerel}
\usepackage{tikz}
\usepackage{cancel}
\usepackage{bbm,amsbsy}
\usepackage{mathrsfs}  
\usepackage[colorlinks=true,pagebackref=true]{hyperref}
\usepackage[normalem]{ulem}
\usepackage{nicefrac}
\usepackage{xfrac}
\usepackage{faktor}
\usepackage{tikz-cd} 

\usepackage{nomencl}
\usepackage[normalem]{ulem}

\makenomenclature

\makeatletter
\newtheorem*{rep@theorem}{\rep@title}
\newcommand{\newreptheorem}[2]{%
\newenvironment{rep#1}[1]{%
 \def\rep@title{#2 \ref{##1}}%
 \begin{rep@theorem}}%
 {\end{rep@theorem}}}
\makeatother

\makeatletter
\newtheorem*{rep@cor}{\rep@title}
\newcommand{\newrepcor}[2]{%
\newenvironment{rep#1}[1]{%
 \def\rep@title{#2 \ref{##1}}%
 \begin{rep@cor}}%
 {\end{rep@cor}}}
\makeatother

\makeatletter
\newtheorem*{rep@prop}{\rep@title}
\newcommand{\newrepprop}[2]{%
\newenvironment{rep#1}[1]{%
 \def\rep@title{#2 \ref{##1}}%
 \begin{rep@prop}}%
 {\end{rep@prop}}}
\makeatother

\newtheorem{cor}{Corollary}[section]
\newtheorem{corx}{Corollary}

\newtheorem{theorem}[cor]{Theorem}
\newtheorem{thmx}[corx]{Theorem}

\newtheorem{prop}[cor]{Proposition}

\newrepcor{cor}{Corollary}
\newreptheorem{theorem}{Theorem}
\newrepprop{prop}{Proposition}

\newtheorem{lemma}[cor]{Lemma}

\theoremstyle{definition}
\newtheorem{defi}[cor]{Definition}
\theoremstyle{remark}
\newtheorem{remark}[cor]{Remark}
\newtheorem*{remark*}{Remark}
\newtheorem{example}[cor]{Example}

\newtheorem*{notation*}{Notation}

\newlist{steps}{enumerate}{1}
\setlist[steps, 1]{itemsep=8pt,leftmargin=0cm,itemindent=.5cm,labelwidth=\itemindent,labelsep=0cm,align=left,label = \textbf{\emph{Step \arabic*}:\,}}

\newcommand{\R}{{\mathbb R}}
\newcommand{\Z}{{\mathbb Z}}

\newcommand{\hol}{\mathrm{hol}}
\newcommand{\Hyp}{\mathbb{H}}
\newcommand{\Sph}{\mathbb{S}}

\newcommand{\SO}{\mathrm{SO}}

\newcommand{\G}[1]{\mathcal G(\Hyp^{#1})}
\newcommand{\gs}[1]{g_{T^1\Hyp^{#1}}}
\newcommand{\JJ}{\mathbbm{J}}
\newcommand{\GG}{\mathbbm{G}}

\newcommand{\ddt}{\left.\frac{d}{dt}\right|_{t=0}}
\newcommand{\Isom}{\mathrm{Isom}^+}

\newcommand{\Ham}{\mathrm{Ham}}
\newcommand{\Flux}{\mathrm{Flux}}

\newcommand{\V}{\mathcal{V}}
\newcommand{\VP}{\mathcal{V}^0}
\newcommand{\HH}{\mathcal{H}}
\newcommand{\HP}{\mathcal{H}^0}

\newcommand{\I}{\mathrm{I}}
\newcommand{\II}{\mathrm{I}\hspace{-0.04cm}\mathrm{I}}
\newcommand{\III}{\mathrm{I}\hspace{-0.04cm}\mathrm{I}\hspace{-0.04cm}\mathrm{I}}

\newcommand{\arctanh}{\mathrm{arctanh}}

\newcommand{\inner}[1] {\langle #1 \rangle}

\newcommand{\changelocaltocdepth}[1]{%
  \addtocontents{toc}{\protect\setcounter{tocdepth}{#1}}%
  \setcounter{tocdepth}{#1}%
}

\begin{document}

\setcounter{secnumdepth}{2}
\setcounter{tocdepth}{1}

\title[On the Gauss map of equivariant immersions in $\Hyp^{n}$]{On the Gauss map of equivariant immersions in hyperbolic space}

\author[Christian El Emam]{Christian El Emam}
\address{Christian El Emam: Department of Mathematics, University of Luxembourg, 6 avenue de la Fonte, L-4364 Esch-Sur-Alzette, Luxembourg.} \email{christian.elemam@uni.lu}

\author[Andrea Seppi]{Andrea Seppi}
\address{Andrea Seppi: Institut Fourier, UMR 5582, Laboratoire de Math\'ematiques,
Universit\'e Grenoble Alpes, CS 40700, 38058 Grenoble cedex 9, France.} \email{andrea.seppi@univ-grenoble-alpes.fr}

%\date{\today}

\thanks{The authors are members of the national research group GNSAGA. The first author has been partially supported by Blue Sky Research project "Analytic and geometric
properties of low-dimensional manifolds" and by the FNR OPEN grant CoSH
(O20/14766753/CoSH)}

\maketitle

\begin{abstract}
Given an oriented immersed hypersurface in hyperbolic space $\mathbb H^{n+1}$, its  Gauss map is  defined with values in the space of oriented geodesics of $\mathbb H^{n+1}$, which is endowed with a natural para-K\"ahler structure. In this paper we address the question of whether an immersion $G$ of the universal cover of an $n$-manifold $M$, equivariant for some group representation of $\pi_1(M)$ in $\mathrm{Isom}(\Hyp^{n+1})$, is the Gauss map of an equivariant immersion in $\Hyp^{n+1}$. We fully answer this question for immersions with principal curvatures in $(-1,1)$: while the only local obstructions are the conditions that $G$ is Lagrangian and Riemannian, the global obstruction is more subtle, and we provide two characterizations, the first in terms of the Maslov class, and the second (for $M$ compact) in terms of the action of the group of compactly supported Hamiltonian symplectomorphisms.
\end{abstract}

\tableofcontents

\section{Introduction}

The purpose of the present paper is to study immersions of hypersurfaces in the hyperbolic space $\Hyp^{n+1}$, in relation with the geometry of their Gauss maps in the space of oriented geodesics of $\Hyp^{n+1}$. We will mostly restrict to immersions having principal curvatures in $(-1,1)$, and our main aim is to study immersions of $\widetilde M$ which are equivariant with respect to some group representation $\rho:\pi_1(M)\to\Isom(\Hyp^{n+1})$, for $M$ a $n$-manifold. The two main results in this direction are Theorem \ref{teorema hol H baby} and Theorem \ref{thm:second char ham}: the former holds for any $M$, while the latter under the assumption that $M$ is closed.

\subsection{Context in literature}
 In the groundbreaking paper \cite{zbMATH03791301}, Hitchin observed the existence of a natural complex structure on the space of oriented geodesics in Euclidean three-space.  A large interest has then grown on the geometry of the space of oriented (maximal unparametrized) geodesics of Euclidean space of any dimension (see \cite{zbMATH02228497,zbMATH02232144,zbMATH05530808,zbMATH06300570}) and of several other Riemannian and pseudo-Riemannian manifolds (see \cite{zbMATH05988505,zbMATH06268759,zbMATH06828623,MR3888623,zbMATH07050784}). In this paper, we are interested in the case of hyperbolic $n$-space $\Hyp^n$, whose  space of oriented geodesics is denoted here by $\G{n}$. The  geometry of $\G{n}$ has been addressed in \cite{zbMATH05187891} and, for $n=3$, in \cite{zbMATH05821527,zbMATH05779824,zbMATH06173011,zbMATH06481975}.  For the purpose of this paper, the most relevant geometric structure on $\G{n}$ is a natural \emph{para-K\"ahler structure} $(\GG, \JJ, \Omega)$ (introduced in \cite{zbMATH05988505,zbMATH06268759}), a notion which we will describe in Section \ref{sec intro 3} of this introduction and more in detail in Section \ref{sec:parakahler metric GG}.  A particularly relevant feature of such para-K\"ahler structure is the fact that the Gauss map of an oriented immersion $\sigma:M\to\Hyp^n$, which is defined as the map that associates to a point of $M$ the orthogonal geodesic of $\sigma$ endowed with the compatible orientation, is a Lagrangian immersion in $\G{n}$. We will come back to this important point in Section \ref{sec intro 1}. Let us remark here that, as a consequence of the geometry of the hyperbolic space $\Hyp^n$, an oriented geodesic in $\Hyp^n$ is characterized, up to orientation preserving reparametrization, by the ordered couple of its ``endpoints'' in the visual boundary $\partial \Hyp^n$: this gives an identification 
$
\G{n}\cong\partial \Hyp^n\times\partial \Hyp^n \setminus \Delta
$. Under this identification the Gauss map $G_\sigma$ of an immersion $\sigma:M\to\Hyp^n$ corresponds to a pair of  \emph{hyperbolic Gauss maps} $G_\sigma^\pm:M\to\partial \Hyp^n$.

A parallel research direction, originated by the works of Uhlenbeck \cite{zbMATH03840752} and Epstein \cite{Epstein:1986aa,zbMATH03948676,zbMATH04015637}, concerned the study of immersed hypersurfaces in $\Hyp^n$, mostly in dimension $n=3$. These works highlighted the relevance of hypersurfaces satisfying the geometric condition {for which principal curvatures are everywhere}  different from $\pm 1$, sometimes called \emph{horospherically convexity}: this is the condition that ensures that the hyperbolic Gauss maps $G_\sigma^\pm$ are locally invertible.  On the one hand, Epstein developed this point of view to give a description ``from infinity'' of horospherically convex hypersurfaces as envelopes of horospheres. This approach has been pursued by many authors by means of analytic techniques, see for instance \cite{zbMATH01773370,zbMATH05119776,zbMATH05279031}, and permitted to obtain remarkable classification results often under the assumption that the principal curvatures are larger than 1 in absolute value (\cite{zbMATH04113461,zbMATH04190632,zbMATH00431673,zbMATH05608949,zbMATH05815812,zbMATH06441147,zbMATH06788668,zbMATH06987704}). On the other hand, Uhlenbeck considered the class of so-called \emph{almost-Fuchsian manifolds}, which has been largely studied in \cite{zbMATH05200423,zbMATH05851259,zbMATH06050482,zbMATH06204974,zbMATH06460565,zbMATH06663964,zbMATH06653236} afterwards. These are complete hyperbolic manifolds diffeomorphic to $S\times\R$, for $S$ a closed orientable surface of genus $g\geq 2$, containing a minimal surface with principal curvatures different from $\pm 1$. These surfaces lift on the universal cover to immersions $\sigma:\widetilde S\to\Hyp^3$ which are equivariant for a quasi-Fuchsian representation $\rho:\pi_1(S)\to\Isom(\Hyp^3)$ and, by the Gauss-Bonnet formula, have principal curvatures in $(-1,1)$, a condition to which we will refer as having \emph{small principal curvatures}.

\subsection{Integrability of immersions in $\G n$}\label{sec intro 1}

One of the main goals of this paper is to discuss when an immersion $G\colon M^n\to \G{n+1}$ is \emph{integrable}, namely when it is the Gauss map of an immersion $M\to \Hyp^{n+1}$, in terms of the geometry of $\G{n+1}$. We will distinguish three types of integrability conditions, which we list from the weakest to the strongest: \begin{itemize}
	\item An immersion $G\colon M\to \G {n+1}$ is \emph{locally integrable} if for all $p\in M$ there exists a neighbourhood $U$ of $p$ such that $G|_{U}$ is the Gauss map of an immersion $U\to \Hyp^{n+1}$;
	\item An immersion $G\colon M \to \G{n+1}$ is \emph{globally integrable} if it is the Gauss map of an immersion $M\to \Hyp^{n+1}$;
	\item Given a representation $\rho\colon \pi_1(M)\to \Isom (\Hyp^{n+1})$, a $\rho$-equivariant immersion $G\colon \widetilde M\to \G{n+1}$ is \emph{$\rho$-integrable} 
	if it is the Gauss map of a $\rho$-equivariant immersion $\widetilde M \to \Hyp^{n+1}$.
\end{itemize}

Let us clarify here that, since the definition of Gauss map requires to fix an orientation on $M$ (see Definition \ref{defi:lift and gauss}),  the above three definitions of integrability have to be interpreted as: ``there exists an orientation on $U$ (in the first case) or $M$ (in the other two) such that $G$ is the Gauss map of an immersion in $\Hyp^{n+1}$ with respect to that orientation''.

We will mostly restrict to immersions $\sigma$ with small principal curvatures, which is equivalent to the condition that the Gauss map $G_\sigma$ is Riemannian, meaning that the pull-back by $G_\sigma$ of the ambient pseudo-Riemannian metric of $\G{n+1}$ is positive definite, hence a Riemannian metric (Proposition \ref{prop: small curv sse riemannian}). 

\subsubsection{Local integrability} As it was essentially observed in \cite[Theorem 2.10]{zbMATH06268759}, local integrability admits a very simple characterization in terms of the symplectic geometry of $\G{n+1}$.
\begin{reptheorem}{cor: local integrability}
Let $M^n$ be a manifold and $G\colon M\to \G {n+1}$ be an immersion. Then $G$ is locally integrable if and only if it is Lagrangian.
\end{reptheorem}
 The methods of this paper easily provide a proof of Theorem \ref{cor: local integrability}, which is independent from the content of \cite{zbMATH06268759}.
{Let us denote by $T^1\Hyp^{n+1}$  the unit tangent bundle of $\Hyp^{n+1}$ and {by}
\begin{equation}\label{defi p intro}
\mathrm p\colon T^1\Hyp^{n+1}\to \G{n+1}~,
\end{equation} the map such that $\mathrm{p}(x,v)$ is the oriented geodesic of $\Hyp^{n+1}$ tangent to $v$ at $x$. Then, if $G$ is Lagrangian, we prove that one can locally construct maps $\zeta:U\to T^1\Hyp^{n+1}$ (for $U$ a simply connected open set) such that $\mathrm p\circ\zeta=G$. 
Up to restricting the domain again, one can find such a $\zeta$ so that it projects to an immersion $\sigma$ in $\Hyp^{n+1}$ (Lemma \ref{lemma:desingularize for small t}), and the Gauss map of $\sigma$ is $G$ by construction.} 
 
 Our next results are, to our knowledge, completely new and give characterizations of global integrability and $\rho$-integrability under the assumption of small principal curvatures.

\subsubsection{Global integrability} The problem of global integrability is in general more subtle than local integrability. As a matter of fact, in Example \ref{ex: Lagrangian not globally integrable} we construct an example of a locally integrable immersion $G:(-T,T)\to\G{2}$ that is not globally integrable. By taking a cylinder on this curve, one easily sees that the same phenomenon occurs in any dimension. We stress that in our example $M=(-T,T)$ (or the product $(-T,T)\times\R^{n-1}$ for $n>2$) is simply connected: {the key point in our example is that one can find globally defined maps $\zeta:M\to T^1\Hyp^{n+1}$ such that $G=\mathrm p\circ\zeta$, but no such $\zeta$ projects to an immersion in $\Hyp^{n+1}$. }

Nevertheless, we show that this issue does not occur for Riemannian immersions $G$. In this case any immersion $\sigma$ whose Gauss map is $G$ (if it exists) necessarily has small principal curvatures. We will always restrict  to this setting hereafter. In summary, we have the following characterization of global integrability for $M$ simply connected:
\begin{reptheorem}{prop: riemannian global integrability}
Let $M^n$ be a simply connected manifold and $G\colon M\to \G {n+1}$ be a Riemannian immersion. Then $G$ is globally integrable if and only if it is Lagrangian.
\end{reptheorem}

We give a characterization of global integrability for $\pi_1(M)\ne \{1\}$ in Corollary \ref{cor hol H baby}, which is a direct consequence of our first characterization of $\rho$-integrability (Theorem \ref{teorema hol H baby}). Anyway, we remark that if a Riemannian and Lagrangian immersion  $G\colon M\to \G {n+1}$ is also complete (i.e. has complete first fundamental form), then $M$ is necessarily simply connected:

\begin{reptheorem}{cor G complete}
Let $M^n$ be a manifold. If $G\colon M\to \G {n+1}$ is a complete Riemannian and Lagrangian immersion, then $M$ is diffeomorphic to $\R^n$ and $G$ is the Gauss map of a proper embedding $\sigma:M\to\Hyp^{n+1}$.
\end{reptheorem}

In Theorem \ref{cor G complete} the conclusion that $G=G_\sigma$ for $\sigma$ a proper embedding follows from the fact that $\sigma$ is complete, which is an easy consequence of Equation \eqref{eq:fff gauss} relating the first fundamental forms of $\sigma$ and $G_\sigma$, and the non-trivial fact that complete immersions in $\Hyp^{n+1}$ with small principal curvatures are proper embeddings (Proposition \ref{prop injectivity}).

\subsubsection{$\rho$-integrability}  {Let us first observe that the problem of $\rho$-integrability presents some additional difficulties than global integrability}. Assume $G\colon \widetilde M \to \G{n+1}$ is a Lagrangian, Riemannian and $\rho$-equivariant immersion for some representation $\rho\colon \pi_1(M^n)\to \Isom(\Hyp^{n+1})$.
Then, by Theorem \ref{prop: riemannian global integrability}, there exists $\sigma\colon \widetilde M\to\Hyp^{n+1}$ with Gauss map $G$, {but the main issue is that such a $\sigma$ will not be}  $\rho$-equivariant in general, as one can see in Examples \ref{ex: global integrable non equivariant} and \ref{ex: global integrable non equivariant2}.

Nevertheless, $\rho$-integrability of Riemannian immersions into $\G{n+1}$ can still be characterized in terms of their extrinsic geometry. Let $\overline {\mathrm H}$ be the mean curvature vector of $G$, defined as the trace of the second fundamen{tal form, and $\Omega$ the symplectic form of $\G{n+1}$. Since $G$ is $\rho$-equivariant, the $1$-form $G^*(\Omega(\overline {\mathrm H},\cdot))$ on $\widetilde M$ is invariant under the action of $\pi_1(M)$, so it descends to a $1$-form on $M$. 
	One can prove that such $1$-form on $M$ is closed (Corollary \ref{cor:maslov closed}):
	we will denote its cohomology class in $H^1_{dR}(M,\R)$ with $\mu_G$ and we will call it the \emph{Maslov class} of $G$, in accordance with some related interpretations of the Maslov class in other geometric contexts (see among others \cite{zbMATH03730903,zbMATH00704931,zbMATH01523513,zbMATH01786838}). 
	The Maslov class encodes the existence of equivariantly integrating immersions, in the sense stated in the following theorem.

	\begin{reptheorem}{teorema hol H baby}
		Let $M^n$ be an {orientable} manifold, $\rho\colon \pi_1(M) \to \Isom(\Hyp^{n+1})$ be a representation and $G\colon \widetilde M \to \G{n+1}$ be a $\rho$-equivariant Riemannian and Lagrangian {immersion}. Then $G$   is $\rho$-integrable if and only if the Maslov class $\mu_G$ vanishes.
	\end{reptheorem}

Applying Theorem \ref{teorema hol H baby} to a trivial representation, we immediately obtain a characterization of global integrability for Riemannian immersions, thus extending Theorem \ref{prop: riemannian global integrability} to the case $\pi_1(M)\neq \{1\}$.

\begin{repcor}{cor hol H baby}
Let $M^n$ be an {orientable} manifold and $G\colon M \to \G{n+1}$ be a Riemannian and Lagrangian immersion. Then $G$ is globally integrable if and only if its Maslov class $\mu_G$ vanishes.
\end{repcor}

\subsection{Nearly-Fuchsian representations}\label{sec intro 2}

Let us now focus on the case of $M$ a closed oriented manifold. Although our results apply to any dimension, we borrow the terminology from the three-dimensional case (see \cite{zbMATH06204974}) and say that a representation $\rho\colon \pi_1(M)\to \Isom(\Hyp^{n+1})$ is \emph{nearly-Fuchsian} if there exists a $\rho$-equivariant immersion $\sigma\colon\widetilde M\to \Hyp^{n+1}$ with small principal curvatures. We show (Proposition \ref{prop:action free prop disc0}) that the action of a nearly-Fuchsian representation on $\Hyp^{n+1}$ is free, properly discontinuously and convex cocompact; the quotient of $\Hyp^{n+1}$ by $\rho(\pi_1(M))$ is called \emph{nearly-Fuchsian manifold}. 

Moreover, the action of $\rho(\pi_1(M))$ extends to a free and {properly} discontinuous action on the complement of a topological $(n-1)$-sphere $\Lambda_\rho$ (the \emph{limit set} of $\rho$) in the visual boundary $\partial\Hyp^{n+1}$. Such complement is the disjoint union of two topological $n$-discs which we denote by $\Omega_+$ and $\Omega_-$.  It follows that there exists a maximal open region of $\G {n+1}$ over which a nearly-Fuchsian representation $\rho$ acts freely and properly discontinuously. This region is defined as the subset of $\G{n+1}$ consisting of oriented geodesics having either final endpoint in $\Omega_+$ or initial endpoint in $\Omega_-$. The quotient of this open region via the action of $\rho$, that we denote with $\mathcal G_\rho$, inherits a para-K\"ahler structure.

Let us first state a uniqueness result concerning nearly-Fuchsian representations. A consequence of Theorem \ref{teorema hol H baby} and the definition of Maslov class is that if $G$ is a $\rho$-equivariant, Riemannian and Lagrangian immersion which is furthermore \emph{minimal}, i.e. with $\overline {\mathrm{H}}=0$, then it is $\rho$-integrable. Together with an application of a maximum principle in the corresponding nearly-Fuchsian manifold, we prove:

\begin{repcor}{cor:uniqueness min lag}
	Given a closed orientable manifold $M^n$ and a representation $\rho:\pi_1(M)\to\Isom(\Hyp^{n+1})$, there exists at most one  $\rho$-equivariant {Riemannian} minimal Lagrangian immersion $G:\widetilde M\to\G{n+1}$ up to reparametrization. If such a $G$ exists, then $\rho$ is nearly-Fuchsian and $G$ induces a {minimal Lagrangian} embedding of $M$ in $\mathcal G_\rho$.
\end{repcor}

In fact, for any $\rho$-equivariant immersion $\sigma:\widetilde M\to\Hyp^{n+1}$ with small principal curvatures, the hyperbolic Gauss maps $G_\sigma^\pm$ are equivariant diffeomorphisms between $\widetilde M$ and $\Omega_\pm$. Hence up to changing the orientation of $M$, which corresponds to swapping the two factors {$\partial \Hyp^{n+1}$} in the identification $\G{n+1}\cong\partial\Hyp^{n+1}\times\partial\Hyp^{n+1}\setminus\Delta$, the Gauss map of $\sigma$ takes values in the maximal open region defined above, and induces an embedding of $M$ in $\mathcal G_\rho$. 

This observations permits to deal (in the cocompact case) with embeddings in $\mathcal G_\rho$ instead of $\rho$-equivariant embeddings  in $\G{n+1}$. In analogy with the definition of $\rho$-integrability defined above, we will say that a $n$-dimensional
 submanifold $\mathcal L\subset \mathcal G_\rho$ is $\rho$-\emph{integrable} if it is the image in the quotient of a $\rho$-integrable embedding in $\G{n+1}$. Clearly such $\mathcal L$ is necessarily Lagrangian by Theorem \ref{cor: local integrability}. We are now ready to state our second characterization result for $\rho$-integrability .

\begin{reptheorem}{thm:second char ham}
	Let $M$ be a closed orientable $n$-manifold, $\rho:\pi_1(M)\to\Isom(\Hyp^{n+1})$ be a nearly-Fuchsian representation and $\mathcal L\subset\mathcal G_\rho$ a Riemannian $\rho$-integrable submanifold. Then a Riemannian submanifold $\mathcal L'$ is $\rho$-integrable if and only if there exists $\Phi\in \mathrm{Ham}_c(\mathcal G_\rho,\Omega)$ such that $\Phi(\mathcal L)=\mathcal L'$.
\end{reptheorem}

{In Theorem \ref{thm:second char ham} we denoted} by $\Ham_c(\mathcal G_\rho,\Omega)$ the group of compactly-supported \emph{Hamiltonian symplectomorphisms} of $\mathcal G_\rho$ with respect to its symplectic form $\Omega$. (See Definition \ref{defi Hamc}). The proof of Theorem \ref{thm:second char ham} in fact shows that if $\mathcal L$ is $\rho$-integrable and $\mathcal L'=\Phi(\mathcal L)$ for $\Phi\in \mathrm{Ham}_c(\mathcal G_\rho,\Omega)$, then $\mathcal L'$ is integrable as well, even without the hypothesis that $\mathcal L$ and $\mathcal L'$ are Riemannian submanifolds.

If $\rho$ admits an equivariant Riemannian minimal Lagrangian immersion, then Theorem \ref{thm:second char ham} can be restated by saying that a Riemannian and Lagrangian submanifold $\mathcal L'$ is $\rho$-integrable if and only if it is in the $\Ham_c(\mathcal G_\rho,\Omega)$-orbit of \emph{the} minimal Lagrangian submanifold $\mathcal L\subset\mathcal G_\rho$, which is unique by Theorem \ref{cor:uniqueness min lag}.

\subsection{The geometry of $\G{n}$ and $T^1\Hyp^n$}
\label{sec intro 3}

Let us now discuss more deeply the geometry of the space of oriented geodesics of $\Hyp^{n}$ and some of the tools involved in the proofs. In this paper we give an alternative construction of the para-K\"ahler structure of $\G{n}$ with respect to the previous literature (\cite{zbMATH05187891,zbMATH05779824,zbMATH05988505,zbMATH06268759}), which is well-suited for the problem of (equivariant) integrability.
The geodesic flow induces a natural principal $\R$-bundle structure {whose total space is $T^1\Hyp^{n+1}$ and whose bundle map is $\mathrm{p}:T^1\Hyp^{n+1}\to \G{n+1}$ defined in Equation \eqref{defi p intro}}, and acts by isometries of the \emph{para-Sasaki metric} $g$, which is a pseudo-Riemannian version of the classical Sasaki metric on $T^1\Hyp^{n+1}$. Let us denote by $\chi$ the infinitesimal generator of the geodesic flow, which is a vector field on $T^1\Hyp^{n+1}$ tangent to the fibers of $\mathrm p$. The idea is to define each element that constitutes the para-K\"ahler structure of $\G{n+1}$ (see the items below) by push-forward of certain tensorial quantities defined on the $g$-orthogonal complement of $\chi$, showing that the push-forward is well defined by invariance under the action of the geodesic flow. More concretely:

\begin{itemize}
	\item The \emph{pseudo-Riemannian metric} $\GG$ of $\G{n+1}$ (of signature $(n,n)$) is defined as push-forward of the restriction of $g$ to $\chi^\perp$;
	\item The \emph{para-complex structure} $\JJ$ (that is, a $(1,1)$ tensor whose square is the identity and whose $\pm 1$-eigenspaces are integrable distributions of the same dimension) is obtained from an endomorphism $J$ of $\chi^\perp$, invariant under the geodesic flow, which essentially switches {the} horizontal and vertical {distributions} in $T^1\Hyp^{n+1}$;
	\item The \emph{symplectic form} $\Omega$ arises from a similar construction on $\chi^\perp$, in such a way that $\Omega(X,Y)=\GG(X,\mathbb J Y)$.
\end{itemize}

It is worth mentioning  that in dimension 3, the pseudo-Riemannian metric $\GG$ of $\G 3$ can be seen as the real part of a holomorphic Riemannian manifold of constant curvature $-1$, see \cite{bee}.

The symplectic geometry of $\G{n+1}$ has a deep relation with the structure of $T^1\Hyp^{n+1}$. Indeed 
the total space of $T^1\Hyp^{n+1}$ is endowed with a connection form 
 $\omega$, whose kernel consists precisely of $\chi^\perp$ (See Definition \ref{defi connection form}). In Proposition \ref{prop:identity curvature form} we prove the following fundamental relation between the curvature of $\mathrm p$ and the symplectic form $\Omega$:
\begin{equation}\label{eq omega Omega intro}
d\omega=\mathrm p^*\Omega~.
\end{equation}
This identity is an essential point in the proofs of our main results, which we now briefly outline.

\subsection{Overview of the proofs}
Let us start by Theorem \ref{cor: local integrability}, namely the equivalence between locally integrable and Lagrangian. Given a locally integrable immersion $G:M\to\G {n+1}$, the corresponding (local) immersions $\sigma:U\to\Hyp^{n+1}$ provide \emph{flat} sections  of the principal $\R$-bundle obtained by pull-back of the bundle $\mathrm p:T^1\Hyp^{n+1}\to\G{n+1}$ by $G$. Hence the obstruction to local integrability is precisely the curvature of the pull-back bundle  $G^*\mathrm p$. By Equation \eqref{eq omega Omega intro}, it follows that the vanishing of $G^*\Omega$ is precisely the condition that characterizes local integrability of $G$.

Moreover, $\rho$-integrability of a $\rho$-equivariant Lagrangian immersion $G:\widetilde M\to\G {n+1}$ can be characterized by the condition that the quotient of the bundle  $G^*\mathrm p$ by the action of $\pi_1(M)$ induced by $\rho$ is a \emph{trivial} flat bundle over $M$, meaning that it admits a \emph{global} flat section.  
Once these observations are established, Theorem \ref{teorema hol H baby} will be deduced as a consequence of Theorem \ref{Teorema hol H} which states that $\mu_G$ is dual, in the sense of de Rham Theorem, to the holonomy of such flat bundle over $M$. In turn, Theorem \ref{Teorema hol H} relies on the important expression (proved in Proposition \ref{Prop: formula H in G}):
\begin{equation}\label{eq maslov intro}
G_\sigma^*(\Omega (\overline {\mathrm H}, \cdot ) )=df_\sigma~,
\end{equation} where $G_\sigma$ is the Gauss map of an immersion $\sigma$ in $\Hyp^{n+1}$ {and $f_\sigma$ is the function defined by} 
\begin{equation}\label{eq fsigma intro}
f_\sigma=\frac{1}{n}\sum_{i=1}^n\arctanh\lambda_i~.
\end{equation}
where $\lambda_1,\ldots,\lambda_n$ are the principal curvatures of $\sigma$.
\vspace{5pt}

Let us move on to a sketch of the proof of Theorem \ref{thm:second char ham}, which again relies on the reformulation of $\rho$-integrability in terms of triviality of flat bundles. Assuming {that} $\mathcal L$ is a $\rho$-integrable submanifold of $\mathcal G_\rho$ and {that} we have a Lagrangian isotopy connecting  $\mathcal L$ to another Lagrangian submanifold $\mathcal L'$, Proposition \ref{prop: flux = diff holonomy} states that the holonomy of the flat bundle associated to $\mathcal L'$ is dual, again in the sense of de Rham Theorem, to the cohomology class of a 1-form which is built out of the Lagrangian isotopy, by a variant for Lagrangian submanifolds of the so-called \emph{flux homomorphism}. This variant has been developed in \cite{solomon} and applied in \cite{zbMATH07050784} for a problem in the Anti-de Sitter three-dimensional context which is to some extent analogous to those studied here. However, in those works stronger topological conditions are assumed which are not applicable here, and therefore our proof of  Theorem \ref{thm:second char ham} uses independent methods.

To summarize the proof, one implication is rather straightforward: if there exists a compactly supported Hamiltonian symplectomorphism $\Phi$ mapping $\mathcal L$ to $\mathcal L'$, then a simple computation shows that the flux homomorphism vanishes along the Hamiltonian isotopy connecting the identity to $\Phi$. This implication does not even need the assumption that $\mathcal L$ and $\mathcal L'$ are Riemannian submanifolds. The most interesting implication is the converse one: assuming that both $\mathcal L$ and $\mathcal L'$ are Riemannian and integrable, we use a differential geometric construction in $\Hyp^{n+1}$ to produce an interpolation between the corresponding hypersurfaces in the nearly-Fuchsian manifold associated to $\rho$. For technical reasons, we need to arrange such interpolation by convex hypersurfaces (Lemma \ref{lemma:convex interpolation}). An extension argument then provides the time-depending Hamiltonian function whose time-one flow is the desired symplectomorphism $\Phi\in\Ham_c(\mathcal G_\rho,\Omega)$ such that $\Phi(\mathcal L)=\mathcal L'$.

\subsection{Relation with geometric flows}

Finally, in Appendix \ref{app:geometric flows} we apply these methods to study the relation between evolutions by geometric flows in $\Hyp^{n+1}$ and in $\G{n+1}$. More precisely, suppose that $\sigma_\bullet:M\times(-\epsilon,\epsilon)\to\Hyp^{n+1}$ is a smoothly varying family of Riemannian immersions that satisfy:
$$\frac{d}{dt} \sigma_t  = f_t \nu_t$$
where $\nu_t$ is the normal vector of $\sigma_t$ and $f_\bullet:M\times(-\epsilon,\epsilon)\to\R$ is a smooth function. Then the variation of the Gauss map $G_t$ of $\sigma_t$ is given, up to a tangential term, by the normal term  $-\JJ(dG_t({\overline\nabla}{}^t f_t))$, where ${\overline \nabla}{}^t f_t$ denotes the gradient with respect to the first fundamental form of $G_t$, that is, the Riemannian metric $G_t^*\GG$. 

Let {us} consider the special case of the function
$f_t:=f_{\sigma_t}$, as defined in Equation \eqref{eq fsigma intro}, namely the sum of hyperbolic inverse tangent of the principal curvatures. The study of the associated flow has been suggested in dimension three in \cite{zbMATH01789966}, by analogy of a similar flow on surfaces in the three-sphere. 
Combining the aforementioned result of Appendix \ref{app:geometric flows} with
Equation \eqref{eq maslov intro}, we obtain that such flow in $\Hyp^{n+1}$ induces the Lagrangian mean curvature flow in $\G{n+1}$ up to tangential diffeomorphisms. A similar result has been obtained in Anti-de Sitter space (in dimension three) in \cite{zbMATH06182655}.

\subsection*{Organization of the paper}

The paper is organized as follows. In Section \ref{sec space geodesics} we introduce the space of geodesics $\G{n+1}$ and its natural para-K\"ahler structure. In Section \ref{sec:gauss map} we study the properties of the Gauss map and provide useful examples. Section \ref{sec:small princ curv} focuses on immersions with small principal curvatures and prove several properties. In Section \ref{sec:local int}  we study the relations with the geometry of flat principal bundles, in particular Equation \eqref{eq omega Omega intro} (Proposition \ref{prop:identity curvature form}, that relies the symplectic form in the space of geodesics to the curvature of principal bundles), and we prove the statements concerning local integrability and global integrability of Riemannian immersions in $\G{n+1}$, including Theorem \ref{cor: local integrability}, Theorem \ref{prop: riemannian global integrability} and Theorem \ref{cor G complete}. In Section \ref{sec:equivariant integrability} we discuss the problem of equivariant integrability in relation with the Maslov class of the immersion: we prove Theorem \ref{teorema hol H baby} (more precisely, the stronger version given in Theorem  \ref{Teorema hol H}) and deduce Corollaries \ref{cor hol H baby} and \ref{cor:uniqueness min lag}. Finally, in Section \ref{sec:hamiltonian}, focusing on nearly-Fuchsian representations, we prove Theorem \ref{thm:second char ham}.

\subsection*{Acknowledgements}
We are grateful to Francesco Bonsante for many related discussion and for his interest and encouragement in this work. Part of this work has been done during visits of the first author to the University of Luxembourg in November 2018 and to Grenoble Alpes University in April-May 2019; we are thankful to these Institutions for the hospitality. We would like to thank an anonymous referee for several useful comments.

\changelocaltocdepth{2}

\section{The space of geodesics of hyperbolic space}\label{sec space geodesics}

In this section, we introduce the hyperbolic space $\Hyp^n$ and its space of (oriented {maximal unparametrized}) geodesics, which will be endowed with a natural para-K\"ahler structure by means of a construction on the unit tangent bundle $T^1\Hyp^n$.

\subsection{Hyperboloid model}

In this paper we will mostly use the hyperboloid model of $\Hyp^n$. Let us denote by $\R^{n,1}$ the $(n+1)$-dimensional Minkowski space, namely the vector space $\R^{n+1}$ endowed with the standard bilinear form of signature $(n,1)$:
$$\langle x,y\rangle=x_1y_1+\ldots+x_ny_n-x_{n+1}y_{n+1}~.$$
The \emph{hyperboloid model} of hyperbolic space is
$$\Hyp^n=\{x\in\R^{n,1}\,|\,\langle x,x\rangle =-1,\,x_{n+1}>0\}~.$$
Then the tangent space at a point $x$ is identified to its orthogonal subspace:
$$T_x\Hyp^n\cong x^\perp=\{v\in\R^{n,1}\,|\,\langle x,v\rangle =0\}~.$$
The \emph{unit tangent bundle} of $\Hyp^n$ is the bundle of unit tangent vectors, and therefore has the following model:
\begin{equation}\label{eq:modelT}
T^1\Hyp^n=\{(x,v)\in \R^{n,1}\times \R^{n,1}\,|\,x\in\Hyp^n,\,\langle x,v\rangle=0,\,\langle v,v\rangle=1\}~,
\end{equation}
where the obvious projection map is simply 
$$\pi:T^1\Hyp^n\to\Hyp^n\qquad \pi(x,v)=x~.$$
In this model, we can give a useful description of the tangent space of $T^1\Hyp^n$ at a point $(x,v)$, namely:
\begin{equation}\label{eq:modelTT}
T_{(x,v)}T^1\Hyp^n=\{(\dot x,\dot v)\in\R^{n,1}\times \R^{n,1}\,|\,\langle x,\dot x\rangle=\langle v,\dot v\rangle=\langle x,\dot v\rangle+\langle v,\dot x\rangle=0\}~,
\end{equation}
where the relations $\langle x,\dot x\rangle=0$ and $\langle v,\dot v\rangle=0$ arise by differentiating $\langle x,x\rangle=-1$ and $\langle v,v\rangle=1$, while the relation $\langle x,\dot v\rangle+\langle v,\dot x\rangle=0$ is the linearized version of $\langle x,v\rangle=0$.

Finally, let us denote by $\G{n}$ the \emph{space of (maximal, oriented, unparameterized) geodesics} of $\Hyp^n$, namely, the space of non-constant geodesics $\gamma:\R\to\Hyp^n$ up to monotone increasing reparameterizations. Recalling that an oriented geodesic is uniquely determined by its two (different) endpoints in the visual boundary $\partial\Hyp^n$, we have the following identification of this space:
$$\G{n}\cong \partial\Hyp^n\times\partial\Hyp^n\setminus \Delta~,$$
where $\Delta$ represents the diagonal. We recall that, in the hyperboloid model, $\partial\Hyp^n$ can be identified to the projectivization of the null-cone in Minkowski space:
\begin{equation}\label{eq:bdy hyperboloid model}
\partial\Hyp^n=\faktor{\{x\in\R^{n,1}\,|\,\langle x,x\rangle =0,\,x_{n+1}>0\}}{\R_{>0}}~.
\end{equation}
It will be of fundamental importance in the following to endow $T^1\Hyp^n$ with \emph{another} bundle structure, a \emph{principal} bundle structure, over $\G{n}$. 
For this purpose, recall that the \emph{geodesic flow} is the $\R$-action over $T^1\Hyp^n$ given by:
$$t\cdot (x,v)=\varphi_t(x,v)= (\gamma(t),\gamma'(t))~,$$
where $\gamma$ is the unique parameterized geodesic such that $\gamma(0)=x$ and $\gamma'(0)=v$. {In the hyperboloid model, the flow $\varphi_t\colon T^1\Hyp^n\to T^1\Hyp^n$ can be written explicitly as}
\begin{equation}\label{eq:geodflow}
\varphi_t(x,v)=(\cosh(t)x+\sinh(t)v,\sinh(t)x+\cosh(t)v)~.
\end{equation}

Then $T^1\Hyp^n$ is naturally endowed with a principal $\R$-bundle structure:
$$\mathrm p:T^1\Hyp^n\to\G n\qquad \mathrm p(x,v)=\gamma~,$$
for $\gamma$ the geodesic defined as above, that is, $\mathrm p(x,v)$ is the element of $\G{n}$ going through $x$ with speed $v$.

Finally, recall that the group of orientation-preserving isometries of  $\Hyp^n$, which we will denote by $\Isom(\Hyp^n)$, is identified to $\SO_0(n,1)$, namely the connected component of the identity in the group of linear isometries of $\R^{n,1}$. Clearly $\Isom(\Hyp^n)$ acts both on $T^1\Hyp^n$ and on $\G n$, in the obvious way, and moreover the two projection maps $\pi:T^1\Hyp^n\to\Hyp^n$ and $\mathrm p:T^1\Hyp^n\to\G n$ are equivariant with respect to these actions. In the next sections, we will introduce some additional structures on $T^1\Hyp^n$ and $\G n$ that are \emph{natural} in the sense that they are preserved by the action of $\Isom(\Hyp^n)$.

\subsection{Para-Sasaki metric on the unit tangent bundle}

We shall now introduce a pseudo-Riemannian metric on $T^1\Hyp^n$. For this purpose, let us first recall the construction of the {horizontal} and {vertical} lifts and distributions in the unit tangent bundle of a Riemannian manifold, which for simplicity we only recall for $\Hyp^n$.  Given $(x,v)\in T\Hyp^n$,  the \emph{vertical subspace} at $(x,v)$ is defined as:
$$\VP_{(x,v)}=T_{(x,v)} \big(\pi^{-1}(x)\big)\cong v^\perp\subset T_x\Hyp^n$$ 
since $\pi^{-1}(x)$ is naturally identified to the sphere of unit vectors in the vector space $T_x\Hyp^n$.\footnote{We use here $\VP$ to distinguish with the vertical subspace in the full tangent bundle $T\Hyp^n$, which is usually denoted by $\V$.} 
Hence given a vector $w\in T_x\Hyp^n$ orthogonal to $v$, we can define its \emph{vertical lift}
$w^\V\in \VP_{(x,v)}$, and 
vertical lifting gives a map from $v^\perp$ to $\VP_{(x,v)}\subset T_{(x,v)}T^1\Hyp^n$ which is simply the identity map under the above identification. More concretely, in the model for $T_{(x,v)}T^1\Hyp^n$ introduced in \eqref{eq:modelTT}, we have
$$w^\V=(0,w)\in \R^{n,1}\times \R^{n,1}~.$$

Let us move to the horizontal lift. This is defined as follows. Given $u\in T_x\Hyp^n$, let us consider the parameterized geodesic $\gamma:\R\to\Hyp^n$ with $\gamma(0)=x$ and $\gamma'(0)=u$, and let $v(t)$ be the parallel transport of $v$ along $\gamma$. Then $u^\HH$ is defined as the derivative of {$(\gamma(t),v(t))$} at time $t=0$. This gives an injective linear map from $T_x\Hyp^n$ to $T_{(x,v)}T^1\Hyp^n$, whose image is the \emph{horizontal subspace} $\HH_{(x,v)}$. Let us compute this map {in the hyperboloid model} by distinguishing two different cases.

First, let us consider the case of $u=w\in v^\perp\subset T_x\Hyp^n$. In the model \eqref{eq:modelTT}, using that the image of the parameterized geodesic $\gamma$ is the intersection of $\Hyp^n$ with a plane in $\R^{n,1}$ {orthogonal to $v$}, the parallel transport of $v$ along $\gamma$ is the  vector field constantly equal to $v$, and therefore
$$w^\HH=\ddt(\gamma(t),v)=(w,0)\in \R^{n,1}\times \R^{n,1}~.$$ 
We shall denote by $\HP_{(x,v)}$ the subspace of horizontal lifts of this form, which is therefore a horizontal subspace in $T_{(x,v)}T^1\Hyp^n$ isomorphic to $v^\perp$.

There remains to understand the case of $u=v$. 

\begin{lemma} \label{lemma:generator geoflow hor lift}
	Given $(x,v)\in T^1\Hyp^n$, the horizontal lift $v^\HH$ coincides with the infinitesimal generator $\chi_{(x,v)}$ of the geodesic flow, and has the expression:
	$$\chi_{(x,v)}=(v,x)\in\R^{n,1}\times\R^{n,1}~.$$
\end{lemma}
\begin{proof}
	Since the tangent vector to a parameterized geodesic is parallel along the geodesic itself, $\varphi_t(x,v)$ also equals {$(\gamma(t),v(t))$, for $v(t)$ the vector field} used to define the horizontal lift. 
	Hence clearly 
	$$v^\HH=\chi_{(x,v)}=\ddt \varphi_t(x,v)~.$$ 
	Differentiating Equation \eqref{eq:geodflow} at $t=0$ we obtain the desired expression.
\end{proof}

In conclusion, we have the direct sum decomposition:
\begin{equation}\label{eq:direct sum}
T_{(x,v)}T^1\Hyp^n=\HH_{(x,v)}\oplus\VP_{(x,v)}=\mathrm{Span}(\chi_{(x,v)})\oplus\HP_{(x,v)}\oplus\VP_{(x,v)}~.
\end{equation}
{We are now able to} introduce the para-Sasaki metric on the unit tangent bundle.

\begin{defi}\label{defi:parasasaki}
	The \emph{para-Sasaki metric} on $T^1\Hyp^n$ is the pseudo-Riemannian metric $\gs n$ defined by
	$$\gs{n}(X_1,X_2)=\begin{cases} 
	+\langle u_1,u_2\rangle & \text{if }X_1,X_2\in\HH_{(x,v)}\text{ and }X_i=u_i^\HH \\
	-\langle w_1,w_2\rangle & \text{if }X_1,X_2\in\VP_{(x,v)}\text{ and }X_i=w_i^\V \\
	0 & \text{if }X_1\in\HH_{(x,v)}\text{ and }X_2\in\VP_{(x,v)}~.
	\end{cases}$$
\end{defi}

The metric  $\gs n$ is immediately seen to be non-degenerate of signature $(n,n-1)$. It is also worth observing that, from Definition \ref{defi:parasasaki} and Lemma \ref{lemma:generator geoflow hor lift},
\begin{equation}\label{eq:generator geoflow is unit}
\gs{n}(\chi_{(x,v)},\chi_{(x,v)})=1~,
\end{equation}
and that $\chi_{(x,v)}$ is orthogonal to both $\VP_{(x,v)}$ and $\HP_{(x,v)}$.

The para-Sasaki metric, together with the decomposition \eqref{eq:direct sum} will be essential in our definition of the para-K\"ahler metric on $\G{n}$ and in several other constructions. 

Before that, we need an additional  observation. Clearly the obvious action of the isometry group $\Isom(\Hyp^n)$ on $T^1\Hyp^n$ preserves the para-Sasaki metric, since all ingredients involved in the definition are invariant by isometries. The same is also true for the action of the geodesic flow, and this fact is much more peculiar of the choice we made in Definition \ref{defi:parasasaki}.

\begin{lemma}\label{lemma:geodflow isometric}
	The $\R$-action of the geodesic flow on $T^1\Hyp^n$ is isometric for the para-Sasaki metric, and commutes with the action of $\Isom(\Hyp^n)$.
\end{lemma}
\begin{proof}
	Let us first consider the differential of $\varphi_t$, for a given $t\in\R$. Since the expression for $\varphi_t$ from Equation \eqref{eq:geodflow} is linear in $x$ and $v$, we have: 
	\begin{equation}\label{eq:geodflow linearized}
	d\varphi_t(\dot x,\dot v)=(\cosh(t)\dot x+\sinh(t)\dot v,\sinh(t)\dot x+\cosh(t)\dot v)~,
	\end{equation}
	for $X=(\dot x,\dot v)$ as in \eqref{eq:modelTT}.
	Let us distinguish three cases. 
	
	If $X=w^\HH=(w,0)$ for $w\in v^\perp\subset T_x\Hyp^n$, then
	\begin{equation}\label{eq:diff geoflow1}
	d\varphi_t(w^\HH)=(\cosh(t)w,\sinh(t)w)=\cosh(t)w^\HH+\sinh(t)w^\V~.
	\end{equation}
	For $X=w^\V=(0,w)$ a completely analogous computation gives
	\begin{equation}\label{eq:diff geoflow2}
	d\varphi_t(w^\V)=(\sinh(t)w,\cosh(t)w)=\sinh(t)w^\HH+\cosh(t)w^\V~.
	\end{equation}
	Finally, for $X=\chi_{(x,v)}$, by constriction 
	\begin{equation}\label{eq:diff geoflow3}
	d\varphi_t(\chi_{(x,v)})=\chi_{\varphi_t(x,v)}~.
	\end{equation}
	Now using \eqref{eq:diff geoflow1} and \eqref{eq:diff geoflow2}, and Definition \ref{defi:parasasaki}, we can check that that 
	$$\gs{n}(d\varphi_t(w_1^\HH),d\varphi_t(w_2^\HH))=(\cosh^2(t)-\sinh^2(t))\langle w_1,w_2\rangle=\langle w_1,w_2\rangle=\gs{n}(w_1^\HH,w_2^\HH)~.$$
	A completely analogous computation shows  that
	$$\gs{n}(d\varphi_t(w_1^\V),d\varphi_t(w_2^\V))=-\langle w_1,w_2\rangle=\gs{n}(w_1^\V,w_2^\V)~$$ and that
	$$\gs{n}(d\varphi_t(w_1^\HH),d\varphi_t(w_2^\V))=0=\gs{n}(w_1^\HH,w_2^\V)~.$$
	By  \eqref{eq:generator geoflow is unit} and \eqref{eq:diff geoflow3}, the norm of vectors proportional to $\chi_{(x,v)}$ is preserved. Together with \eqref{eq:diff geoflow1} and \eqref{eq:diff geoflow2}, vectors of the form $d\varphi_t(w^\HH)$ and $d\varphi_t(w^\V)$ are orthogonal to $d\varphi_t(\chi_{(x,v)})=\chi_{\varphi_t(x,v)}$.
	This concludes the first part of the statement.
	
	Finally, since isometries map parameterized geodesics to parameterized geodesics, it is straightworward to see that the $\R$-action commutes with $\Isom(\Hyp^n)$.
\end{proof}

\subsection{A para-K\"ahler metric on the space of geodesics}\label{sec:parakahler metric GG}
Let us start by quickly recalling the basic definitions of para-complex and para-K\"ahler geometry. First introduced by Libermann in \cite{libermann}, the reader can refer to the survey \cite{zbMATH00894673} for more details on para-complex geometry. 

Given a manifold $\mathcal M$ of dimension $2n$,  an \emph{almost para-complex structure} on $\mathcal M$ is a tensor $\JJ$ of type $(1,1)$ (that is, a smooth section of the bundle of endomorphisms of $T\mathcal M$) such that $\JJ^2=\mathbbm 1$ and that at every point $p\in \mathcal M$ the eigenspaces {$T_p^\pm \mathcal M=\ker(\JJ\mp \mathbbm 1)$} have dimension $n$. The almost para-complex structure $\JJ$ is a \emph{para-complex structure} if the distributions $T_p^\pm \mathcal M$ are integrable. 

A \emph{para-K\"ahler structure} on $\mathcal M$ is the datum of a para-complex structure $\JJ$ and a pseudo-Riemannian metric $\GG$ such that $\JJ$ is $\GG$-skewsymmetric, namely
\begin{equation}\label{eq:JGcompatible}
\GG(\JJ X,Y)=-\GG(X,\JJ Y)
\end{equation}
for every $X$ and $Y$, and such that the \emph{fundamental form}, namely the 2-form
\begin{equation}
\label{eq: Definizione Omega}
\Omega(X,Y):=\GG(X,\JJ Y),
\end{equation} 
is closed.

Observe that Equation \eqref{eq:JGcompatible} is equivalent to the condition that $\JJ$ is anti-isometric for $\GG$, namely:
\begin{equation}\label{eq:JGantiisometry}
\GG(\JJ X,\JJ Y)=-\GG(X,Y)
\end{equation}
which implies immediately that the metric of $\GG$ is necessarily neutral (that is, its signature is $(n,n)$).

\vspace{10pt}

Let us start to introduce the  para-K\"ahler structure on the space of geodesics $\G{n+1}$, whose dimension is $2n$. Recalling the $\R$-principal bundle structure $\mathrm p:T^1\Hyp^{n+1}\to \G{n+1}$, we will introduce the defining objects on $T^1\Hyp^{n+1}$, and show that they can be pushed forward to $\G{n+1}$. More precisely, given a point $(x,v)\in T^1\Hyp^{n+1}$, the decomposition \eqref{eq:direct sum} 
shows that the tangent space $T_\ell \G{n+1}$ identifies to $\chi_{(x,v)}^\perp=\HP_{(x,v)}\oplus\VP_{(x,v)}$, where $\ell\in \G{n+1}$ is the oriented unparameterized geodesic going through $x$ with speed $v$, and the orthogonal subspace is taken with respect to the para-Sasaki metric $\gs n$.
Indeed, the kernel of the projection $\mathrm p$ equals the subspace generated by $\chi_{(x,v)}$, and therefore the differential of $\mathrm p$ induces a vector space isomorphism 
\begin{equation}                                                                  \label{eq dp ristretto e' iso}
d\mathrm p|_{\chi^\perp_{(x,v)}}\colon \chi_{(x,v)}^\perp \xrightarrow{\sim} T_\ell \G{n+1}~.
\end{equation}

Now, let us define $J\in\mathrm{End}(\chi_{(x,v)}^\perp)$ by the following expression:
$$J(\dot x,\dot v)=(\dot v,\dot x)~.$$
In other words, recalling that $\HP_{(x,v)}$ consists of the vectors of the form $(w,0)$, and $\VP_{(x,v)}$ of those of the form $(0,w)$, for $w\in v^\perp$, $J$ is defined by
\begin{equation}\label{eq:defiJ2}
J(w^\HH)=w^\V\qquad \text{and}\qquad J(w^\V)=w^\HH~.
\end{equation}

\begin{lemma}\label{lemma:paracomplex structure}
The endomorphism $J$ induces an almost para-complex structure $\JJ$ on $T_\ell \G{n+1}$, which does not depend on the choice of $(x,v)\in \mathrm p^{-1}(\ell)$.
\end{lemma}
\begin{proof}
By definition of the $\R$-principal bundle structure $\mathrm p:T^1\Hyp^{n+1}\to \G{n+1}$ and of the geodesic flow $\varphi_t$, $\mathrm p\circ \varphi_t=\mathrm p$ for every $t\in\R$. Moreover $\varphi_t$ preserves the infinitesimal generator $\chi$ (Equation \eqref{eq:diff geoflow3}) and acts isometrically on $T^1\Hyp^{n+1}$ by Lemma \ref{lemma:geodflow isometric}, hence it preserves the orthogonal complement of $\chi$.
Therefore, for all given vectors $X,Y\in T_\ell \G{n+1}$, any two lifts of $X$ and $Y$ on $T^1\Hyp^{n+1}$ orthogonal to $\mathrm p^{-1}(\ell)$ differ by push-forward by $\varphi_t$.

However, it is important to stress that the differential of $\varphi_t$ does \emph{not} preserve the distributions $\HP$ and $\VP$ {individually}  (see Equations \eqref{eq:diff geoflow2} and \eqref{eq:diff geoflow3}).
Nevertheless, by Equation \eqref{eq:geodflow linearized}, {we get that}
\begin{align*}(\varphi_t)_*(J(\dot x,\dot v))=(\varphi_t)_*(\dot v,\dot x)&=(\cosh(t)\dot v+\sinh(t)\dot x,\sinh(t)\dot v+\cosh(t)\dot x) \\
&=J(\sinh(t)\dot v+\cosh(t)\dot x,\cosh(t)\dot v+\sinh(t)\dot x)=J(\varphi_t)_*(\dot x,\dot v)~,
\end{align*}
which shows that {the geodesic flow preserves $J$, and therefore that $J$ induces a well-defined tensor} $\JJ$ on $T_\ell \G{n+1}$. It is clear from the expression of $J$ that $\JJ^2=\mathbbm 1$, and moreover that the $\pm 1$-eigenspaces of $\JJ$ both have dimension $n$, since the eigenspaces of $J$ consist precisely of the vectors of the form $(w,w)$ (resp. $(w,-w)$) for $w\in v^\perp\subset T_x\Hyp^{n+1}$.
\end{proof}

Let us now turn our attention to the construction of the neutral metric $\GG$, which will be defined by a similar construction. In fact, given $(x,v)\in \mathrm p^{-1}(\ell)$, we simply define $\GG$ on $T_\ell \G{n+1}$ as the push-forward of the restriction $\gs{n+1}|_{\chi^\perp_{(x,v)}}$ by the isomorphism in Equation \eqref{eq dp ristretto e' iso}. 

Well-posedness of this definition follows immediately from Equation \eqref{eq:diff geoflow3} and Lemma \ref{lemma:geodflow isometric}.

\begin{lemma}
The restriction of $\gs{n+1}$ to $\chi^\perp_{(x,v)}$ induces a neutral metric $\GG$ on $T_\ell \G{n+1}$, which does not depend on the choice of $(x,v)\in \mathrm p^{-1}(\ell)$, such that $\JJ$ is $\GG$-skewsymmetric.
\end{lemma}
\begin{proof}
It only remains to show the $\GG$-skewsymmetry, namely Equation \eqref{eq:JGcompatible}. The latter is indeed equivalent to Equation \eqref{eq:JGantiisometry}, which {simply follows from observing that,
as a consequence} of Definition \ref{defi:parasasaki} and of the definition of $J$ in \eqref{eq:defiJ2}, {one has  $$\gs{n+1}(JX,JY)=-\gs{n+1}(X,Y)$$ for all $X,Y\in \chi^\perp$}
\end{proof}

There is something left to prove in order to conclude that the constructions of $\JJ$ and $\GG$ induce a para-K\"ahler structure on $\G{n+1}$, but we defer the remaining checks to the following sections: in particular, we are left to prove that the almost para-complex structure $\JJ$ is integrable (it will be a consequence of Example \ref{ex:horospheres}) and that the  2-form $\Omega=\GG(\cdot,\JJ\cdot)$ is closed (which is the content of  Corollary \ref{cor:omega closed}).

{
	\begin{remark}
		\label{rmk: Isom preserva Omega e J}
		The group $\Isom(\Hyp^{n})$ acts naturally on $\G{n}$ and the map $\mathrm p\colon T^1 \Hyp^n\to \G{n}$ is equivariant, namely $\mathrm p (\psi\cdot (x,v))= \psi \cdot \mathrm p(x,v)$ for all $\psi\in \Isom(\Hyp^n)$. As a result, by construction of $\GG$ and $\JJ$, the action of $\Isom(\Hyp^n)$ on $\G{n}$ preserves  $\GG$, $\JJ$ and $\Omega$.
\end{remark}}

\begin{remark}\label{rmk other metric1}
Of course some choices have been made in the above construction, in particular in the expression of the para-Sasaki metric of Definition \ref{defi:parasasaki}, which has a fundamental role when introducing the metric $\GG$. The essential properties we used are the naturality with respect to the isometry group of $\Hyp^{n+1}$ and to the action of the geodesic flow (Lemma \ref{lemma:geodflow isometric}).

Some alternative definitions for $\gs{n+1}$ would produce the same expression for $\GG$.
For instance one can define for all $c\in \R^+$ a metric $g_c$ on $T^1 \Hyp^{n+1}$ so that, with respect to the direct sum decomposition \eqref{eq:direct sum}:
\begin{itemize}
	\item $g_c(w_1^\HH,w_2^\HH)=-g_c(w_1^\V,w_2^\V)=\langle w_1,w_2\rangle$ for any $w_1,w_2\in v^\perp\subset T_x\Hyp^{n+1}$,
	\item $g_c(\chi_{(x,v)},\chi_{(x,v)})=c$,
	\item $\mathrm{Span}(\chi_{(x,v)})$, $\HP_{(x,v)}$ and $\VP_{(x,v)}$ are mutually $g_c$-orthogonal.
\end{itemize}

Replacing $\gs{n+1}$ with such a $g_c$, one would clearly obtain the same metric $\GG$ since it only depends on the restriction of $g_c$ to the orthogonal complement of $\chi$.
{Moreover, $g_c$ is invariant under the action of $\Isom(\Hyp^n)$ and under the geodesic flow.}

\end{remark}
\begin{remark}\label{rmk other metric2}
It will  be convenient to use Remark \ref{rmk other metric1} in the following, by considering $T^1\Hyp^{n}$ as a submanifold of $\R^{n,1}\times \R^{n,1}$, and taking the metric given by the Minkowski product on the first factor, and its opposite on the second factor, restricted to $T^1\Hyp^{n}$, i.e.
\begin{equation}\label{eq:metric minkxmink}
\widehat g_{T^1 \Hyp^n}((\dot x_1,\dot v_1),(\dot x_2,\dot v_2))=\langle \dot x_1,\dot x_2\rangle-\langle \dot v_1,\dot v_2\rangle~.
\end{equation}
In fact, it is immediate to check that $\widehat g_{T^1 \Hyp^n}(w_1^\HH,w_2^\HH)=\widehat g_{T^1 \Hyp^n}((w_1,0),(w_2,0))=\langle w_1,w_2\rangle$ for $w_i\in v^\perp$, that similarly $\widehat g_{T^1 \Hyp^n}(w_1^\V,w_2^\V)=-\langle w_1,w_2\rangle$, and that 
$$\widehat g_{T^1 \Hyp^n}(\chi_{(x,v)},\chi_{(x,v)})=\widehat g_{T^1 \Hyp^n}((v,x),(v,x))=\langle v,v\rangle-\langle x,x\rangle =2~.$$
Finally elements of the three types are mutually orthogonal, and therefore
$\widehat g_{T^1 \Hyp^n}=g_2$ with $g_2$ as in Remark \ref{rmk other metric1}.
\end{remark}

\section{The Gauss map of hypersurfaces in $\Hyp^{n+1}$}\label{sec:gauss map}
In this section we will focus on the construction of the Gauss map of an immersed hypersurface, its relation with the normal evolution and the geodesic flow action on the unit tangent bundle, and provide several examples of great importance for the rest of this work.

\subsection{Lift to the unit tangent bundle}
Let us introduce the notions of lift to the unit tangent bundle and Gauss map for an immersed hypersurface in hyperbolic space, and start discussing some properties.

\begin{defi}\label{defi:lift and gauss}
Let $M$ be an oriented $n$-dimensional manifold, let $\sigma:M\to\Hyp^{n+1}$ be an immersion, and let $\nu$ be the unit normal vector field of $\sigma$ compatible with the orientations of $M$ and $\Hyp^{n+1}$. Then we define the \emph{lift} of $\sigma$ as 
$$\zeta_\sigma:M\to T^1\Hyp^{n+1}\qquad \zeta_\sigma(p)=(\sigma(p),\nu(p))~.$$
The \emph{Gauss map} of $\sigma$ is then the map
$$G_\sigma:M\to\G{n+1}\qquad G_\sigma=\mathrm{p}\circ \zeta_\sigma~.$$
\end{defi}

In other words, the Gauss map of $\sigma$ is the map which associates to $p\in M$ the geodesic $\ell$ of $\Hyp^{n+1}$ orthogonal to the image of $d_p\sigma$ at $\sigma(p)$, oriented compatibly with the orientations of $M$ and $\Hyp^{n+1}$.

{Also {recall that} the \emph{shape operator} $B$ of $\sigma$ {is defined} as the $(1,1)$-tensor on $M$ defined by 
\begin{equation}\label{eq:shape}
d\sigma\circ B(W)=-D_{W} \nu~,
\end{equation} for $D$ the Levi-Civita connection of $\Hyp^{n+1}$ and $W\in TM$.}

\begin{prop}\label{prop:gauss immersion}
Given an oriented manifold $M^n$ and an immersion $\sigma:M\to\Hyp^{n+1}$, the lift of $\sigma$ is an immersion orthogonal to the fibers of $\mathrm{p}:T^1\Hyp^{n+1}\to \G{n+1}$. As a consequence $G_\sigma$ is an immersion.
\end{prop}
\begin{proof}
By a direct computation in the hyperboloid model, the differential of $\zeta_\sigma$ has the expression 
\begin{equation}
\label{eq: differential of sigma tilde}
d_p\zeta_\sigma (W)=(d_p\sigma(W),d_p\nu(W))=(d_p\sigma(W),-d_p\sigma(B(W)))~.
\end{equation} Indeed, the ambient derivative  in $\R^{n+1,1}$ of the vector field $\nu$ equals the covariant derivative with respect to $D$, since {$d_p\nu(W)$} is orthogonal to $\sigma(p)$ as a consequence of the condition $\langle \sigma(p),\nu(p)\rangle=0$. 
 
As both $d_p\sigma(W)$ and $d_p\sigma(B(W))$ are tangential to the image of $\sigma$ at $p$, hence orthogonal to $\nu(p)$, $d_p\zeta_\sigma (W)$ can be written as:
\begin{equation}\label{eq:differential lift}
d_p\zeta_\sigma(W)=d_p\sigma(W)^\HH-d_p\sigma(B(W))^\V~.
\end{equation}
Therefore, for every $W\neq 0$, $d_p \zeta_\sigma(W)$ is a non-zero vector orthogonal to $\chi_{\zeta_\sigma(p)}$ by Definition \ref{defi:parasasaki}. Since the differential of $\mathrm p$ is a vector space isomorphism between $\chi_{\zeta_\sigma(p)}^\perp$ and $T_{G_\sigma(p)}\G{n+1}$, the Gauss map $G_\sigma$ is also an immersion. 
\end{proof}

As a consequence of Proposition \ref{prop:gauss immersion}, we can compute the first fundamental form of the Gauss map $G_\sigma$, that is, the pull-back metric $G_\sigma^*\GG$, which we denote by $\overline\I$. Since the lift $\zeta_\sigma$ is orthogonal to $\chi$, it suffices to compute the pull-back metric of $\gs{n+1}$ by $\zeta_\sigma$. By Equation \eqref{eq:differential lift}, we obtain:
\begin{equation}\label{eq:fff gauss}
\overline\I=\I - \III~,
\end{equation}
where $\I=\sigma^*g_{\Hyp^{n+1}}$ is the first fundamental form of $\sigma$, and $\III=\I(B\cdot,B\cdot)$ its third fundamental form in $\Hyp^{n+1}$.

Let us now see that the orthogonality to the generator of the geodesic flow essentially characterizes the lifts of immersed hypersurfaces in $\Hyp^{n+1}$, in the following sense.

\begin{prop}\label{prop:gauss immersion converse}
Let $M^n$ be an orientable manifold and $\zeta:M\to T^1\Hyp^{n+1}$ be an immersion orthogonal to the fibers of $\mathrm{p}:T^1\Hyp^{n+1}\to \G{n+1}$. If $\sigma:=\pi\circ\zeta$ is an immersion, then $\zeta$ is the lift of $\sigma$ with respect to an orientation of $M$.
\end{prop}
\begin{proof}
Let us write $\zeta=(\sigma,\nu)$. Choosing the orientation of $M$ suitably, we only need to show that the unit vector field $\nu$ is normal  to the immersion $\sigma$. By differentiating, $d\zeta=(d\sigma,\mathrm d\nu)$ and by \eqref{eq:modelTT} we obtain 
$$\langle \nu(p),d\sigma(W)\rangle+\langle \sigma(p),\mathrm d\nu(W)\rangle=0$$ 
 for all  $W\in T_p M$. By the orthogonality hypothesis and the expression $\chi_{\zeta(p)}=(\nu(p),\sigma(p))$ (Lemma \ref{lemma:generator geoflow hor lift}) we obtain
 $$\langle \nu(p),d\sigma(W)\rangle-\langle \sigma(p),d\nu(W)\rangle=0~,$$
 hence $\langle \nu(p),d\sigma(W)\rangle=0$ for all $W$. Since by hypothesis the differential of $\sigma$ is injective, $\nu(p)$ is uniquely determined up to the sign, and is a unit normal vector to the immersion $\sigma$. 
 \end{proof}

In relation with Proposition \ref{prop:gauss immersion converse}, it is important to remark that there are (plenty of) immersions in $T^1\Hyp^{n+1}$ which are orthogonal to $\chi$ but are \emph{not} the lifts of immersions in $\Hyp^{n+1}$, meaning that they become singular when post-composed with the projection $\pi:T^1\Hyp^{n+1}\to \Hyp^{n+1}$. 
{Some examples of this phenomenon (Example \ref{ex:spheres}), and more in general of the Gauss map construction, are presented   in Section \ref{subsec:examples} below.}

\subsection{Geodesic flow and normal evolution} We develop here the construction of normal evolution, starting from an immersed hypersurface in $\Hyp^{n+1}$.

\begin{defi}\label{defi normal evolution}
Given an oriented manifold $M^n$ and an immersion $\sigma:M\to\Hyp^{n+1}$, the \emph{normal evolution} of $\sigma$ is the map $\sigma_t:M\to\Hyp^{n+1}$ defined by
$$\sigma_t(p)=\exp_{\sigma(p)}(t\nu(p))~,$$
where $\nu$ is the unit normal vector field of $\sigma$ compatible with the orientations of $M$ and $\Hyp^{n+1}$.
\end{defi}

The relation between the normal evolution and the geodesic flow is encoded in the following proposition. 

\begin{prop}
\label{prop: flusso geod e flusso normale}
Let $M^n$ be an orientable manifold  and $\sigma:M\to \Hyp^{n+1}$ be an immersion. Suppose $\sigma_t$ is an immersion for some $t\in\R$. Then $\zeta_{\sigma_t}=\varphi_t\circ \zeta_\sigma$. 
\end{prop}
\begin{proof}
The claim is equivalent to showing that, if the differential of $\sigma_t$ is injective at $p$, then $(\sigma_t(p),\nu_t(p))=\varphi_t(\sigma(p),\nu(p))$, where $\nu_t$ is the normal vector of $\sigma_t$. The equality on the first coordinate holds by definition of the geodesic flow, since $t\mapsto \gamma(t)=\sigma_t(p)$ is precisely the parameterized geodesic such that $\gamma(0)=\sigma(p)$ with speed $\gamma'(0)=\nu(p)$. It thus remains to check that the normal vector of $\sigma_t(p)$ equals $\gamma'(t)$.

 By the usual expression of the exponential map in the hyperboloid model of $\Hyp^{n+1}$, the normal evolution is:
\begin{equation}\label{eq:normal evolution hyperboloid}
\sigma_t(p)=\cosh(t)\sigma(p)+\sinh(t)\nu(p)~,
\end{equation}
and therefore 
\begin{equation}\label{eq:normal evolution hyperboloid2}
d\sigma_t(V)=d\sigma\circ(\cosh(t)\mathrm{id}-\sinh(t)B)(V)
\end{equation}
for $V\in T_p M$.
It is then immediate to check that
$$\gamma'(t)=\sinh(t)\sigma(p)+\cosh(t)\nu(p)$$
is orthogonal to $d\sigma_t(V)$
for all $V$.
If $d\sigma_t$ is injective, this implies that $\gamma'(t)$ is the unique unit vector normal to the image of $\sigma$ and compatible with the orientations, hence it equals $\nu_t(p)$. This concludes the proof.
\end{proof}

A straightforward consequence, recalling that the Gauss map is defined as $G_\sigma=\mathrm{p}\circ\zeta_\sigma$ and that $\mathrm p\circ\varphi_t=\mathrm p$, is the following: 

\begin{cor}\label{prop:gauss map invariant normal evo}
The Gauss map of an immersion $\sigma:M^n\to \Hyp^{n+1}$ is invariant under the normal evolution, namely $G_{\sigma_t}=G_\sigma$, as long as $\sigma_t$ is an immersion.
\end{cor}

\begin{remark}
It follows from Equation \eqref{eq:normal evolution hyperboloid2} that, for any immersion $\sigma:M^n\to \Hyp^{n+1}$, the differential of $d\sigma_t$ at a point $p$ is injective for small $t$. However, in general  $\sigma_t$ might fail to be a global immersion for all $t\neq 0$.
In the next section we will discuss the condition of small principal curvatures for $\sigma$, which is a sufficient condition to ensure that $\sigma_t$ remains an immersion for all $t$.

As a related phenomenon, it is possible to construct examples of immersions $\zeta:M^n\to T^1\Hyp^{n+1}$ {which are orthogonal to the fibers of $\mathrm p$ but} such that $\pi\circ\varphi_t\circ\zeta$ fails to be an immersion for all $t\in\R$. We will discuss this problem later on, {and} such an example is {exhibited} in Example \ref{ex: Lagrangian not globally integrable}.
\end{remark}

\subsection{Fundamental examples}\label{subsec:examples}

It is now useful to describe several explicit examples. All of {them} will actually play a role in some of the proofs in the next sections.

\begin{example}[Totally geodesic hyperplanes] \label{ex:totally geodesic}
Let us consider a totally geodesic hyperplane $\mathcal P$ in $\Hyp^{n+1}$, and let $\sigma:\mathcal P\to \Hyp^{n+1}$ be the inclusion map. Since in this case the shape operator vanishes everywhere, from Equation \eqref{eq:fff gauss} the Gauss map is an isometric immersion (actually, an embedding) into $\G{n+1}$ with respect to the first fundamental form of $\sigma$. Totally geodesic immersions are in fact the only cases for which this occurs.

A remark that will be important in the following is that the lift $\zeta_\sigma$ is horizontal: that is, by Equation \eqref{eq:differential lift}, $d\zeta_\sigma(w)$ equals the horizontal lift of $w$ for every vector $w$ tangent to $\mathcal P$ at $x$. Therefore for every $x\in \mathcal P$, the image of $d\zeta_\sigma$ at $x$ is exactly the horizontal subspace $\HP_{(x,\nu(x))}$, for $\nu(x)$ the unit normal vector of $\mathcal P$. 
\end{example}

\begin{example}[Spheres in tangent space] \label{ex:spheres}
A qualitatively opposite example is the following. Given a point $x\in\Hyp^{n+1}$, let us choose an isometric identification of $T_x \Hyp^{n+1}$ with the $(n+1)$-dimensional Euclidean space, and consider the $n$-sphere $\Sph^n$ as a hypersurface in $T_x \Hyp^{n+1}$ by means of this identification. Then we can define the map 
$$\zeta:\Sph^n\to T^1 \Hyp^{n+1}\qquad \zeta(v)=(x,v)~.$$
The differential of $\zeta$ reads $d\zeta_v(w)=(0,w)=w^\V$ for every $w\in T_v\Sph^n\cong v^\perp$, hence $\zeta$ is an immersion, which is orthogonal to the fibers of $\mathrm{p}$. Actually,  $\zeta$ is vertical: this means that $d\zeta_v(w)$ is the vertical lift of $w$ for every $w\in v^\perp$, and therefore  $d_v\zeta(T_v \Sph^n)$ is exactly the vertical subspace $\VP_{(x,v)}$. 

Clearly we are not in the situation of Propositions \ref{prop:gauss immersion} and \ref{prop:gauss immersion converse}, as $\pi\circ \zeta$ is a constant map. On the other hand, $\mathrm{p}\circ \zeta$ has image in $\G{n+1}$ consisting of all the (oriented) geodesics $\ell$ going through $x$. {However, when post-composing $\zeta$ with the geodesic flow, $\varphi_t\circ \zeta$ projects to an immersion in $\Hyp^{n+1}$} {for all $t\ne 0$} {and is in fact an embedding with image a geodesic sphere of $\Hyp^{n+1}$ of radius $|t|$ centered at $x$.}
As a final remark, the first fundamental form of $\mathrm{p}\circ \zeta$,  is \emph{minus} the spherical metric of $\Sph^n$, since by Definition \ref{defi:parasasaki} $\gs{n+1}(w^\V,w^\V)=-\langle w,w\rangle$. 
\end{example}

The previous two examples can actually be seen as special cases of a more general construction, which will be very useful in the next section. 
\begin{example}[A mixed hypersurface in the unit tangent bundle] \label{ex:mixed}
Let us consider a totally geodesic $k$-dimensional submanifold $Q$ of $\Hyp^{n+1}$, for $0\leq k\leq n$. Consider the unit normal bundle 
$$\mathrm N^1 Q=\{(x,v)\in T^1\Hyp^{n+1}\,|\,x\in Q,v\text{ orthogonal to }Q  \}~,$$
which is an $n$-dimensional submanifold of $T^1\Hyp^{n+1}$, and 
let $\iota$ be the obvious inclusion map of $\mathrm N^1 Q$ in $T^1\Hyp^{n+1}$. Observe that $\pi\circ \iota$ is nothing but the bundle map $\mathrm N^1 Q\to Q$, hence not an immersion unless $k=n$ which is the case we discussed in Example \ref{ex:totally geodesic}. The map $\mathrm{p}\circ\iota$ has instead image in $\G{n+1}$ which consists of all the oriented geodesics $\ell$ orthogonal to $Q$.  See Figure \ref{fig:cylinder}.
Let us {focus on} its geometry in $\G{n+1}$.  

Given  $(x,v)\in N^1 Q$, take an orthonormal basis $\{w_1,\ldots,w_k\}$ of $T_x Q$. Clearly the $w_i${'s} are orthogonal to $v$, and let us complete them to an orthonormal basis $\{w_1,\ldots,w_n\}$ of $v^\perp\subset T_x \Hyp^{n+1}$. Then $\{w_1,\ldots,w_n\}$ identifies to a basis of $T_{(x,v)}\mathrm N^1 Q$. By repeating the arguments of the previous two examples, $d\iota_{(x,v)}(w_i)$ is the \emph{horizontal} lift of $w_i$ at $(x,v)$ if $1\leq i\leq k$, and is the \emph{vertical} lift if $i>k$. In particular they are all orthogonal to $\chi_{(x,v)}$, and therefore the induced metric on $\mathrm N^1 Q$ by the metric $\gs{n+1}$ coincides with the first fundamental form of $\mathrm{p}\circ\iota$. This metric has signature $(k,n-k)$, and $\{w_1,\ldots,w_n\}$ is an orthonormal basis, for which $w_1,\ldots,w_k$ are positive directions and $w_{k+1},\ldots,w_n$ negative directions. 

{Similarly to the previous example, for all $t\ne 0$ the map $\varphi_t\circ\iota$ has the property that, when post-composed with the projection $\pi$, it gives an embedding with image the equidistant ``cylinder'' around $Q$.}
\end{example}

\begin{figure}[htbp]
\centering
\includegraphics[height=4.5cm]{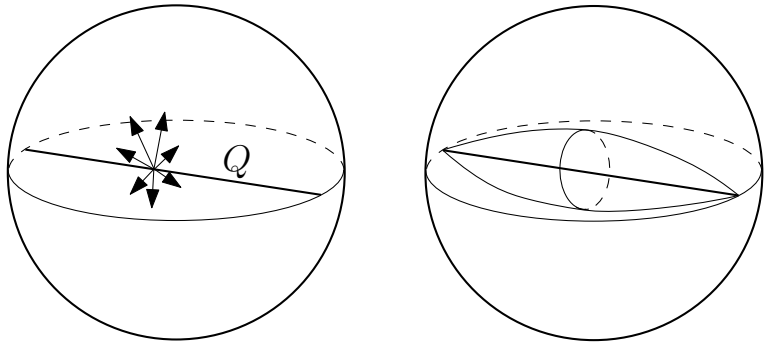} 

\caption{The normal bundle $N^1Q$ of a $k$-dimensional totally geodesic submanifold $Q$ in $\Hyp^{n+1}$ (here $k=1$ and $n=2$). On the right: after composing with the geodesic flow $\varphi_t$ for $t\neq 0$, one obtains an equidistant cylinder.}\label{fig:cylinder}
\end{figure}

Let us now consider a final example, which allows also to {prove} the integrability of the almost para-complex structure $\JJ$ of $\G{n+1}$ we introduced in Lemma \ref{lemma:paracomplex structure}.

\begin{example}[Horospheres] \label{ex:horospheres}
Let us consider a horosphere $H$ in $\Hyp^{n+1}$, and apply the Gauss map construction of Definition \ref{defi:lift and gauss} to the inclusion $\sigma:H\to \Hyp^{n+1}$.
	
It is known that {the shape operator of $H$ is $\pm\mathrm{id}$ (the sign varies according to the choice of the normal vector field, or equivalently by the choice of orientation on $H$)}, a result we will also deduce later on from our arguments in Remark \ref{rmk princ curv horosphere}.  Define $\zeta_\pm$ as the lift of $\sigma$ induced by the choice of the normal vector field for which the shape operator is $\pm \mathrm{id}$.

Now, by Proposition \ref{prop:gauss immersion}, the lift $\zeta_{\pm}$ is orthogonal to the fibers of $\mathrm{p}$, and moreover, by Equation \eqref{eq:differential lift}, $d\zeta_{\pm}(w)=w^\HH\pm w^\V$. As a result, by Equation \eqref{eq:defiJ2}, one has in fact that the image of $d_x\zeta_{\pm}$ is the whole $(\pm 1)$-eigenspace {of $J$} in $T_{\zeta_\pm(x)}T^1 \Hyp^n$.
\end{example}

A direct application of Example \ref{ex:horospheres} shows that the almost para-complex structure $\JJ$ is integrable:

\begin{cor}
The $(1,1)$-tensor $\JJ$ is a para-complex structure on $\G{n+1}$.
\end{cor}
\begin{proof}
Given $(x,v)\in T^1\Hyp^{n+1}$, consider the two horospheres $H^\pm$ containing $x$ with normal vector $v$ at $x$. The vector $v$ points to the convex side of one of them, and to the concave side of the other. Let us orient them in such a way that $v$ is compatible with the ambient orientation. Then Example \ref{ex:horospheres} shows that the Gauss maps of the horospheres $H^\pm$ have image integral submanifolds for the distributions $T^\pm\G{n+1}$ associated to the almost para-complex structure $\JJ$, which is therefore integrable. 
\end{proof}

\section{Immersions with small principal curvatures}\label{sec:small princ curv}

In this section we define and study the properties of immersed hypersurfaces in $\Hyp^{n+1}$ with small principal curvatures and their Gauss map.

\subsection{Extrinsic geometry of hypersurfaces}\label{sec:extr geom small princ curv}
Let us start by defining our condition of small principal curvatures. Recall that the principal curvatures of an immersion of a hypersurface in a Riemannian manifold (in our case the ambient manifold is $\Hyp^{n+1}$) are the eigenvalues of the shape operator, which was defined in \eqref{eq:shape}.

\begin{defi}
An immersion $\sigma\colon M^n \to \Hyp^{n+1}$ has \emph{small principal curvatures} if its principal curvatures at every point {lie in $(-1,1)\subset \R$}. 
\end{defi}

As a consequence of Equation \eqref{eq:fff gauss}, we have a direct characterization of immersions with small principal curvatures in terms of their Gauss map:

\begin{prop}\label{prop: small curv sse riemannian}
Given an oriented manifold $M^n$ and an immersion $\sigma:M\to\Hyp^{n+1}$, $\sigma$ has small principal curvatures if and only if its Gauss map $G_\sigma$ is a Riemannian immersion.
\end{prop}

We recall that an immersion into a pseudo-Riemannian manifold is \emph{Riemannian} if the pull-back of the ambient pseudo-Riemannian metric, which in our case is {the first fundamental form  $\overline\I=G_\sigma^*\GG$}, is a Riemannian metric.

\begin{proof}[Proof of Proposition \ref{prop: small curv sse riemannian}]
The condition that the Gauss map is a Riemannian immersion is equivalent to ${\overline\I}(W,W)>0$ for every $W\neq 0$. By Equation \eqref{eq:fff gauss}, this is equivalent to $\| B(W)\|^2<\|W\|^2$ for the norm on $M$ induced by $\I$, and this is equivalent to the eigenvalues of $B$ (that is, the principal curvatures) being strictly smaller than 1 in absolute value. 
\end{proof}

\begin{remark} \label{rmk: negative curv}
Let us observe that a consequence of the hypothesis of small principal curvatures is that the first fundamental form of $\sigma$ has negative sectional curvature. Indeed, if $V,W$ is a pair of orthonormal vectors on $T_p M$, then by the Gauss' equation the sectional curvature of the plane spanned by $V$ and $W$ is:
$$K_{\mathrm{Span}(V,W)}=-1+\II(V,V)\II(W,W)-\II(V,W)^2~.$$
Since the principal curvatures of $\sigma$ are less than one in absolute value, we have $\| B(V)\|<\|V\|$ and the same for $W$. Moreover $V$ and $W$ are unit vectors, hence both $|\II(V,V)|$ and $|\II(W,W)|$ are less than one and the sectional curvature is negative.
\end{remark}

Recall that we introduced in Definition \ref{defi normal evolution} the normal evolution $\sigma_t$ of an immersion $\sigma:M\to\Hyp^{n+1}$, for $M$ an oriented $n$-manifold. An immediate consequence of Proposition \ref{prop: small curv sse riemannian} is the following:

\begin{cor}\label{cor:equidistant is immersed}
Given an oriented manifold $M^n$ and an immersion $\sigma:M\to\Hyp^{n+1}$ with small principal curvatures, for all $t\in\R$ the normal evolution $\sigma_t$ is an immersion with small principal curvatures.
\end{cor}
\begin{proof}
It follows from Equation \eqref{eq:normal evolution hyperboloid2} that $\sigma_t$ is an immersion if the shape operator $B$ of $\sigma$ satisfies $\| B(W)\|^2<\|W\|^2$ for every $W\neq 0$, that is, if $\sigma$ has small principal curvatures. Since the Gauss map is invariant under the normal evolution by Corollary \ref{prop:gauss map invariant normal evo}, $\sigma_t$ has small principal curvatures for all $t$ as a consequence of Proposition \ref{prop: small curv sse riemannian}.
\end{proof}

It will be useful to describe more precisely, under the hypothesis of small principal curvatures, the behaviour of the principal curvatures under the normal evolution.

\begin{lemma} \label{lemma:evolution fsigma}
Given an oriented manifold $M^n$ and  an immersion $\sigma:M\to\Hyp^{n+1}$ with small principal curvatures, let $f_\sigma:M\to\R$ be the function
\begin{equation} \label{eq:aux function}
 f_\sigma(p)=\frac{1}{n}\sum_{i=1}^n\arctanh(\lambda_i(p))~,
\end{equation}
where $\lambda_1(p),\ldots,\lambda_n(p)$ denote the principal curvatures of $\sigma$ at $p$. Then $f_{\sigma_t}=f_\sigma-t$ for every $t\in \R$.
\end{lemma}
\begin{proof}
We showed in the proof of Proposition \ref{prop: flusso geod e flusso normale} that in the hyperboloid model of $\Hyp^{n+1}$ the normal vector of $\sigma_t$, compatible with the orientations of $M$ and $\Hyp^{n+1}$, has the expression 
$$\nu_t(p)=\sinh(t)\sigma(p)+\cosh(t)\nu(p)~,$$
where $\nu=\nu_0$ is the normal vector for $\sigma=\sigma_0$. Using also Equation \eqref{eq:normal evolution hyperboloid2}, the shape operator $B_t$ of $\sigma_t$, whose defining condition is 
$d\sigma_t\circ B_t(W)=-D_{W} \nu_t$ as in Equation \eqref{eq:shape}, is:
\begin{equation}\label{eq:shape normal evo}
B_t=(\mathrm{id}-\tanh(t)B)^{-1}\circ(B-\tanh(t)\mathrm{id})~.
\end{equation}
First, Equation \eqref{eq:shape normal evo} shows that if $V$ is a principal direction (i.e. an eigenvalue of the shape operator) for $\sigma$, then it is also for $\sigma_t$. Second, if $\lambda_i$ is a principal curvature of $\sigma$, then the corresponding principal curvature for $\sigma_t$ is 
{\begin{equation}\label{eq della discordia}
\frac{\lambda_i-\tanh(t)}{1-\tanh(t)\lambda_i}=\tanh(\mu_i-t)~,
\end{equation}} where $\mu_i=\arctanh(\lambda_i)$. The formula $f_{\sigma_t}=f_\sigma-t$ then follows. 
\end{proof}

\begin{remark}\label{rmk fsigma smooth}
Although the principal curvatures of $\sigma$ are not smooth functions, the function $f_\sigma$ defined in \eqref{eq:aux function} is smooth as long as $\sigma$ has small principal curvatures. Indeed, using the expression of the inverse hyperbolic tangent in terms of the elementary functions, we may express:
$$f_\sigma(p)=\frac{1}{2n}\left(\sum_{i=1}^n\log\left(\frac{1+\lambda_i(p)}{1-\lambda_i(p)}\right)\right)=\frac{1}{2n}\log\left(\frac{\prod_{i=1}^n(1+\lambda_i(p))}{\prod_{i=1}^n(1-\lambda_i(p))}\right)=\frac{1}{2n}\log\left(\frac{\det(\mathrm {id}+B)}{\det(\mathrm{id}-B)}\right)~,$$
where $B$ is the shape operator of $\sigma$ as usual.
This proves the smoothness of $f_\sigma$, which is implicitely used in Proposition \ref{Prop: formula H in G}.
\end{remark}

\subsection{Comparison horospheres}

Our next goal is to discuss global injectivity of immersions with small principal curvatures (Proposition \ref{prop injectivity}) and {of} their Gauss maps (Proposition \ref{prop:gauss maps diffeo onto image}), {under the following completeness assumption.}
\begin{defi}
An immersion $\sigma\colon M^n \to \Hyp^{n+1}$  is \emph{complete} if the first fundamental form $\I$ is a complete Riemannian metric.
\end{defi}

Here we provide some preliminary steps.

\begin{defi}\label{defi:convex immersion}
{Given an oriented manifold $M^n$ and  an immersion $\sigma:M\to\Hyp^{n+1}$, let $B=-D\nu$ be its shape operator with respect to the unit normal vector field $\nu$ compatible with the orientations of $M$ and $\Hyp^{n+1}$. We say that $\sigma$} is (\emph{strictly}) \emph{convex}  if {$B$} is negative semi-definite (resp. definite), and, conversely, that it is (\emph{strictly}) \emph{concave} if {$B$} is positive semi-definite (resp. definite).
\end{defi}
{When $\sigma$ is an embedding, we refer to its image as a (strictly) convex/concave hypersurface.} Clearly reversing the orientation (and therefore the normal vector field) of a (strictly) convex hypersurface it becomes (strictly) concave, and viceversa.

A classical fact is 
	that a properly embedded strictly convex hyperurface in $\Hyp^{n+1}$ disconnects it into two connected components and that exactly one of them is geodesically convex (the one towards which $-\nu$ is pointing): we denote the closure of this connected component as the \emph{convex side} of the hypersurface, and denote as the \emph{concave side} the closure of the other one.

{We need another definition before stating the next Lemma.} We say that a smooth curve $\gamma:[a,b]\to\Hyp^n$ parameterized by arclength has \emph{small acceleration} if $\| D_{\gamma'(t)}\gamma'(t)\|<1$ for all $t$, where $D$ denotes the Levi-Civita connection of $\Hyp^n$ as usual.

\begin{lemma}\label{lemma curve small acc} \label{rmk:strictly concave side} 
Let $\gamma:[a,b]\to\Hyp^n$ be a smooth curve of small acceleration. Then the image of $\gamma$ lies on the concave side of any horosphere tangent to $\gamma$.
{More precisely, $\gamma$ lies in the interior of the concave side except for the tangency point. }
\end{lemma}
\begin{proof}
Up to reparametrization we can assume that the tangency point is $\gamma(0)$, and we shall prove that $\gamma(t)$ lies on the concave side of any horosphere tangent to $\gamma'(0)$ for every $t>0$. Recall that we are also assuming that $\gamma$ is parameterized by arclength. We will use the upper half-space model of $\Hyp^n$, namely, $\Hyp^n$ is the region $x_n>0$ in $\R^n$ endowed with the metric $(\frac 1 {x_n^2})(dx_1^2+\ldots+dx_n^2)$. Up to isometry, we can assume that $\gamma(0)=(0,\ldots,0,1)$, $\gamma'(0)=(1,0,\ldots,0)$ and that the tangent horosphere is $\{x_n=1\}$. 

{Let us first} show that {$\gamma$ lies on the concave side of the horosphere for small $t$, namely, denoting $\gamma(t)=(\gamma_1(t),\ldots,\gamma_n(t))$, that} $\gamma_n(t)<1$ for small $t$. Since $\gamma_n(0)=1$ and $\gamma_n'(0)=0$, it will be sufficient to check that $\gamma_n''(0)<0$. Using the assumption on $\gamma'(0)$ and a direct computation of the Christoffel symbols $\Gamma_{11}^n=1$, we get
$$(D_{\gamma'}\gamma')_n(0)=\gamma_n''(0)+1~.$$
Since by hypothesis $\gamma$ has small acceleration, and at $\gamma(0)$ the metric of the upper half-space model coincides with the standard metric $dx_1^2+\ldots+dx_n^2$, $|(D_{\gamma'}\gamma')_n(0)|<1$ and therefore $\gamma_n''(0)<0$. {We conclude that, for suitable $\epsilon>0$, $\gamma_n(t)< 1$ for all $t\in (-\epsilon, \epsilon)\setminus \{0\}$.}

Let us now show that $\gamma(t)$  {lies in the interior of} the concave side of the tangent horosphere $\{x_n=1\}$ for all {$t\ne 0$}, that is, that {$\gamma_n(t)< 1$ for all $t\ne 0$}. Suppose by contradiction that $\gamma_n(t_0)=1$ for some {$t_0\geq \epsilon$}. Then $\gamma_n$ has a minimum point {$t_{\min}$} in $(0,t_0)$, with minimum value $m<1$. The horosphere $\{x_n=m\}$ is then tangent to $\gamma$ at {$\gamma(t_{\min})$} and $\gamma_n(t)\geq m$ for $t$ in a neighbourhood of $t_{\min}$. By re-applying the argument of the previous part of the proof, this gives a contradiction. See Figure \ref{fig:tangenthorospheres}.
\end{proof}

	\begin{figure}[htbp]
\centering
\includegraphics[height=5.5cm]{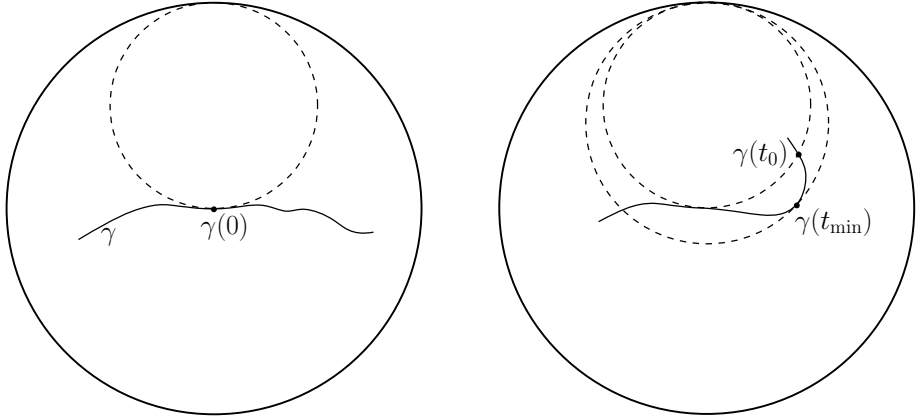} 

\caption{A schematic picture of the argument in the proof of Lemma \ref{lemma curve small acc}. On the left, for $t\in(-\epsilon,\epsilon)$ the image of the curve $\gamma(t)$ lies in the concave side of the horosphere tangent to $\gamma$ at $t=0$. On the right, the same holds in fact for every $t$, for otherwise one would obtain a contradiction with the first part of the proof at the minimum point $t_{\min}$.}\label{fig:tangenthorospheres}
\end{figure}

\begin{remark}\label{rmk princ curv horosphere}
Given an immersion $\sigma:M^n\to \Hyp^{n+1}$ (or in a general Riemannian manifold), a curve $\gamma:[a,b]\to M$ is a geodesic for the first fundamental form {of $\sigma$} (in short, {it} is an \emph{intrinsic} geodesic) if and only if $D_{(\sigma\circ\gamma)'}(\sigma\circ\gamma)'$ is orthogonal to the image of $\sigma$. In this case we have indeed
\begin{equation}\label{eq:acceleration intrinsic geodesic}
D_{(\sigma\circ\gamma)'}(\sigma\circ\gamma)'=\II(\gamma'(t),\gamma'(t))\nu(\gamma(t))
\end{equation}
where $\nu$ is the unit normal vector of the immersion with respect to the chosen orientations.

By applying this remark to an intrinsic geodesic for the horosphere $\{x_n=1\}$, which has the form $\gamma(t)=(a_1t,\ldots,a_{n-1}t,1)$ (here $\sigma$ is simply the inclusion), and repeating  the same computation of the proof of Lemma \ref{lemma curve small acc}, we see that the second fundamental form of a horosphere equals the first fundamental form. Hence the principal curvatures of a horosphere are all identically equal to $1$ for the choice of inward normal vector, and therefore the shape operator is the identity at every point, a fact we have already used in Example \ref{ex:horospheres}. 
\end{remark}

An immediate  consequence of Lemma \ref{lemma curve small acc} is the following:
\begin{lemma} \label{lemma concave side horospheres}
Given a {complete}  immersion $\sigma:M^n\to\Hyp^{n+1}$ with small principal curvatures, the image of $\sigma$ lies strictly on the concave side of any tangent horosphere. That is, for every $p\in M$, $\sigma(M\setminus \{p\})$ lies  in the interior of the concave side of each of the two horospheres tangent to $\sigma$ at $\sigma(p)$.
\end{lemma}
\begin{proof}

Let us fix $p\in M$ and let $q\in M$, with $p\ne q$. By completeness
 there exists an intrinsic geodesic $\gamma$ on $M$ joining $p$ and $q$, which we assume to be parameterized by arclength. Applying Equation \eqref{eq:acceleration intrinsic geodesic} as in Remark \ref{rmk princ curv horosphere}, we have 
$$\| D_{(\sigma\circ\gamma)'}(\sigma\circ\gamma)'\|=|\II(\gamma'(t),\gamma'(t))|<\I(\gamma'(t),\gamma'(t))=\|(\sigma\circ\gamma)'(t)\|^2=1~,$$
hence $\sigma\circ\gamma$ has small acceleration. The conclusion follows from Lemma \ref{lemma curve small acc}.
\end{proof}

	\begin{remark}\label{rmk: tangent metric spheres}
		Observe that any metric sphere in $\Hyp^{n+1}$ is contained in the convex side of any tangent horosphere. As a result, a hypersurface with small principal curvatures lies in the complementary of any metric ball of {$\Hyp^{n+1}$} whose boundary is tangent to the hypersurface. See Figure \ref{fig:tangentsurfaces}.
	\end{remark}

\begin{remark}\label{rmk:convex side caps}
A $r$-\emph{cap} in the hyperbolic space is the hypersurface at (signed) distance $r$ from a totally geodesic plane. By a simple computation (for instance using Equation \eqref{eq:shape normal evo}), $r$-caps are umbilical hypersurfaces with principal curvatures  {identically equal to $-\tanh(r)$, computed with respect to the unit normal vector pointing to the side where $r$ is increasing}. Now, if $\sigma:M\to\Hyp^{n+1}$ is  an immersion with principal curvatures smaller that $\epsilon=\tanh(r){\in (0,1)}$ in absolute value, then one can repeat wordly the proofs of Lemma \ref{lemma curve small acc} and Lemma \ref{lemma concave side horospheres}, by replacing horospheres with $r$-caps, and conclude that the image of $\sigma$ lies strictly on the {concave} side of every tangent $r$-cap {for $r=\arctanh(\epsilon)$}.  See Figure \ref{fig:tangentsurfaces}.
A similar conclusion (which is however not interesting for the purpose of this paper) could of course be obtained  under the assumption that $\sigma$ has principal curvatures bounded by some constant $\epsilon>1$, in terms of tangent {metric} spheres with curvature greater than $\epsilon$ in {absolute value}.
\end{remark}

\begin{figure}[htbp]
\centering
\includegraphics[height=5.5cm]{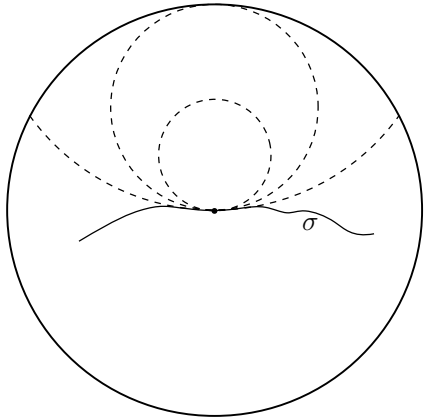} 

\caption{Schematically, an immersion $\sigma$ tangent at one point to a metric sphere (whose principal curvatures are larger than $1$), a horosphere (equal to $1$) and a $r$-cap (smaller than $1$). The image of $\sigma$ is contained in the concave side of the three of them.}\label{fig:tangentsurfaces}
\end{figure}

\subsection{Injectivity results}
Having established these preliminary results, let us finally discuss the global injectivity of $\sigma$ and $G_\sigma$ under the hypothesis of completeness.
{Before that, we relate the completeness assumption for $\sigma$ to some topological conditions}.

	\begin{remark} \label{rmk:proper implies complete}
	Let us observe that proper immersions $\sigma\colon M\to \Hyp^{n+1}$ are complete.  
	{Indeed, i}f $p,q\in M$ have distance {at most} $r$ for the first fundamental form $\I$, then, by definition of distance on a Riemannian manifold,  $dist_{\Hyp^{n+1}}(\sigma(p),\sigma(q))\le r$: as a result, 
	\[
	\sigma(B_{\I}(x,r)) \subset B_{\Hyp^{n+1}}(x,r).
	\]
	Assuming $\sigma$ is proper, $\sigma^{-1}(\overline{B_{\Hyp^{n+1}}(x,r)})$ is a compact subspace of $M$ containing $B_{\I}(x,r)$, therefore $\overline {B_{\I}(x,r)}$ is compact. We conclude that $\I$ is complete by Hopf-Rinow Theorem.
\end{remark}

{A less trivial result is that Remark \ref{rmk:proper implies complete} can be reversed for immersions with small principal curvatures: in fact,} for immersions with small principal curvatures, being properly immersed, properly embedded and complete are all equivalent conditions

\begin{prop} \label{prop injectivity}
Let $M^n$ be a manifold and $\sigma:M\to\Hyp^{n+1}$ be a {complete}  immersion  with small principal curvatures. Then $\sigma$ is a proper embedding and $M$ is diffeomorphic to $\R^n$.
\end{prop}
\begin{proof}
To show that $\sigma$ is injective, let us suppose {by contradiction} that $\sigma(p)=\sigma(q)=y_0$ for $p\neq q$. Let $\gamma:[a,b]\to M$ be an intrinsic {$\I$-}geodesic joining $p$ and $q$ {parametrized by arclength}, which exists because $\I$ is complete. As in Lemma \ref{lemma concave side horospheres}, $\sigma\circ\gamma$ has small acceleration. Let 
$$r_0:=\max_{t\in [a,b]}d(y_0,\sigma\circ\gamma(t))~.$$
Then $\sigma\circ\gamma$ is tangent at some point $\sigma\circ\gamma(t_0)$ to the metric sphere in {$\Hyp^{n+1}$} centered at $y_0$ of radius $r_0$, and contained in its convex side. {By Remark \ref{rmk: tangent metric spheres}, $\sigma\circ\gamma$ lies in the convex side of the horosphere tangent to the hypersurface at $\sigma\circ\gamma(t_0)$.} This contradicts Lemma \ref{lemma curve small acc} and shows that $\sigma$ is an injective immersion. See Figure \ref{fig:tangenthorospheres2}.

\begin{figure}[htbp]
\centering
\includegraphics[height=5.5cm]{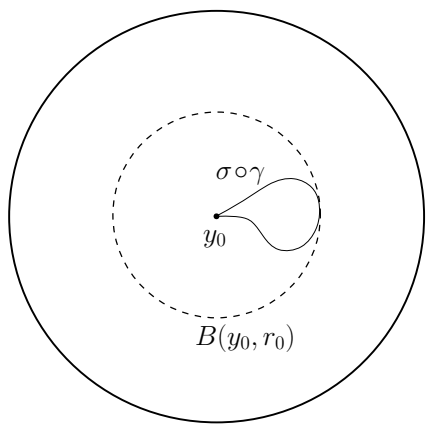} 

\caption{A sketch of the proof of the first part of Proposition \ref{prop injectivity}, namely the injectivity of $\sigma$. If $\sigma(p)=\sigma(q)=y_0$ for $p\neq q$, then the image $\sigma\circ\gamma$ of a $\I$-geodesic connecting $p$ and $q$ would be tangent to a metric ball centered at $y_0$, which contradicts the assumption that $\sigma$ has small principal curvatures.}\label{fig:tangenthorospheres2}
\end{figure}

It follows that $M$ is simply connected. Indeed, let $c:\widetilde M\to M$ be a universal covering. If $M$ were not simply connected, then $c$ would not be injective, hence $\sigma\circ c$ would give a non-injective immersion in $\Hyp^{n+1}$ with small principal curvatures, contradicting the above part of the proof. Since the first fundamental form is a {complete} negatively curved Riemannian metric on $M$ (Remark \ref{rmk: negative curv}), $M$ is {diffeomorphic to $\R^n$} {by the Cartan-Hadamard theorem}. 
	
Let us now show that $\sigma$ is proper, which also implies that it is a homeomorphism onto its image and thus an embedding. As a first step, suppose $y_0\in\Hyp^{n+1}$ is in the closure of the image of $\sigma$. We claim that the normal direction of $\sigma$ extends to $y_0$, meaning that there exists a vector $\nu_0\in T^1_{x_0}\Hyp^{n+1}$ such that  $[\nu(p_n)]\to [\nu_0]$ for every sequence $p_n\in M$ satisfying $\sigma(p_n)\to y_0$, where $\nu(p)$ denotes the unit normal vector of $\sigma$ at $p$ and $[\cdot]$ denotes the equivalence class up to multiplication by $\pm 1$. By compactness of unit tangent spheres, if $\sigma(p_n)\to y_0$ then one can extract a  subsequence $\nu(p_n)$ converging to (say) $\nu_0$. Observe that by Lemma \ref{lemma concave side horospheres}, the image of $\sigma$ lies in the concave side of any horosphere orthogonal to $\nu(p_n)$ at $\sigma(p_n)$. By a continuity argument, it lies also on the concave side of each of the two horospheres orthogonal to $\nu_0$ at $y_0$. The claim follows by a standard subsequence argument once we show that there can be no limit other than $\pm\nu_0$ along any subsequence. 

{We will assume hereafter, in the upper half-space model 
$$(\{(x_1,\ldots,x_{n+1})\in\R^{n+1}\,|\,x_{n+1}>0\},\frac 1 {x_{n+1}^2}(dx_1^2+\ldots+dx_{n+1}^2))~,$$ 
that $y_0=(0,\ldots,0,1)$ and $\nu_0=(0,\ldots,0,1)$. {See Figure \ref{fig:graphs} on the left.} In this model, horospheres are either horizontal hyperplanes $\{x_{n+1}=c\}$ or spheres with south pole on $\{x_{n+1}=0\}$. By Lemma \ref{lemma concave side horospheres}, the image of $\sigma$ is contained in the concave side of both horospheres orthogonal to $\nu_0$, hence it lies in the region defined by $0<x_{n+1}\leq 1$ and $x_1^2+\ldots+x_{n}^2+(x_{n+1}-\frac{1}{2})^2\geq \frac{1}{4}$. Now, if $\nu_1\neq \pm\nu_0$ were a subsequential limit of $\nu(q_n)$ for some sequence $q_n$ with $\sigma(q_n)\to y_0$, then the image of $\sigma$ would lie on the concave side of some sphere with south pole on $\{(x_1,\ldots,x_{n+1})\in\R^{n+1}\,|\,x_{n+1}=0,(x_1,\ldots,x_{n})\neq(0,\ldots,0)\}$. But then $\sigma$ would either enter the region $x_{n+1}>1$ or the region $x_1^2+\ldots+x_{n}^2+(x_{n+1}-\frac{1}{2})^2< \frac{1}{4}$ in a neighbourhood of $y_0$, which gives a contradiction.}

Having established the convergence of the normal direction to $[\nu_0]$, we can now find a neighbourhood $U$ of $y_0$ of the form $B(0,\epsilon)\times(\frac{1}{2},\frac{3}{2})$, where $B(0,\epsilon)$ is the ball of Euclidean radius $\epsilon$ centered at the origin in $\{x_{n+1}=0\}\subset\R^{n+1}$, such that if $\sigma(p)\in U$, then the vertical projection from the tangent space of $\sigma$ at $\sigma(p)$ to $\{x_{n+1}=0\}$ is a linear isomorphism. By the implicit function theorem, $\sigma(M)\cap U$ is locally a graph over $\R^n$. Up to taking a smaller $\epsilon$, we can {arrange $U$ so that $\sigma(M)\cap U$ is a global graph over some open set of  $B(0,\epsilon)\subset\R^n$. Indeed as long as the normal vector $\nu$ is in a small neighbourhood of $\pm\nu_0$, the vertical lines over points in $B(0,\epsilon)$ may intersect the image of $\sigma$ in at most one point as a consequence of Lemma \ref{lemma concave side horospheres}. Let us denote $V\subseteq B(0,\epsilon)$ the image of the vertical projection from $\sigma(M)\cap U$ to $\R^n$, so that $\sigma(M)\cap U$ is 
the graph of some function 
 $h:V\to(\frac{1}{2},\frac{3}{2})$ satisfying $h(0)=1$}.  Since the gradient of $h$ converges to $0$ at $0$, {up to restricting $U$ again, there is a constant $C>0$ such that the Euclidean norm of the gradient of $h$ in $U$ is bounded by $C$.}

We shall now apply again the hypothesis that $\sigma$ is complete to show that in fact $V=B(0,\epsilon)$. For this purpose, we assume that $V$ is a proper (open) subset of $B(0,\epsilon)$ and we will derive a contradiction. Under the assumption $V\neq B(0,\epsilon)$ we would find a Euclidean segment $c:[0,1]\to\R^n$ such that $c(s)\in V$ for $s\in[0,1)$ and $c(1)\in B(0,\epsilon)\setminus V$. The path $s\mapsto (c(s),h\circ c(s))$ is contained in $\sigma(M)$;  
using $h\geq\frac{1}{2}$ and the bound on the gradient, we obtain that its hyperbolic length is less than $2\sqrt{1+C^2}$ times the Euclidean length of $c$, hence is finite. This contradicts completeness of $\sigma$. 
In summary,  $\sigma(M)\cap U$ is the graph of a function globally defined on $B(0,\epsilon)$, and clearly contains the point $y_0$.	{See Figure \ref{fig:graphs} on the right.}

{We are now ready to complete the proof of the fact that $\sigma$ is proper. Indeed, let $p_n\in M$ be a sequence such that $\sigma(p_n)\to y_0$. We showed above that $y_0$ is in the image of $\sigma$ (say $y_0=\sigma(p_0)$) and that $p_n$ is definitively in $\sigma^{-1}(U)$, whose image is a graph over $B(0,\epsilon)$. Hence $p_n$ is at bounded distance from $p_0$ for the first fundamental form of $\sigma$, and therefore admits a subsequence $p_{n_k}$ converging to $p_0$. In conclusion, $\sigma$ is a proper embedding.}
\end{proof}

\begin{figure}[htbp]
\centering
\includegraphics[height=5.5cm]{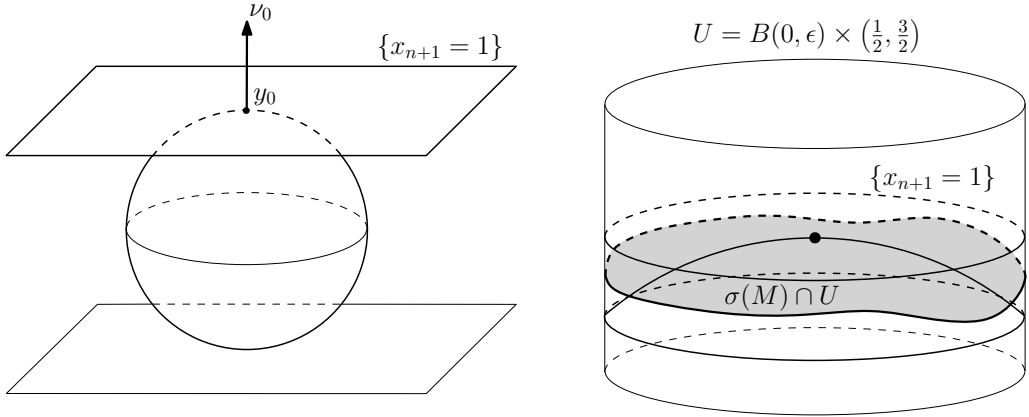} 

\caption{The setting of the proof that $\sigma$ is proper in Proposition \ref{prop injectivity}: in the upper half-plane model, the image of $\sigma$ is contained below the horosphere $\{x_{n+1}=1\}$ and in the outer side of the horosphere $x_1^2+\ldots+x_{n}^2+(x_{n+1}-\frac{1}{2})^2= \frac{1}{4}$. On the right, the neighbourhood $U$ of $y_0$, where the image of $\sigma$ is proved to be the graph of a function $h:B(0,\epsilon)\to\R$.}\label{fig:graphs}
\end{figure}

{By an application of Lemma \ref{lemma concave side horospheres} one can easily show that the Gauss map $G_\sigma:M\to \G{n+1}$ is injective as well {if $\sigma$ is complete and has small principal curvatures}. However, we will prove here (Proposition \ref{prop:gauss maps diffeo onto image}) a stronger property of the Gauss map.}

Recall that the space of oriented geodesics of $\Hyp^{n+1}$ has the natural identification 
$$\G{n+1}\cong \partial\Hyp^{n+1}\times\partial\Hyp^{n+1}\setminus \Delta~,$$
 for $\Delta$ the diagonal, given by mapping an oriented geodesic $\ell$ to its endpoints at infinity according to the orientation: {as a consequence, the map $G_\sigma$ can be seen as a pair of maps with values in the boundary of $\Hyp^n$}. More precisely, if we denote by $\gamma:\R\to\Hyp^{n+1}$ a parameterized geodesic, then the above identification reads:
\begin{equation}\label{eq:identification space geodesics}
\gamma\mapsto\left(\lim_{t\to+\infty}\gamma(t),\lim_{t\to-\infty}\gamma(t)\right)~.
\end{equation}

Given an immersion of an oriented manifold $M^n$
into $\Hyp^{n+1}$, composing the Gauss map $G_\sigma:M\to\G{n+1}$ with the above map \eqref{eq:identification space geodesics} with values in $\partial\Hyp^{n+1}\times\partial\Hyp^{n+1}$ and projecting on each factor, we obtain the so-called \emph{hyperbolic Gauss maps} $G^\pm_\sigma:M\to\partial\Hyp^{n+1}$. They are explicitely expressed by
$$G^\pm_\sigma(p)=\lim_{t\to\pm\infty}\exp_{\sigma(p)}(t\nu(p))\in \partial\Hyp^{n+1}~.$$
The following proposition states their injectivity property under the small principal curvatures assumption, which will be applied in Proposition \ref{prop:action free prop disc},

\begin{prop}\label{prop:gauss maps diffeo onto image}
Let $M^n$ be an oriented manifold and $\sigma:M\to\Hyp^{n+1}$ be a complete  immersion  with small principal curvatures. Then both hyperbolic Gauss maps $G^\pm_\sigma:M\to\partial\Hyp^{n+1}$ are diffeomorphisms onto their images. {In particular, the Gauss map $G_\sigma$ is an embedding.}
\end{prop}
\begin{proof}
Let us first show that $G_\sigma^\pm$ are local diffeomorphisms. Recalling the definition of $\sigma_t$ (Definition \ref{defi normal evolution}) and its expression in the hyperboloid model of $\Hyp^{n+1}$ (Equation \eqref{eq:normal evolution hyperboloid}), $G_\sigma^\pm$ is the limit for $t\to \pm\infty$ in $\partial\Hyp^{n+1}$ of 
$$\sigma_t(p)=\cosh(t)\sigma(p)\pm\sinh(t)\nu(p)~.$$

{Recalling the definition of the boundary of $\Hyp^{n+1}$ as the projectivization of the null-cone (Equation \eqref{eq:bdy hyperboloid model})}, we will consider the boundary at infinity of $\Hyp^{n+1}$ {as the slice of the null-cone defined by $\{x_{n+2}=1\}$}. Given $p_0\in M$, up to isometries we can assume that $\sigma(p_0)=(0,\ldots,0,1)$ and $\nu(p_0)=(1,0,\ldots,0)$, 
so that $G_\sigma^\pm(p_0)=(\pm 1,0,\ldots,0,1)$ and the tangent spaces to the image of $\sigma$ at $p_0$ and of $\partial\Hyp^{n+1}$ at $G_\sigma^\pm(p)$ are identified to the same subspace {$\{x_1=x_{n+2}=0\}$ in $\R^{n,1}$}.

To compute the differential of $G_\sigma^\pm$ at $p_0$, we must differentiate the maps
$$p\mapsto \lim_{t\to \pm\infty}\frac{\sigma_t(p)}{|\langle \sigma_t(p),\sigma(p_0)\rangle|}=\frac{\sigma(p)\pm\nu(p)}{|\langle \sigma(p)\pm\nu(p),\sigma(p_0)\rangle|}$$
at $p=p_0$. Under these identifications, a direct computation for $V\in T_{p_0}M$ gives:
$$dG_\sigma^\pm(V)=d\sigma\circ (\mathrm{id}\mp B)(V)~.$$
Hence both differentials of $G_\sigma^\pm$ are invertible at $p_0$ if the eigenvalues of $B$ are always different from $1$ and $-1$, as in our hypothesis. This shows that $G_\sigma^+$ and $G_\sigma^-$ are local diffeomorphisms. 

To see that $G_\sigma^\pm$ is injective, suppose that $G_\sigma^\pm(p)=G_\sigma^\pm(q)$. This means that $\sigma$ is orthogonal at $p$ and $q$ to two geodesics having a common point at infinity. Hence $\sigma$ is tangent at $p$ and $q$ to two horospheres $H_p$ and $H_q$ having the same point at infinity. By Lemma \ref{lemma concave side horospheres} the image of $\sigma$ must lie in the concave side of both $H_p$ and $H_q$, hence the two horospheres must coincide. But by Lemma \ref{lemma concave side horospheres} again, $\sigma(M\setminus \{p\})$ lies strictly in the concave side of $H_p$, hence necessarily $p=q$. See Figure \ref{fig:tangenthorospheres3}.

{By the invariance of the domain, $G_\sigma^\pm$ are diffeomorphisms onto their images. Under the identification between $\G{n+1}$ and $\partial\Hyp^{n+1}\times\partial\Hyp^{n+1}\setminus \Delta$ the Gauss map $G_\sigma$ corresponds to $(G_\sigma^+,G_\sigma^-)$, and it follows that $G_\sigma$ is an embedding.}
\end{proof}

	\begin{figure}[htbp]
\centering
\includegraphics[height=5.5cm]{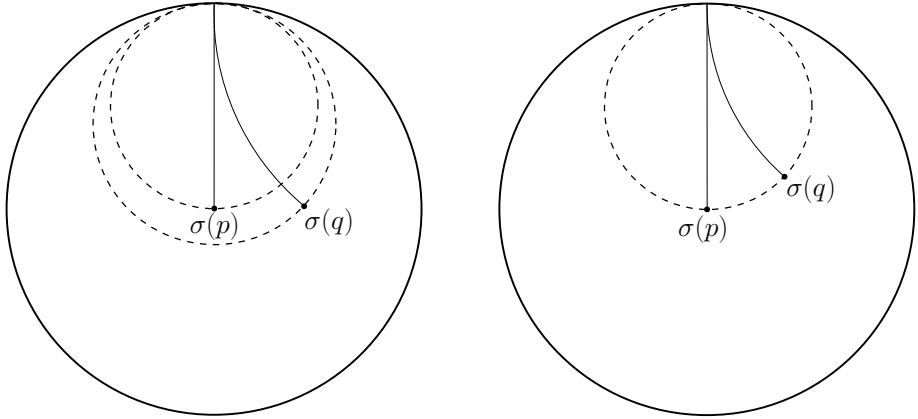} 

\caption{The proof of the injectivity of $G_\sigma^+$. Suppose two orthogonal geodesics share the final point, hence the image of $\sigma$ is tangent at $\sigma(p)$ and $\sigma(q)$ have the same point at infinity. As a consequence of Lemma \ref{lemma concave side horospheres}, this is only possible if the two tangent horospheres coincide, and therefore if $p=q$. Replacing horospheres by metric spheres, the same argument proves that the orthogonal geodesics at different points are disjoint (See Proposition \ref{prop:action free prop disc0}).}\label{fig:tangenthorospheres3}
\end{figure}

\subsection{Nearly-Fuchsian manifolds}\label{sec:nearly fuchsian}
Taking advantage of the results so far established in this section, we now introduce nearly-Fuchsian representations and manifolds. These will appear again at the end of Section \ref{sec:equivariant integrability} and in Section \ref{sec:hamiltonian}.

\begin{defi}\label{defi:nearly fuch rep}
Let $M^n$ be a closed orientable manifold. A representation $\rho:\pi_1(M)\to\Isom(\Hyp^{n+1})$ is called \emph{nearly-Fuchsian}  if there exists a $\rho$-equivariant immersion $\widetilde\sigma:\widetilde M\to\Hyp^{n+1}$ with small principal curvatures. 
\end{defi}

We recall that an immersion $\widetilde\sigma\colon \widetilde M \to \Isom(\Hyp^{n+1})$ is \emph{$\rho$-equivariant} if
 \begin{equation}\label{eq:equivariant immersion}
 \widetilde\sigma\circ\alpha= \rho(\alpha)\circ\widetilde \sigma~.\end{equation}
for all $\alpha\in \pi_1(M)$. Let us show that the action of nearly-Fuchsian representations is ``good'' on $\Hyp^{n+1}$ {(Proposition \ref{prop:action free prop disc0})} and also on a  region in $\partial\Hyp^{n+1}$ which is the disjoint union of two topological discs {(Proposition \ref{prop:action free prop disc})}.

\begin{prop}\label{prop:action free prop disc0}
Let $M^n$ be a closed orientable manifold and $\rho:\pi_1(M)\to\Isom(\Hyp^{n+1})$ be a nearly-Fuchsian representation. Then $\rho$ gives a free and properly discontinous action of $\pi_1(M)$ on $\Hyp^{n+1}$. Moreover $\rho$ is convex cocompact, namely there exists a $\rho$-invariant geodesically convex subset $\mathcal C\subset\Hyp^{n+1}$ such that the quotient $\faktor{\mathcal C}{\rho(\pi_1(M))}$ is compact.
\end{prop}
\begin{proof}
{Let $\widetilde\sigma$ be an equivariant immersion as in Definition \ref{defi:nearly fuch rep}. }We claim that the family of geodesics orthogonal to $\widetilde\sigma(\widetilde M)$ gives a foliation of $\Hyp^{n+1}$.
Observing that the action of $\pi_1(M)$ on $\widetilde M$ is free and properly discontinous, this immediately implies that the action of $\pi_1(M)$ on $\Hyp^{n+1}$ induced by $\rho$ is free and properly discontinous.
 
By repeating the same argument that shows, in the proof of  Proposition \ref{prop:gauss maps diffeo onto image}, the injectivity of $G_{\widetilde\sigma}^\pm$, replacing horospheres with {metric spheres of $\Hyp^{n+1}$} {and using Remark \ref{rmk: tangent metric spheres}}, one can prove that  two geodesics orthogonal to $\widetilde\sigma(\widetilde M)$ at different points are disjoint.

To show that the orthogonal geodesics give a foliation of $\Hyp^{n+1}$, it remains to show that  every point $x\in\Hyp^{n+1}$ is contained in a geodesic of this family (which is necessarily unique). Of course we can assume $x\notin \widetilde\sigma(\widetilde M)$. {By cocompactness, $\widetilde\sigma$ is complete, hence it is a proper embedding by Proposition \ref{prop injectivity}. Then} the map that associates to each element of $\widetilde\sigma(\widetilde M)$ its distance from $x$ attains its minimum: this implies that there exists $r>0$ such that the metric sphere of radius $r$ centered at $x$ is tangent to $\widetilde\sigma(\widetilde M)$ at some point $p$. Hence $x$ is on the geodesic through $p$. See Figure \ref{fig:limit2}.
 
Let us now prove that $\rho$ is also convex-cocompact. To show this, we claim that there exists $t_+,t_-\in\R$ such that $\widetilde\sigma_{t_+}$ is a convex embedding, and $\widetilde\sigma_{t_-}$ a concave one.  
Indeed in the proof of Lemma \ref{lemma:evolution fsigma} we showed that the principal curvatures of the normal evolution $\widetilde\sigma_t$ are equal to $\tanh(\mu_i-t)$, where $\mu_i$ is the hyperbolic arctangent of the corresponding principal curvature of $\widetilde\sigma$. Hence taking $t\ll0$ (resp. $t\gg 0$) one can make sure that the principal curvatures of $\widetilde\sigma_t$ are all negative (resp. positive), hence $\widetilde\sigma_t$ is convex (resp. concave). The region bounded by the images of $\widetilde\sigma_{t_+}$ and $\widetilde\sigma_{t_-}$ is then {geodesically} convex and diffeomorphic to $\widetilde M\times [t_-,t_+]$. Under this diffeomorphism, the action of $\pi_1(M)$ corresponds to the action by deck transformations on $\widetilde M$ and the trivial action on the second factor. Hence its  quotient is compact, being diffeomorphic to $M\times [t_-,t_+]$.
\end{proof}

{This implies that in dimension three, nearly-Fuchsian manifolds are quasi-Fuchsian. }

\begin{remark}
There is another important consequence of Proposition \ref{prop:action free prop disc0}.
Given $\widetilde\sigma$ an equivariant immersion as in Definition \ref{defi:nearly fuch rep}, it follows from the cocompactness of the action of $\rho$ on the geodesically convex region $\mathcal C$ that $\widetilde \sigma$ is a quasi-isometric embedding {in the sense of metric spaces}. By cocompactness and Remark \ref{rmk: negative curv}, $\widetilde M$ is a complete simply connected Riemannian manifold of negative sectional curvature, hence its visual boundary $\partial\widetilde M$ in the sense of Gromov is homeomorphic to $S^{n-1}$.
By \cite[Proposition 6.3]{zbMATH01496599}, $\widetilde\sigma$  extends to a continuous injective map $\partial\widetilde\sigma$ from the visual boundary $\partial\widetilde M$ of $\widetilde M$ to $\partial\Hyp^{n+1}$. By compactness of $\partial\widetilde M$,
 the extension of $\widetilde\sigma$  is a homeomorphism onto its image.

Since any two $\rho$-equivariant embeddings $\widetilde\sigma_1,\widetilde\sigma_2:\widetilde M\to\Hyp^{n+1}$ are at bounded distance from  each other by cocompactness, the extension $\partial\widetilde\sigma$ does not depend on $\widetilde\sigma$, but only on the representation $\rho$. In conclusion, the image of $\partial\widetilde\sigma$ is a topological $(n-1)$-sphere $\Lambda_\rho$ in $\partial\Hyp^{n+1}$, called the \emph{limit set} of the representation $\rho$. We remark that the limit set is equivalently defined as the set of accumulation points of the orbit of any point with respect to the action of $\rho$ on $\Hyp^{n+1}$. See Figure \ref{fig:limit}.
\end{remark}

\begin{figure}[htbp]
\centering
\includegraphics[height=4.8cm]{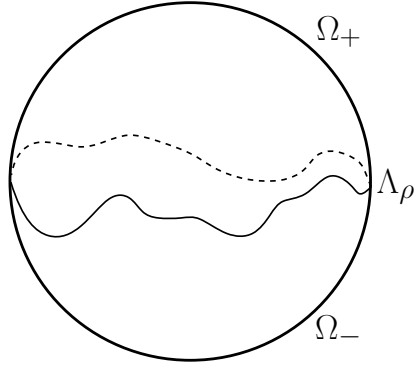} 

\caption{A picture of the limit set $\Lambda_\rho$, which is a topological $(n-1)$-sphere  and disconnects $\partial\Hyp^{n+1}$ in two connected components $\Omega_+$ and $\Omega_-$, which are homeomorphic to $n$-discs.}\label{fig:limit}
\end{figure}

The following proposition is well-known. We provide a proof for the sake of completeness.

\begin{prop} \label{prop:action free prop disc}
Let $M^n$ be a closed orientable manifold, $\rho:\pi_1(M)\to\Isom(\Hyp^{n+1})$ be a nearly-Fuchsian representation and $\Lambda_\rho$ be its limit set. Then {the action of $\rho$ extends to} a free and properly discontinous action on $\partial\Hyp^{n+1}\setminus\Lambda_\rho$, which is the disjoint union of two topological $n$-discs.
\end{prop}
\begin{proof}
{Since the action of $\pi_1(M)$ on $\widetilde M$ is free and properly discontinuous, and $G_{\widetilde\sigma}^\pm$ are diffeomorphisms onto their image by Proposition \ref{prop:gauss maps diffeo onto image}, it follows that the action of $\rho(\pi_1(M))$} is free and properly discontinuous on $G_{\widetilde\sigma}^+(\widetilde M)$ and $G_{\widetilde\sigma}^-(\widetilde M)$, which are topological discs in $\Hyp^{n+1}$ since $\widetilde M$ is diffeomorphic to $\R^n$. We claim that 
{$$G_{\widetilde\sigma}^+(\widetilde M)\cup G_{\widetilde\sigma}^-(\widetilde M)=\partial\Hyp^{n+1}\setminus\Lambda_\rho~.$$} 
Observe that, by the Jordan-Brouwer separation Theorem, the complement of $\Lambda_\rho$ has two connected components, hence the claim will also imply that $G_{\widetilde\sigma}^+(\widetilde M)$ and $G_{\widetilde\sigma}^-(\widetilde M)$ are disjoint because they are both connected.

In order to show that $\partial\Hyp^{n+1}\setminus\Lambda_\rho\subseteq G_{\widetilde\sigma}^+(\widetilde M)\cup G_{\widetilde\sigma}^-(\widetilde 
	M)$, one can repeat the same argument as Proposition \ref{prop:action free prop disc0}, now using horospheres, to see that every $x$ in the complement of $\Lambda_\rho$ is the endpoint of some geodesic orthogonal to ${\widetilde\sigma}(\widetilde M)$. See Figure \ref{fig:limit2}.

{It only remains to show the other inclusion. By continuity, it suffices to show that every $x\in\Lambda_\rho$ is not on the image of $G_{\widetilde\sigma}^\pm$. Observe that by cocompactness the principal curvatures of $\widetilde\sigma$ are bounded by some constant $\epsilon<1$ in absolute value. Now,} if $x\in \partial \Hyp^{n+1}$ is the endpoint of an orthogonal line $\ell$, then, for all $r$, one would be able to construct a $r$-cap tangent to $\ell \cap \widetilde \sigma(\widetilde M)$ such that $x$ lies in the convex side of the $r$-cap: since, by Remark \ref{rmk:convex side caps}, for some $r$ the image of $\widetilde \sigma$ lies in the concave side of the $r$-cap, $x$ cannot lie in $\partial \widetilde \sigma(\widetilde M)= \Lambda_\rho$.
See Figure \ref{fig:limit2} again.
\end{proof}

\begin{figure}[htbp]
\centering
\includegraphics[height=4.8cm]{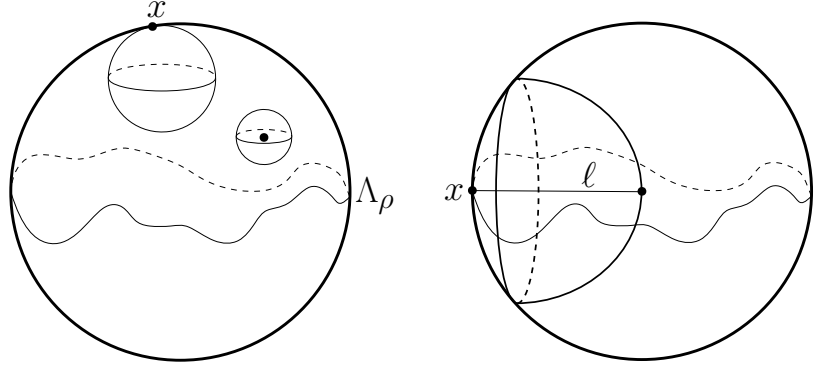} 

\caption{The arguments in the proof of Proposition \ref{prop:action free prop disc}. On the left, since $\widetilde\sigma$ is proper and extends to $\Lambda_\rho$ in $\partial\Hyp^{n+1}$, from every point $x\notin\Lambda_\rho$ one can find a horosphere with point at infinity $x$ tangent to the image of $\widetilde\sigma$. The same argument works for an interior point $x$, using metric balls, which is the argument of Proposition \ref{prop:action free prop disc0}. On the right, a $r$-cap orthogonal to a geodesic $\ell$ with endpoint $x$. Since $\widetilde\sigma$ lies on the concave side of tangent $r$-caps for large $r$, $x$ cannot be in the image of $G_{\widetilde\sigma}^\pm$. The same argument is used in Lemma \ref{lemma:convex gauss map diffeo}, under the convexity assumption, in which case it suffices to use tangent hyperplanes instead of $r$-caps.}\label{fig:limit2}
\end{figure}

As a consequence of Proposition \ref{prop:action free prop disc0}, if $\rho:\pi_1(M)\to \Isom(\Hyp^{n+1})$ is a nearly-Fuchsian representation, then the quotient $\faktor{\Hyp^{n+1}}{\rho(\pi_1(M))}$ is a complete hyperbolic manifold {diffeomorphic to $M\times\R$}. This motivates the following definition.

\begin{defi}
A hyperbolic manifold of dimension $n+1$ is \emph{nearly-Fuchsian}  if it is isometric to the quotient $\faktor{\Hyp^{n+1}}{\rho(\pi_1(M))}$, for $M$ a closed orientable $n$-manifold and $\rho:\pi_1(M)\to\Isom(\Hyp^{n+1})$ a nearly-Fuchsian representation.  
\end{defi}

\begin{remark}\label{rmk:embedding in the quotient}

If $\widetilde\sigma:\widetilde M\to\Hyp^{n+1}$ is a $\rho$-equivariant embedding with small principal curvatures, then $\widetilde \sigma$ descends to the quotient defining a smooth injective map $\sigma\colon M \to  \faktor{\Hyp^{n+1}}{\rho(\pi_1(M))}$. Moreover, since $\widetilde\sigma$ is a $\rho$-equivariant homeomorphism with its image, $\sigma$ is a homeomorphism with its image as well {hence its image is an embedded hypersurface}.
\end{remark}

We conclude this section with a final definition which appears in the statement of Theorem \ref{thm:second char ham}. As a preliminary remark, recall from Propositions \ref{prop:gauss maps diffeo onto image} and \ref{prop:action free prop disc}  that if $\widetilde\sigma$ is a $\rho$-equivariant embedding with small principal curvatures, then each of the Gauss maps $G_{\widetilde\sigma}^\pm$ of ${\widetilde\sigma}$ is a diffeomorphism between $\widetilde M$ and a connected component of $\partial\Hyp^{n+1}\setminus \Lambda_\rho$.  Let us denote these connected components by $\Omega_\pm$ {as in Figure \ref{fig:limit}}, so that: 
$$\partial\Hyp^{n+1}\setminus \Lambda_\rho=\Omega_+\sqcup \Omega_-\qquad G_{\widetilde\sigma}^+(\widetilde M)=\Omega_+\qquad G_{\widetilde\sigma}^-(\widetilde M)=\Omega_-~.$$
{with the representation $\rho$ inducing an action of $\pi_1(M)$ on both $\Omega_+$ and $\Omega_-$.  Recalling the identification 
$$\G{n+1}\cong \partial\Hyp^{n+1}\times\partial\Hyp^{n+1}\setminus \Delta~,$$
given by 
$$\gamma\mapsto\left(\lim_{t\to+\infty}\gamma(t),\lim_{t\to-\infty}\gamma(t)\right)~,$$
{the following definition is well-posed}.

\begin{defi}\label{defi quotient Grho}
Given a closed oriented $n$-manifold $M$ and a nearly-Fuchsian representation $\rho:\pi_1(M)\to\Isom(\Hyp^{n+1})$, we define
$\mathcal G_\rho$ the quotient:
$$\faktor{\{\gamma\in \G{n+1}\,|\,\lim_{t\to+\infty}\gamma(t)\in\Omega_+\text{ or }\lim_{t\to-\infty}\gamma(t)\in\Omega_-\}}{\rho(\pi_1(M))}~.$$
\end{defi}

Observe that, since the  action of $\rho(\pi_1(M))$ on $\partial\Hyp^{n+1}$ is free and properly discontinuous on both $\Omega_+$ and $\Omega_-$, it is also free and properly discontinuous on the region of $\G{n+1}$ consisting of geodesics having either final point in $\Omega_+$ or initial point in $\Omega_-$. Moreover, such region is simply connected, because it is the union of $\Omega_+\times \partial\Hyp^{n+1}\setminus \Delta$ and $\partial\Hyp^{n+1}\times\Omega_-\setminus \Delta$, both of which are simply connected (they retract on $\Omega_+\times\{\star\}$ and $\{\star\}\times\Omega_-$ respectively) and whose intersection $\Omega_+\times\Omega_-$ is again simply connected. 
We conclude that $\mathcal G_\rho$ is a $2n$-manifold {whose fundamental group is isomorphic to $\pi_1(M)$}, and is endowed with a natural para-K\"ahler structure induced from that of  $\G{n+1}$ (which is preserved by the action of $\Isom(\Hyp^{n+1})$). 

It is important to stress once more that $\Omega_+$ and $\Omega_-$ only depend on $\rho$ and not on ${\widetilde\sigma}$. We made here a choice in the labelling of $\Omega_+$ and $\Omega_-$ which only depends on the orientation of $M$. The other choice of labelling would result in a ``twin'' region, which differs from $\mathcal G_\rho$ by switching the roles of initial and final endpoints.

 A consequence of this construction, which is implicit in the statement of Theorem \ref{thm:second char ham}, is the following:

\begin{cor} \label{cor:embedding in Grho}
Let $M^n$ be a closed orientable manifold, $\rho:\pi_1(M)\to\Isom(\Hyp^{n+1})$ be a nearly-Fuchsian representation and ${\widetilde\sigma}:\widetilde M\to\Hyp^{n+1}$ be a $\rho$-equivariant embedding  of small principal curvatures. For a suitable choice of an orientation on $M$,  the Gauss map of $\widetilde \sigma$ takes values in $\Omega_+\times\Omega_-$, and induces an embedding of $M$ in $\mathcal G_\rho$ 
\end{cor}
\begin{proof}
The only part of the statement which is left to prove is that the induced immersion of $M$ in $\mathcal G_\rho$ is an embedding, but the proof follows with the same argument as in Remark \ref{rmk:embedding in the quotient}. 
 \end{proof}

\section{Local and global integrability of immersions into $\G{n+1}$}
\label{sec:local int} 

In this section we introduce a connection on the principal $\R$-bundle $\mathrm{p}$ and relate its curvature with the symplectic geometry of the space of geodesics. As a consequence, we characterize, in terms of the Lagrangian condition, the immersions in the space of geodesics which can be locally seen as Gauss maps of immersions into $\Hyp^n$.

\subsection{Connection on the $\R$-principal bundle}
Recall that a \emph{connection} on a principal $G$-bundle $P$ is a $\mathfrak g$-valued 1-form $\omega$ on $P$ such that:
\begin{itemize}
\item $\mathrm{Ad}_g(R_g^*\omega)=\omega$;
\item for every $\xi\in G$, $\omega(X_\xi)=\xi$ where $X_\xi$ is the vector field associated to $\xi$ by differentiating the action of $G$.
\end{itemize}
In the special case $G=\R$ which we consider in this paper, {$\omega$ is a real-valued 1-form and }the first property simply means that $\omega$ is invariant under the $\R$-action.

Let us now introduce the connection that we will concretely use.
\begin{defi}\label{defi connection form}
We define the connection form on $\mathrm{p}:T^1\Hyp^{n}\to \G{n}$ as 
$$\omega(X)=\gs{n}(X,\chi_{(x,v)})$$
for $X\in T_{(x,v)}T^1\Hyp^{n}$, where $\chi$ is the infinitesimal generator of the geodesic flow.
\end{defi}

The 1-form $\omega$ indeed satisfies the two properties of a connection: the invariance  under the $\R$-action follows immediately from the invariance of $\gs{n}$ (Lemma \ref{lemma:geodflow isometric}) and of $\chi$ (Equation \eqref{eq:diff geoflow3}); the second property follows from Equation \eqref{eq:generator geoflow is unit}, namely $\omega(\chi_{(x,v)})=\gs{n}(\chi_{(x,v)},\chi_{(x,v)})=1$.

The connection $\omega$ is defined in such a way that the associated \emph{Ehresmann connection}, which we recall being a distribution of $T^1\Hyp^n$ in direct sum with the tangent space of the fibers of $\mathrm{p}$, is simply the distribution orthogonal to $\chi$. Indeed the Ehresmann connection associated to $\omega$ is the kernel of $\omega$. The subspaces in the Ehresmann connections defined by the kernel of $\omega$ are usually called \emph{horizontal}; we will avoid this term here since it might be confused with the horizontal distribution $\HH$ with respect to the other bundle structure of $T^1\Hyp^n$, namely the unit tangent bundle $\pi:T^1\Hyp^n\to\Hyp^n$. 

Now, a connection on a principal $G$-bundle is \emph{flat} if the Ehresmann distribution is integrable, namely, every point admits a local section tangent to the kernel of $\omega$. We will simply refer to such a section as a \emph{flat} section. The bundle is a \emph{trivial flat} principal $G$-bundle if it has a global flat section.

Having introduced the necessary language, the following statement is a direct reformulation of Proposition \ref{prop:gauss immersion}.

\begin{prop}\label{prop pullback flat}
Given an oriented manifold $M^n$ and an immersion $\sigma:M\to\Hyp^{n+1}$, the $G_\sigma$-pull-back of $\mathrm{p}:T^1\Hyp^{n+1}\to\G{n+1}$ is a trivial flat bundle.
 \end{prop}
\begin{proof}
The lift $\zeta_\sigma:M\to T^1\Hyp^{n+1}$ is orthogonal to $\chi$ by Proposition \ref{prop:gauss immersion} and therefore induces a global flat section of the pull-back bundle {via $G_\sigma= \mathrm p \circ \zeta_\sigma$}.
\end{proof}

\subsection{Curvature of the connection}
The purpose of this section is to compute the curvature of the connection $\omega$, which is simply the $\R$-valued 2-form $d\omega$. (The general formula for the curvature of a connection on a principal $G$-bundle is $d\omega+ \frac1 2 \omega\wedge\omega$, but the last term vanishes in our case $G=\R$.)

\begin{remark}\label{rmk:curvature and bracket}
It is known that the curvature of $\omega$ is the obstruction to the existence of local flat sections. In particular in the next proposition we will use extensively that, given $X,Y\in \chi_{(x,v)}^\perp\subset T_{(x,v)}T^1\Hyp^{n+1}$, if there exists an embedding {$ \zeta:M\to T^1\Hyp^{n+1}$ such that $\zeta(p)=(x,v)$, that $X,Y\in d\zeta(T_p M)$ and that $d\zeta(T_pM ) \subset \chi^\perp$,} then $d\omega(X,Y)=0$.

This can be easily seen by a direct argument: if we now denote by $X$ and $Y$ some extensions tangential to the image of $\zeta$, one has 
$$
d\omega(X,Y)=\partial_X(\omega(Y))-\partial_Y(\omega(X))-\omega([X,Y])=0~,
$$
since $\omega(X)=\omega(Y)=0$ by the hypothesis that $X$ and $Y$ are orthogonal to $\chi$, whence $\omega(X)=\gs{n}(X,\chi)=0$, and moreover $\omega([X,Y])=0$ since $[X,Y]$ remains tangential to the image of $\sigma$. 

The argument can in fact be reversed to see that $d\omega$ is exactly the obstruction to the existence of a flat section, by the Frobenius theorem. 
\end{remark}

The following proposition represents an essential step to relate the curvature of $\omega$ and the symplectic geometry of the space of geodesics.

\begin{prop}\label{prop:identity curvature form}
The following identity holds for the connection form $\omega$ on the principal $\R$-bundle $\mathrm{p}:T^1\Hyp^{n+1}\to\G{n+1}$ and the symplectic form $\Omega$ of $\G{n+1}$:
$$d\omega=\mathrm{p}^*\Omega~.$$

\end{prop}
\begin{proof}
We shall divide the proof in several steps. 

First, let us show that $d\omega$ is the pull-back of some 2-form on $\G{n+1}$. Since $d\omega$ is obviously invariant under the geodesic flow, we only need to show that $d\omega(X,\chi_{(x,v)})=0$ for all $X\in T_{(x,v)}T^1\Hyp^{n+1}$. Clearly, it suffices to check this for $X\in \chi^\perp$. To apply the exterior derivative formula for $d\omega$, we consider $\chi$ as a globally defined vector field and we shall extend $X$ {around} $(x,v)$. For this purpose, take a curve $c:(-\epsilon,\epsilon)\to T^1\Hyp^{n+1}$ such that $c'(0)=X$ and $c'(s)$ is tangent to $\chi^\perp$ for every $s$. Then define the map $f(s,t)=\varphi_t(c(s))$ and observe that $\chi=\partial_t f$. We thus set $X=\partial_s f$, which is the desired extension {along a two-dimensional submanifold}. Then we have:
$$d\omega(X,\chi)=\partial_X(\omega(\chi))-\partial_\chi(\omega(X))-\omega[X,\chi]=0~,$$
where we have used that $\omega(\chi)\equiv 1$, that $\omega(X)=0$ along the curves $t\mapsto \varphi_t(x,v)$ (since the curve $c$ is taken to be in the distribution $\chi^\perp$ and the geodesic flow preserves both $\chi$ and its orthogonal complement), and finally that $[X,\chi]=0$ since {$\chi=\partial_t f$ and $X=\partial_s f$ are coordinate vector fields} for a submanifold in neighborhood of $(x,v)$. 

Having proved this, it is now sufficient to show that $d\omega$ and $\mathrm{p}^*\Omega$ agree when restricted to $\chi^\perp$. Recall that $\Omega$ is defined as $\GG(\cdot,\JJ\cdot)$, where $\GG$ and $\JJ$ are the push-forward to $T_{\mathrm{p}(x,v)}\G{n+1}$, by means of the differential of $\mathrm{p}$, of the metric $\gs{n+1}$ and of the para-complex structure $J$ on $\chi^\perp$. Thus we must equivalently  show that 
\begin{equation}\label{eq:sufficient}
d\omega(X,Y)=\gs{n+1}(X,JY)\qquad\text{for all }X,Y\in \chi_{(x,v)}^\perp~.
\end{equation}

To see this, take an orthonormal basis $\{w_1,\ldots,w_n\}$ for $v^\perp\subset T_x\Hyp^{n+1}$, and observe that $\{w_1^\HH,\ldots,w_n^\HH,w_1^\V,\ldots,w_n^\V\}$ is a $\gs{n+1}$-orthonormal basis of $\chi^\perp$. It is sufficient to check \eqref{eq:sufficient} for $X,Y$ {distinct elements of} this basis. We distinguish several cases. 

\begin{itemize}
\item First, consider the case $X=w_i^\HH$ and $Y=w_j^\HH$, for $i\neq j$. Then $\gs{n+1}(X,JY)=\gs{n+1}(w_i^\HH,w_j^\V)=0$ by Definition \ref{defi:parasasaki}. On the other hand, by Example \ref{ex:totally geodesic}, if $\sigma$ is the inclusion of the totally geodesic hyperplane orthogonal to $v$ at $x$, then its lift $\zeta_\sigma$ is a submanifold in $T^1\Hyp^{n+1}$ orthogonal to $\chi$ at every point and tangent to $X$ and $Y$. Then $d\omega(X,Y)=0$ by Remark \ref{rmk:curvature and bracket}.

\item Second, consider $X=w_i^\V$ and $Y=w_j^\V$, for $i\neq j$. Then again $\gs{n+1}(X,JY)=\gs{n+1}(w_i^\V,w_j^\HH)=0$. Here we can apply Example \ref{ex:spheres} instead, and see that there is a $n$-dimensional sphere in $\VP_{(x,v)}$ orthogonal to the fibers of $\mathrm{p}$ and tangent to $X$ and $Y$, whence $d\omega(X,Y)=0$ by Remark \ref{rmk:curvature and bracket}.

\item Third, consider $X=w_i^\HH$ and $Y=w_j^\V$, for $i\neq j$. Then $\gs{n+1}(X,JY)=\gs{n+1}(w_i^\HH,w_j^\HH)=0$ since $w_i$ and $w_j$ are orthogonal. Let us now apply Example \ref{ex:mixed}, for instance by taking $Q$ the geodesic going through $x$ with speed $w_i$. The normal bundle $\mathrm N^1 Q$ is a submanifold orthogonal to the fibers of $\mathrm{p}$ and tangent to $w_i^\HH$ and $w_j^\V$. So $d\omega(X,Y)=0$ by the usual argument.

\item Finally, we have to deal with the case $X=w_i^\HH$ and $Y=w_i^\V$. Here $\gs{n+1}(X,JY)=\gs{n+1}(w_i^\HH,w_i^\HH)=1$. For this computation, we may assume $n=1$, up to restricting to the totally geodesic 2-plane spanned by $v$ and $w$, which is a copy of $\Hyp^2$. Hence we will simply denote $w_i=w$, and moreover we can assume (up to changing the sign) that $w=x\boxtimes v$, where $\boxtimes$ denotes the Lorentzian cross product in $\R^{2,1}$. In other words, $(x,v,w)$ forms an oriented orthonormal triple in $\R^{2,1}$. 

Now, let us extend $X$ and $Y$ to globally defined vector fields on $T^1\Hyp^2$, by means of the assignment $(x,v)\mapsto (x\boxtimes v)^\HH$ and $(x,v)\mapsto (x\boxtimes v)^\V$. By this definition, $X$ and $Y$ are orthogonal to $\chi$; we claim that $[X,Y]=-\chi$. This will conclude the proof, since
\begin{equation*}
d\omega(X,Y)=\partial_X(\omega(Y))-\partial_Y(\omega(X))-\omega[X,Y]=-\omega[X,Y]=\gs{2}(\chi,\chi)=1~.
\end{equation*}
For the claim about the Lie bracket, let us use the hyperboloid model. Then $X=(x\boxtimes v,0)$ and $Y=(0,x\boxtimes v)$. We can consider $X$ and $Y$ as globally defined (by the same expressions) in the ambient space $\R^{2,1}\times\R^{2,1}$, so as to compute the Lie bracket in $\R^{2,1}\times\R^{2,1}$, which will remain tangential to $T^1\Hyp^2$ since $T^1\Hyp^2$ is a submanifold. Using the formula $[X,Y]=Jac_XY-Jac_YX$, where $Jac_X$ denotes the Jacobian of the vector field thought as a map from $\R^3$ to $\R^3$, we obtain
$$[X,Y]=(x\boxtimes(x\boxtimes v),0)-(0,(x\boxtimes v)\boxtimes v)=-(v,x)=-\chi_{(x,v)}$$
by the standard properties of the cross-product.
\end{itemize}
In summary, we have shown that $d\omega$ and $\mathrm{p}^*\Omega$ coincide on the basis $\{\chi,w_1^\HH,\ldots,w_n^\HH,w_1^\V,\ldots,w_n^\V\}$ of $T_{(x,v)}T^1\Hyp^{n+1}$, and therefore the desired identity holds.
\end{proof}

{We get as an immediate consequence the closedness of the fundamental form $\Omega$, a fact whose proof has been deferred from Section \ref{sec:parakahler metric GG}.
\begin{cor}\label{cor:omega closed}
The fundamental form $\Omega=\GG(\cdot,\JJ\cdot)$ is closed.
\end{cor}
\begin{proof}
Using Proposition \ref{prop:identity curvature form} we have
$$\mathrm{p}^*(d\Omega)=d(\mathrm{p}^*\Omega)=d(d\omega)=0~.$$
Since $\mathrm{p}$ is surjective, it follows that $d\Omega=0$.
\end{proof}}

\subsection{{Lagrangian immersions}} We have now all the ingredients to relate the Gauss maps of immersed hypersurfaces in $\Hyp^{n+1}$ with the Lagrangian condition for the symplectic geometry of $\G{n+1}$. 

\begin{cor}\label{cor:lagrangian1}
Given an oriented manifold $M^n$ and an immersion $\sigma:M\to\Hyp^{n+1}$, $\nobreak{G_\sigma:M\to \G{n+1}}$ is a Lagrangian immersion. 
\end{cor}
\begin{proof}
By Proposition \ref{prop pullback flat}, the pull-back by $G_\sigma$ of the principal $\R$-bundle
$\mathrm{p}$ is flat {because} there exists {$\widehat G_\sigma=\zeta_\sigma:M\to T^1\Hyp^{n+1}$ orthogonal to $\chi$ such that $G_\sigma=\mathrm p\circ\widehat G_\sigma$. Hence $(\widehat G_\sigma)^*d\omega$ vanishes identically and $\widehat G_\sigma$ induces a flat section of $G_\sigma^*\mathrm p$.}  But by Proposition \ref{prop:identity curvature form}, $(\widehat G_\sigma)^*d\omega=(\widehat G_\sigma)^*(\mathrm{p}^*\Omega)=(G_\sigma)^*\Omega$, therefore $G_\sigma$ is Lagrangian.
\end{proof}

Observe that in Corollary \ref{cor:lagrangian1} we only use the flatness property of Proposition \ref{prop pullback flat}, and not the triviality of the  pull-back principal bundle. When $M$ is simply connected, we can partially reverse Corollary \ref{cor:lagrangian1}, showing that the Lagrangian condition is essentially the only local obstruction.

\begin{cor}\label{cor:lagrangian2}
Given an {orientable} simply connected manifold $M^n$ and a Lagrangian immersion $G:M\to \G{n+1}$, there exists an immersion $\zeta:M\to T^1\Hyp^{n+1}$ orthogonal to the fibers of $\mathrm{p}$ such that $G=\mathrm{p}\circ \zeta$. Moreover, if $\pi\circ \zeta:M\to \Hyp^{n+1}$ is an immersion, then $G$ coincides with its Gauss map.
\end{cor}
\begin{proof}
Since $G$ is Lagrangian, by Proposition \ref{prop pullback flat} the $G$-pull-back bundle of $\mathrm{p}$ is flat, and it is therefore a trivial flat bundle since $M$ is simply connected. Hence it admits a global flat section, which provides the map $\zeta:M\to T^1\Hyp^{n+1}$ orthogonal to $\chi$. Assuming moreover that $\sigma:=\pi\circ \zeta$ is an immersion, by Proposition \ref{prop:gauss immersion converse}, $G=\mathrm{p}\circ \zeta$ is the Gauss map of $\sigma$.
\end{proof}

Clearly the map $\zeta$ is not uniquely determined, and the different choices differ by the action of $\varphi_t$. {Lemma \ref{lemma:desingularize for small t}, and the following corollary, show that (by post-composing with $\varphi_t$ if necessary) one can always find $\zeta$ such that $\pi\circ\zeta$ is \emph{locally} an immersion.

\begin{lemma} \label{lemma:desingularize for small t}
Let $M$ be a $n$-manifold and $\zeta:M\to T^1\Hyp^{n+1}$ be an immersion orthogonal to $\chi$.
Suppose that the differential of $\pi\circ\zeta$ is singular at $p\in M$. Then there exists $\epsilon>0$ such that the differential of $\pi\circ\varphi_t\circ\zeta$ at $p$ is non-singular for all $t\in(-\epsilon,\epsilon)\setminus\{0\}$.
\end{lemma}
\begin{proof}
Let us denote $\zeta_t:=\varphi_t\circ\zeta$, $\sigma:=\pi\circ\zeta$, and $\sigma_t:=\pi\circ\zeta_t$. Assume also $\zeta(p)=(x,v)$. Let $\{V_1,\ldots, V_k\}$   be a basis of the kernel of $d_p\sigma$ and let us complete it to a basis $\{V_1,\ldots,V_n\}$ of $T_p M$. Hence if we denote $w_j:=d_p\sigma(V_j)$ for $j>k$, $\{w_{k+1},\ldots,w_n\}$ is a basis of the image of $d_p\sigma$. Exactly as in the proof of Proposition \ref{prop:gauss immersion converse}, we have $w_{k+1},\ldots,w_n\in v^\perp\subset T_x\Hyp^{n+1}$. Hence we have:
$$d\zeta(V_1)=u_1^\V,\ldots,d\zeta(V_k)=u_k^\V,d\zeta(V_{k+1})=w_{k+1}^\HH+u_{k+1}^\V,\ldots,d\zeta(V_n)=w_n^\HH+u_n^\V$$
for some $u_1,\ldots,u_n\in v^\perp$. 

On the one hand, since $\zeta$ is an immersion, $u_1,\ldots,u_k$ are linearly independent. On the other hand, $\zeta$ is orthogonal to $\chi$, hence by Remark \ref{rmk:curvature and bracket} we have $\zeta^*d\omega=0$. Using Equation \eqref{eq:sufficient}, it follows that:
$$g_{T^1\Hyp^{n+1}}(d\zeta(V_i),J\circ d\zeta(V_j))=0$$
for all {$i,j=1,\dots n$}. Applying this to any choice of {$i\leq k$ and $j>k$}, we find $\langle u_i,w_j\rangle=0$. Hence $\{u_1,\ldots,u_k,w_{k+1},\ldots,w_n\}$ is a basis of $v^\perp$. 

We are now ready to prove the statement. By Equations \eqref{eq:diff geoflow1} and \eqref{eq:diff geoflow2}, we have
$$d\sigma_t(V_i)=d\pi\circ d\varphi_t(u_i^\V)=\sinh(t)u_i$$
for $1\leq i\leq k$, while
$$d\sigma_t(V_j)=d\pi\circ d\varphi_t(w_j^\HH+u_j^\V)=\cosh(t)w_j+\sinh(t)u_j$$
for $k+1\leq j\leq n$. The proof will be over if we show that $\{d\sigma(V_1),\ldots,d\sigma(V_n)\}$ are linearly independent for $t\in(-\epsilon,\epsilon)$, $t\neq 0$. In light of the above expressions, dividing by $\sinh(t)$ (which is not zero if $t\neq 0$) or $\cosh(t)$, this is equivalent to showing that 
$$\{u_1,\ldots,u_k,w_{k+1}+\tanh(t)u_{k+1},\ldots,w_{n}+\tanh(t)u_{n}\}$$
are linearly independent for small $t$. This is true because we have proved above that $\{u_1,\ldots,u_k,w_{k+1},\ldots,w_n\}$ is a basis, and linear independence is an open condition.
\end{proof}

\begin{thmx}\label{cor: local integrability}
{Let $G\colon M^n \to \G{n+1}$ be an immersion. Then $G$ is Lagrangian if and only if for all $p\in M$ there exists a neighbourhood $U$ {of $p$} and an immersion $\sigma\colon U\to \Hyp^{n+1}$ such that $G_\sigma = G|_{U}$.}
\end{thmx}
\begin{proof}
{The ``if'' part follows from Corollary \ref{cor:lagrangian1}. Conversely,} let $U$ be a simply connected neighbourhood of $p$. By Corollary \ref{cor:lagrangian2}, there exists an immersion $\zeta\colon U \to T^1\Hyp^{n+1}$ orthogonal to the fibers of $\mathrm{p}$ such that $G=\mathrm{p}\circ \zeta$. If the differential of $\pi\circ\zeta$ is non-singular at $p$, then, up to restricting $U$, we can assume $\sigma:=\pi\circ\zeta$ is an immersion of $U$ into $\Hyp^{n+1}$. By the second part of Corollary \ref{cor:lagrangian2}, $G|_U$ is the Gauss map of $\sigma$. If the differential of $\pi\circ\zeta$ is instead singular at $p$, by Lemma \ref{lemma:desingularize for small t} it suffices to replace $\zeta$ by $\zeta_t$ for small $t$ and we obtain the same conclusion.
\end{proof}

Let us now approach the problem of global integrability. We provide an example to show that in general {$\pi\circ\varphi_t\circ\zeta$ might fail to be \emph{globally} an immersion for \emph{all} $t\in\R$}, as we already mentioned after Proposition \ref{prop:gauss immersion converse}. 

\begin{example}
	\label{ex: Lagrangian not globally integrable}

	Let us construct a curve {$G:(-T,T)\to\G 2$} with the property of being locally integrable\footnote{As a matter of fact, any curve in $\G 2$ is locally integrable by Theorem \ref{cor: local integrability}, since the domain is simply connected and it is trivially Lagrangian. However, in this example, we will see by construction that $G$ is locally integrable.} but not globally integrable. 
	
	Fix $r>0$ {and a maximal geodesic $\ell$ in $\Hyp^2$}.
	{Let us consider a smooth curve $\sigma_+\colon (-\varepsilon,T)\to \Hyp^2$, for some small enough $\varepsilon$ and big enough $T$, so that: \begin{itemize}
		\item {$\sigma_+$} is an immersion and is parameterized by arclength;
		\item $(\sigma_+)|_{(-\varepsilon, \varepsilon)}$ lies on the $r$-cap equidistant from $\ell$, oriented in such a way that the induced unit normal vector field $(\nu_+)|_{(-\varepsilon, \varepsilon)}$ is directed towards $\ell$;
		\item $(\sigma_+)|_{(T_0, T)}$ lies on the metric circle $\{x\in\Hyp^2\,|\,d_{\Hyp^2}(x,x_0)=r\}$ for some $x_0\in\Hyp^2$ and some $\varepsilon<T_0<T$, oriented in such a way that the induced unit normal vector field $(\nu_+)|_{(T_0, T)}$ is directed towards $x_0$.
	\end{itemize}
	} 
	
	\begin{figure}[htbp]
\centering
\includegraphics[height=5.5cm]{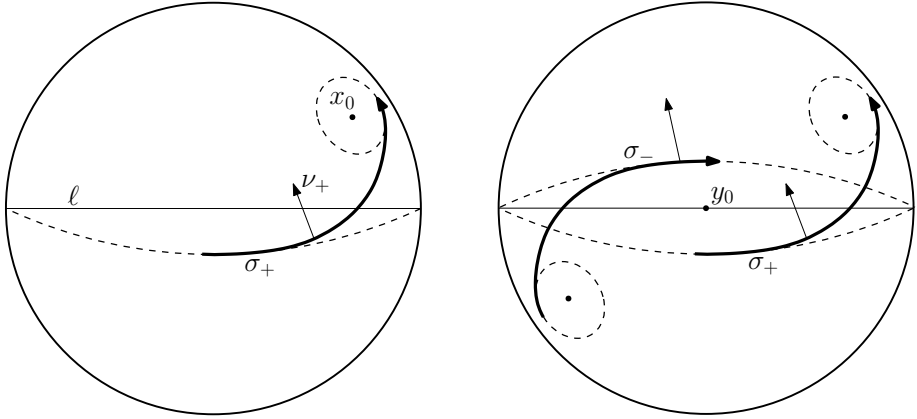} 

\caption{On the left, the construction of the curve $\sigma_+$, which is an arclenght parameterization of a portion of $r$-cap equidistant from $\ell$ for $s\in(-\varepsilon,\varepsilon)$, and of a circle of radius $r$ for $s\in (T_0,T)$. On the right, the curve $\sigma_-$ whose image is obtained by an order-two elliptic isometry from the image of $\sigma_+$.}\label{fig:counterexample1}
\end{figure}

See Figure \ref{fig:counterexample1}. 	By a simple computation, for instance using Equation \eqref{eq della discordia}, the curvature of $\sigma_+$ equals $\tanh(r)$ on $(-\varepsilon, \varepsilon)$ and $\frac 1 {\tanh(r)}$ on $(T_0, T)$. Hence by the intermediate value theorem, {the image of the curvature function $k\colon (-\varepsilon, T)\to \R$ contains the interval $[\tanh(r),\frac 1 {\tanh(r)}]$.}   An important consequence of this observation is that $(\sigma_+)_t$ fails to be an immersion when $t\geq r$. More precisely, if we denote $\zeta_{\sigma_+}$ the lift to $T^1\Hyp^2$ as usual, using Equation \eqref{eq:geodflow} analogously as for Equation \eqref{eq:normal evolution hyperboloid2}, we obtain that:
	$$d(\pi\circ\varphi_t\circ\zeta_{\sigma_+})(V)=(\cosh(t)-\sinh(t)k)d\sigma(V)~.$$
This shows that the differential of $\pi\circ\varphi_t\circ\zeta_{\sigma_+}$ at a point $s\in (-\varepsilon,T)$ vanishes if and only if 
\begin{equation}\label{eq:curvature and differential}
\tanh(t)=\frac 1 {k(s)}~.
\end{equation}
 Since the image of the function $k$ contains the interval $[\tanh(r),\frac 1 {\tanh(r)}]$, if $t\geq r$ then there exists $s$ such that Equation \eqref{eq:curvature and differential} is satisfied and therefore $\pi\circ\varphi_t\circ\zeta_{\sigma_+}$ is not an immersion at $s$.

Now, let $y_0\in \ell$ be the point at distance $r$ from $\sigma_+(0)$ and let ${R}_0\colon \Hyp^2\to \Hyp^2$ be the symmetry at $y_0$, i.e. ${R}_0$ is the isometry of $\Hyp^2$ such that ${R}_0(y_0)=y_0$ and $d_{y_0} {R}_0=-\mathrm{id}$. Define $\sigma_-\colon (-T, \varepsilon)\to \Hyp^2$ as 
	\[
	\sigma_-(s):=  {R}_0 (\sigma_+ (-s)).
	\] 
	As a result, the normal vector field $\nu_-$ of $\sigma_-$  is such that $d R_0 ({\nu_+} (s))= {-}{\nu_-} (-s)$ and the curvature of $\sigma_-$ takes any value in the interval $[-\frac 1 {\tanh(r)}, -{\tanh(r)}]$. Hence $(\sigma_-)_t$ fails to be an immersion for $t\leq -r$.
	
	Finally, consider the two lifts $\zeta_{\sigma_+}$ and $\zeta_{\sigma_-}$ in $T^1\Hyp^2$. By construction one has that for all $s\in (-\varepsilon,\varepsilon)$
	\[
	\varphi_r \circ \zeta_{\sigma_+} (s)= \varphi_{-r} \circ \zeta_{\sigma_-} (s).
	\]
	As a result, we can define our counterexample $\zeta\colon (-T, T)\to T^1\Hyp^2$ as
	\[
	\zeta(s)
	= \begin{cases}
	\varphi_r \circ \zeta_{\sigma_+} (s)\quad &\text{ if $s>-\varepsilon$}\\
	\varphi_{-r} \circ \zeta_{\sigma_-} (s) \quad &\text{ if $s <\varepsilon$}
	\end{cases}.\]
	By construction, we have that $\mathrm p \circ \zeta_{\sigma_+}= \mathrm p \circ \zeta|_{(-\varepsilon, T)}$ and $\mathrm p \circ \zeta_{\sigma_-}= \mathrm p \circ \zeta|_{(-T, \varepsilon)}$, therefore $\mathrm p \circ \zeta$ is an immersion into $\G{2}$ and clearly it is locally integrable. 
However, by the above discussion, $\pi\circ\varphi_t\circ\zeta$ fails to be an immersion for every $t\in\R$: for $t\geq 0$ because $\pi\circ\varphi_t\circ\zeta_{\sigma_+}$ has vanishing differential at some  $s\geq -\varepsilon$, and for $t\leq 0$ because the differential of $\pi\circ\varphi_t\circ\zeta_{\sigma_-}$ vanishes  at some  $s\leq \varepsilon$. See Figure \ref{fig:counterexample2}.
	\begin{figure}[htbp]
\centering
\includegraphics[height=5.5cm]{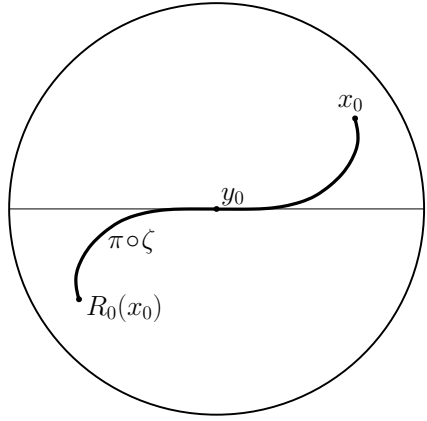} 
\caption{The curve $\zeta:(-T,T)\to T^1\Hyp^2$ projects to a map $\pi\circ\zeta$ into $\Hyp^2$ which is not an immersion, because it is constantly equal to $x_0$ for $s\in(T_0,T)$ and to $R_0(x_0)$ for $s\in(-T,-T_0)$. Moreover the curvature of the regular part takes all the values in $(-\infty,+\infty)$. For this reason, each immersion $\varphi_t\circ\zeta$ into $T^1\Hyp^2$ does not project to an immersion in $\Hyp^2$.}\label{fig:counterexample2}
\end{figure}	
\end{example}

Corollary \ref{cor:lagrangian2} and {Theorem \ref{cor: local integrability}} can be improved under the additional assumption that the immersion $G$ is Riemannian. More precisely, we provide an improved characterization of {immersions into $\G{n+1}$ that are} Gauss maps of immersions with small principal curvatures in terms of the Lagrangian and Riemannian condition, again for {when} $M$ is simply connected.

\begin{thmx}
\label{prop: riemannian global integrability}
Given a simply connected manifold $M^n$ and an {immersion} $G:M^n\to \G{n+1}$, $G$ is Riemannian and Lagrangian if and only if there exists an {immersion} $\sigma\colon M\to \Hyp^{n+1}$ with small principal curvatures such that $G_\sigma=G$. 

{If in addition $\sigma$ is complete, then it is a proper embedding, $G_\sigma$ is an embedding and $M$ is diffeomorphic to $\R^n$.}
\end{thmx}
\begin{proof}
We know from Corollary \ref{cor:lagrangian1} and Proposition \ref{prop: small curv sse riemannian} that the Riemannian and Lagrangian conditions {on $G$} are necessary.
To see that they are also sufficient, by Corollary \ref{cor:lagrangian2} there exists $\zeta\colon M \to \G{n+1}$ orthogonal to the fibers of $\mathrm{p}$ such that $\mathrm{p}\circ \zeta=G$. We claim that $\pi\circ\zeta$ is an immersion, which implies that $G=G_\sigma$ for $\sigma=\pi\circ\zeta$ by the second part of Corollary \ref{cor:lagrangian2}. Indeed, if $X\in T_p M$ is such that $d\zeta (X) \in \mathcal V_{\zeta(x)}=\ker (d\pi_{\zeta(x)}) $, then $d\zeta (X)=w^\V$ for some $w\in T_{\sigma(p)}\Hyp^{n+1}$. Hence by Definition \ref{defi:parasasaki} and the construction of the metric $\GG$,  $\GG (X,X)= -\langle w,w\rangle \le 0$: since $\GG$ is Riemannian, necessarily $w=0$ and therefore $X=0$. 

{By Proposition \ref{prop: small curv sse riemannian} $\sigma$ has small principal curvatures.
The ``in addition'' part follows by Proposition \ref{prop injectivity} and Proposition \ref{prop:gauss maps diffeo onto image}.}
\end{proof}

As another consequence of Proposition \ref{prop injectivity} and Proposition \ref{prop:gauss maps diffeo onto image}, we obtain the following result.

\begin{thmx}\label{cor G complete}
Let $M^n$ be a manifold. If $G\colon M\to \G {n+1}$ is a complete Riemannian and Lagrangian immersion, then $M$ is diffeomorphic to $\R^n$ and there exists a proper embedding $\sigma:M\to\Hyp^{n+1}$ with small principal curvatures such that $G=G_\sigma$.
\end{thmx}
\begin{proof}
Let us lift $G$ to the universal cover $\widetilde M$, obtaining a Riemannian and Lagrangian immersion $\widetilde G:\widetilde M\to\G{n+1}$ which is still complete. By Theorem \ref{prop: riemannian global integrability}, $\widetilde G$ is the Gauss map of an immersion $\sigma$ with small principal curvatures. We claim that $\sigma$ is complete. Indeed by Equation \eqref{eq:fff gauss} the first fundamental form of $\widetilde G$, which is complete by hypothesis, equals $\I-\III$, hence it is complete since $\III$ is positive semi-definite.

It follows from Proposition \ref{prop:gauss maps diffeo onto image} that $\widetilde G$ is injective. Hence $\widetilde M=M$ and $\widetilde G=G$, and therefore $G$ is the Gauss map of $\sigma$, which is complete. 
By the ``in addition'' part of Theorem \ref{prop: riemannian global integrability}, $\sigma$ is properly embedded and $M$ is diffeomorphic to $\R^n$.
\end{proof}

In summary, the Lagrangian condition is essentially the only \emph{local} obstruction to realizing an immersion $G:M\to \G{n+1}$ as the Gauss map  of an immersion into $\Hyp^{n+1}$, up to the subtlety that this might be an immersion only when lifted to $T^1\Hyp^{n+1}$. {This subtlety however never occurs in the small principal curvatures case.} 
{In the remainder of the paper, we will discuss the problem of characterizing immersions into $\G{n+1}$ which are Gauss maps of \emph{equivariant} immersions into $\Hyp^{n+1}$ with small principal curvatures.}

\section{Equivariant integrability: the Maslov class}\label{sec:equivariant integrability}

In this section, we provide the first characterization of \emph{equivariant} immersions in $\G{n+1}$ which arise as the Gauss maps of \emph{equivariant} immersions in $\Hyp^{n+1}$, in the Riemannian case. This is the content of Theorem \ref{teorema hol H baby}. We first try to motivate the problem,  introduce  the obstruction, namely the Maslov class, and study some of its properties. See for instance \cite{zbMATH01523513} for a discussion on the Maslov class in more general settings.

\subsection{Motivating examples}

{Given an $n$-manifold $M$, a representation $\rho:\pi_1(M)\to\Isom(\Hyp^{n+1})$, and a $\rho$-equivariant immersion $\widetilde\sigma:\widetilde M\to \Hyp^{n+1}$, it is immediate to see that the Gauss map $G_{\widetilde\sigma}:\widetilde M\to \G{n+1}$ is $\rho$-equivariant} {(recall also Remark \ref{rmk: Isom preserva Omega e J})}.
Moreover, if ${\widetilde\sigma}$ has small principal curvatures, it follows from the discussion of the previous sections that $G_{\widetilde\sigma}$ is a Lagrangian and Riemannian {immersion}. 

However, a $\rho$-equivariant Lagrangian immersion (even with the additional assumptions of being Riemannian and being an embedding) does not always arise as the Gauss map associated to a \emph{$\rho$-equivariant} immersion in $\Hyp^{n+1}$, as the following simple example shows for $n=1$.

\begin{example}\label{ex: global integrable non equivariant}
Let us {construct} a coordinate system for a portion of $\G{2}$. Let $\gamma:\R\to\Hyp^2$ be a geodesic parameterized by arclength, and let us define a map $\Psi:\R\times (0,\pi)\to\G{2}$ by defining $\Psi(t,\theta)$ as the oriented geodesic that intersects $\gamma$ at $\gamma(t)$ with an angle $\theta$ ({measured counterclockwise} with respect to the standard orientation of $\Hyp^2$). 
We can lift $\Psi$ to a map $\widehat\Psi:\R\times (0,\pi)\to T^1\Hyp^2$, which will however \emph{not} be orthogonal to the fibers of the projection $T^1\Hyp^2\to\G{2}$. The lift is simply defined as 
$$\widehat \Psi(t,\theta)=(\gamma(t),\cos(\theta)\gamma'(t)+\sin(\theta) w)~,$$ where $w$ is the vector {forming an angle $\frac \pi 2$ with $\gamma'(t)$}. We omitted the dependence of $w$ on $t$ since, in the hyperboloid model, $w$ is actually a constant vector in $\R^{2,1}$.  

Let us compute the pull-back of the metric $\GG$ on $\G{2}$ by the map $\Psi$. We have already observed in Example \ref{ex:spheres} that the restriction of $\GG$ on the image of $\theta\mapsto\Psi(t_0,\theta)$ is minus the standard metric of $\Sph^1$. Indeed in this simple case, $d\widehat\Psi_{(t,\theta)}(\partial_\theta)=(0,w)$ is in the vertical subspace 
$\VP$
and by Definition \ref{defi:parasasaki} its squared norm is $-1$. On the other hand, since the vector field $\cos(\theta)\gamma'(t)+\sin(\theta) w$ is parallel along $\gamma$, when we differentiate in $t$ we obtain, by applying the definition of horizontal lift:
\begin{equation}\label{eq:lift example n=1}
d\widehat\Psi_{(t,\theta)}(\partial_t)=\cos(\theta)(\gamma'(t))^\HH+\sin(\theta)w^\HH~.\end{equation}
Moreover, Equation \eqref{eq:lift example n=1} gives the decomposition of $d\widehat\Psi_{(t,\theta)}(\partial_t)$ in $T_{\widehat\Psi_{(t,\theta)}}T^1\Hyp^2={\mathrm{Span}(\chi)\oplus\chi^\perp}$ as in Equation \eqref{eq:direct sum}, since the first term is a multiple of the generator of the geodesic flow, and the second term is in $\HP$.
This shows, by definition of the metric $\GG$, that $d\Psi_{(t,\theta)}(\partial_t)$ has squared norm $\sin^2(\theta)$ and that $d\Psi_{(t,\theta)}(\partial_t)$ and $d\Psi_{(t,\theta)}(\partial_\theta)$ are orthogonal. In conclusion, we have showed:
$$\Psi^*\GG=-d\theta^2+\sin^2(\theta)dt^2~.$$

We are now ready to produce our example of $\rho$-equivariant embedding $G:\widetilde M\to\G{2}$ which is not $\rho$-integrable. Consider $M=\Sph^1$, $\widetilde M=\R$, and the representation $\rho:\Z\to\Isom(\Hyp^2)$ which is a hyperbolic translation along $\gamma$. The induced action on $\G{2}$ preserves the image of $\Psi$ and its generator acts on the $(t,\theta)$-coordinates simply by $(t,\theta)\mapsto (t+c,\theta)$. Hence the map
$$G:\R\to\G{2}\qquad G(s)=\Psi(s,\theta_0)$$
for some constant $\theta_0\in (0,\pi)$ is a $\rho$-equivariant Riemannian embedding by the above expression of $\Psi^*\GG$. Of course the Lagrangian condition is {trivially satisfied since $n=1$}. By Theorem \ref{prop: riemannian global integrability} $G$ coincides with the Gauss map $G_\sigma$ associated to some embedding $\sigma:\R\to\Hyp^2$ with small curvature. It is easy to see that any such embedding $\sigma$ is not $\rho$-equivariant unless $\theta_0=\frac \pi 2$, see Figure \ref{fig:example_dim2}.
\end{example}

\begin{figure}[htbp]
\centering
\begin{minipage}[c]{.5\textwidth}
\centering
\includegraphics[width=.6\textwidth]{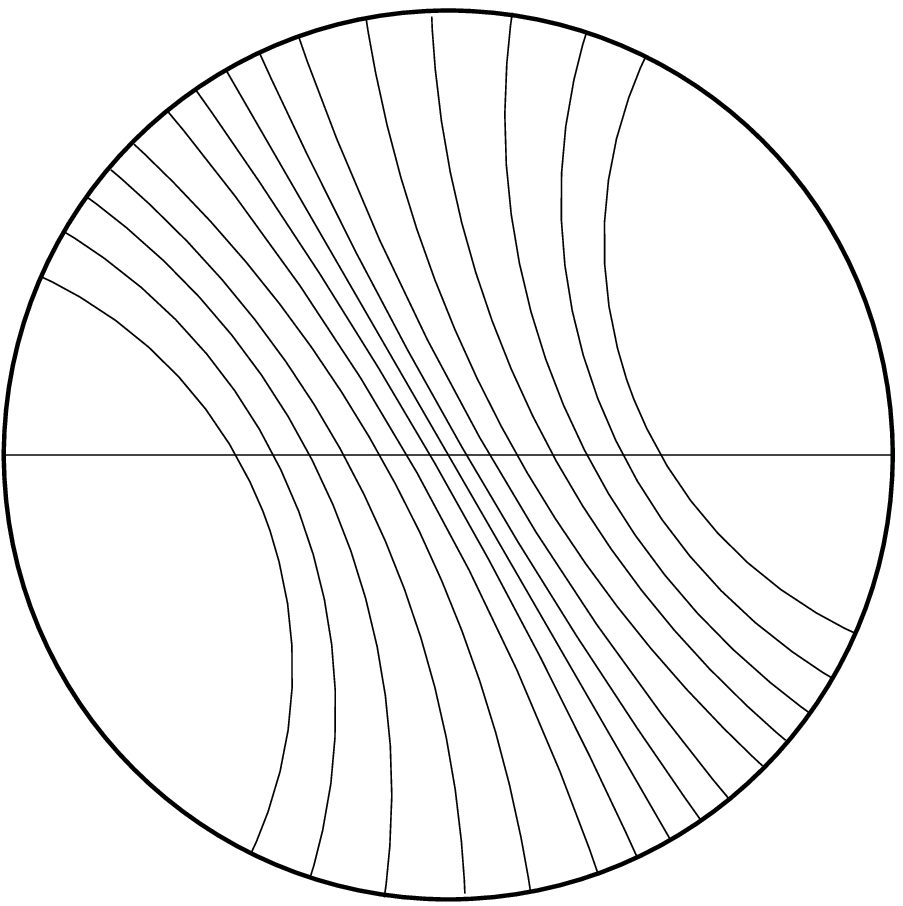} 
\end{minipage}%
\begin{minipage}[c]{.5\textwidth}
\centering
\includegraphics[width=.7\textwidth]{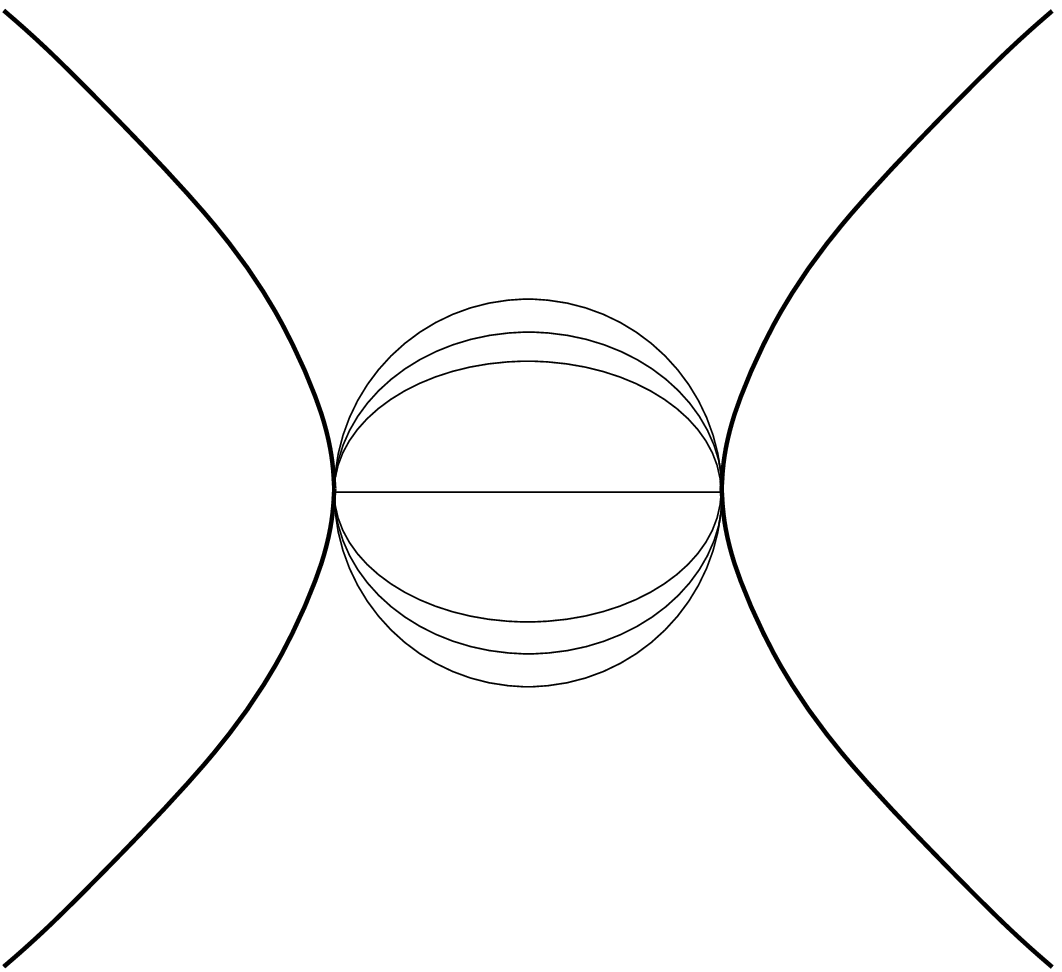} 
\end{minipage}%
\caption{On the left, the family of geodesics forming a fixed angle $\theta$ with the horizontal geodesic (in the Poincar\'e disc model of $\Hyp^2$). On the right, the metric $-d\theta^2+\sin^2(\theta)dt^2$ represents a portion of the Anti-de Sitter space of dimension 2. Here are pictured some lines defined by $\theta=c$, in the projective model of the Anti-de Sitter space.}\label{fig:example_dim2}
\end{figure}

{We also briefly provide an example of a locally integrable, but not globally integrable immersion in $\G{3}$ for $M$ not simply connected (lifting to the universal cover $\widetilde M$ this corresponds to $\rho$ being the trivial representation). This example in particular motivates Corollary \ref{cor hol H baby}, which is a direct consequence of Theorem \ref{teorema hol H baby}.}

\begin{example}\label{ex: global integrable non equivariant2}
First, let us consider a totally geodesic plane $\mathcal P$ in $\Hyp^3$ and an annulus $\mathcal A$ contained in $\mathcal P$. Let $c:\widetilde {\mathcal A}\to \mathcal A$ be the universal covering. Then $c$ is an immersion in $\Hyp^3$ with small principal curvatures (in fact, totally geodesic), and is clearly not injective. See Figure \ref{fig:annulus} on the left. Of course, in light of Proposition \ref{prop injectivity}, this is possible because the immersion $c$ is not complete.

Now, let us deform $\mathcal A$ in the following way. We cut $\mathcal A$ along a geodesic segment $s\subset\mathcal P$ to obtain a rectangle $\mathcal R$ having two (opposite) geodesic sides, say $r_1$ and $r_2$. Then we deform such rectangle to get an immersion $c':\mathcal R\to\Hyp^3$, so that one geodesic side remains unchanged (say $c'(r_1)=s$), while the other side $r_2$ is mapped to an $r$-cap equidistant from $\mathcal P$, for small $r$, in such a way that it projects to $s$ under the normal evolution. We can also arrange $c$ so that a neighbourhood of $r_1$ is mapped to $\mathcal P$, while a neighbourhood of $r_2$ is mapped to the $r$-cap equidistant from $\mathcal P$. See Figure \ref{fig:annulus} on the right.

By virtue of this construction, the Gauss map of $c'$ coincides on the edges $r_1$ and $r_2$ of $\mathcal R$, and therefore induces an immersion $G':\mathcal A\to\G{3}$. Clearly $G'$ is locally integrable, but not globally integrable.  In other words, the lift to the universal cover of $G'$ is a $\rho$-equivariant immersion of $\widetilde A$ to $\G 3$ which is not $\rho$-integrable, for $\rho$ the trivial representation.

\begin{figure}[htbp]
\centering
\includegraphics[height=4.5cm]{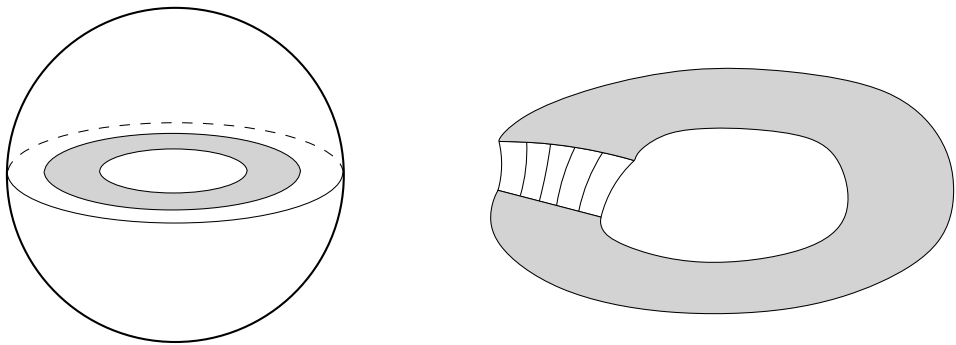} 

\caption{On the left, a totally geodesic annulus $\mathcal A$ in a plane $\mathcal P$. On the right, an embedded rectangle with the property that a neighbourhood of one side lies in $\mathcal P$, while a neighbourhood of the opposite side lies on an $r$-cap equidistant from  $\mathcal P$. Such rectangle induces an embedding of $\mathcal A$ in $\G{3}$ which is locally, but not globally, integrable.}\label{fig:annulus}
\end{figure}

\end{example}

Motivated by the previous examples, we introduce the relevant definition for our problem.

\begin{defi}\label{defi rhointegrable}
Given {an $n$-manifold $M$ and a} representation $\rho\colon \pi_1(M) \to \Isom(\Hyp^{n+1})$, a $\rho$-equivariant immersion $G\colon \widetilde M \to \G{n+1}$ is \emph{$\rho$-integrable} if there exists a $\rho$-equivariant immersion ${\widetilde\sigma}\colon \widetilde M \to \Hyp^{n+1}$ whose Gauss map is $G$.
\end{defi}

\subsection{Maslov class}

Let us now introduce the obstruction which will permit us to classify $\rho$-integrable Lagrangian immersions under the Riemannian assumption, namely the Maslov class. For this purpose, let $G:\widetilde M\to\G{n+1}$ be a Riemannian immersion. The
\emph{second fundamental form} of $G$ is a symmetric bilinear form on $M$ with values in the normal bundle of $G$, defined as
$$\overline\II(V,W)=({\mathbbm D}_{G_*V}(G_*W))^\perp$$
for vector fields $V,W$, where {$\mathbbm D$} denotes the ambient Levi-Civita connection {of $\GG$} and $\perp$ the projection to the normal subspace of $G$. {One can prove that $\overline\II(V,W)$ is a tensor, i.e. that it}  depends on the value of $V$ and $W$ {pointwise.} 
The
 \emph{mean curvature} is then 
 $$\overline{\mathrm H}=\frac{1}{n}\mathrm{tr}_{\overline\I}\overline\II~,$$
 that is, it is the trace of $\overline\II$ with respect to the first fundamental form $\overline\I$ of $G$, and is therefore a section of the normal bundle of $G$.

Consider now the 1-form on $\widetilde M$ given by $G^*(\Omega (\overline {\mathrm H}, \cdot ) )$. It will follow from Proposition \ref{Prop: formula H in G} {(see Corollary \ref{cor:maslov closed})} that this is a closed 1-form. Since $\Isom(\Hyp^{n+1})$ acts by automorphisms of the para-K\"ahler manifold $(\G{n+1}, \GG, \mathbb J, \Omega)$, if $G$ is $\rho$-equivariant, then the form $G^*(\Omega (\overline {\mathrm H}, \cdot ) )$ is $\pi_1(M)$-invariant: as a result, it defines a well-posed closed 1-form on $M$. Its cohomology class is the so-called Maslov class:

\begin{defi}\label{defi Maslov class}
Given an $n$-manifold $M$, a representation $\rho\colon \pi_1(M) \to \Isom(\Hyp^{n+1})$ and a $\rho$-equivariant Lagrangian and Riemannian {immersion} $G\colon \widetilde M \to \G{n+1}$, the
\emph{Maslov class} of $G$ is the cohomology class 
\[
\mu_G:=[G^*(\Omega (\overline {\mathrm H}, \cdot ) )]\in H_{dR}^1(M)~.
\]
\end{defi}

The main result of this section is the following, and it will be deduced as a {consequence of} Theorem \ref{Teorema hol H}.

\begin{thmx}
\label{teorema hol H baby}
Given an orientable $n$-manifold  $M$ and a representation $\rho\colon \pi_1(M) \to \Isom(\Hyp^{n+1})$, a $\rho$-equivariant  Riemannian and Lagrangian {immersion} $G\colon \widetilde M \to \G{n+1}$ is
$\rho$-integrable if and only if $\mu_G=0$ in $H_{dR}^1(M)$.
\end{thmx}

We immediately obtain the following characterization of global integrability {for $\pi_1(M)\ne \{1\}$}.

\begin{corx}
\label{cor hol H baby}
Given an orientable $n$-manifold  $M$ and an {immersion} $G\colon M \to \G{n+1}$, $G$ is the Gauss map of an immersion $\sigma:M\to\Hyp^{n+1}$ of small principal curvatures if and only
if $G$ is Riemannian and Lagrangian and $\mu_G=0$ in $H_{dR}^1(M)$.
\end{corx}
\begin{proof}
{Denote $\rho$ the trivial representation. Given $G:M\to\G{n+1}$, precomposing with the universal covering map we obtain an immersion $\widetilde G:\widetilde M\to\G{n+1}$ which is $\rho$-equivariant. Observe that $\widetilde G$ is the Gauss map of some immersion $\widetilde\sigma:\widetilde M\to \Hyp^{n+1}$ by Theorem \ref{prop: riemannian global integrability}. Then $G$ is the Gauss map of some immersion in $\Hyp^{n+1}$ if and only if $\widetilde\sigma$ descends to the quotient $M$, i.e. it is $\rho$-integrable. Hence this is equivalent to the vanishing of the Maslov class by Theorem \ref{teorema hol H baby}.}
\end{proof}

\subsection{Mean curvature of Gauss maps}

Recall that, given an embedding $\sigma:M\to\Hyp^{n+1}$ with small principal curvatures, we introduced in \eqref{eq:aux function} the function $f_\sigma:M\to\R$ which is the mean of the hyperbolic arctangents of the principal curvatures of $\sigma$. This function is strictly related to the mean curvature of the Gauss map of $\sigma$, as in the following proposition.

\begin{prop}
\label{Prop: formula H in G}
Let $M^n$ be an oriented manifold, $\sigma:M\to\Hyp^{n+1}$ an embedding with small principal curvatures, and $G_\sigma:M\to\G{n+1}$ its Gauss map. Then 
$$G_\sigma^*(\Omega (\overline {\mathrm H}, \cdot ) )=d (f_\sigma)= d\left({\frac{1}{n}}\sum_{i=1}^n\arctanh\lambda_i\right)~,$$
where $\lambda_1,\ldots,\lambda_n$ denote the principal curvatures of $\sigma$.
\end{prop}

The essential step in the proof of Proposition \ref{Prop: formula H in G} is the following computation for the mean curvature vector of the Gauss map $G_\sigma$:
\begin{equation} \label{eq:formula H in G}
\overline{\mathrm H}=-\JJ(dG_\sigma(\overline \nabla f_\sigma))~,
\end{equation}
where $\overline \nabla$ denotes the gradient with respect to the first fundamental form $\overline \I$ of $G_\sigma$. 
Indeed, once Equation \eqref{eq:formula H in G} is established, Proposition \ref{Prop: formula H in G} follows immediately since
$$\Omega(\overline{\mathrm H},dG_\sigma(V))=-\GG(\JJ(\overline{\mathrm H}),dG_\sigma(V))=\GG(dG_\sigma(\overline \nabla f_\sigma),dG_\sigma(V))=df_\sigma(V)~.$$

{The computations leading to Equation \eqref{eq:formula H in G} will be done in $T^1 \Hyp^{n+1}$ equipped with the metric $\widehat g_{T^1 \Hyp^{n+1}}$ defined in Remark \ref{rmk other metric2}, which is the restriction of the flat pseudo-Riemannian metric \eqref{eq:metric minkxmink} of $\R^{n+1,1}\times \R^{n+1,1}$ to $T^1\Hyp^{n+1}$, seen as a submanifold as in \eqref{eq:modelT}. This approach is actually very useful: the Levi-Civita connection of $\widehat g_{T^1 \Hyp^{n+1}}$ on $T^1\Hyp^{n+1}$, that we denote by
 $\widehat D$}, {will be} just the normal projection of the flat connection $\mathrm d$ of $\R^{n+1,1}\times \R^{n+1,1}$ to $T^1\Hyp^{n+1}$.

{Indeed}, the following lemma will be useful to compute the Levi-Civita connection {$\mathbbm D$} of $\G{n+1}$. Given a vector $X\in T_\ell\G{n+1}$ and $(x,v)\in \mathrm{p}^{-1}(\ell)$, we define the \emph{horizontal lift} of $X$ at   $(x,v)$ as the unique vector $\widetilde X\in T_{(x,v)}T^1\Hyp^{n+1}$ such that 
\begin{equation}\label{eq:hor lift}
\widetilde X\in \chi_{(x,v)}^\perp \qquad\text{and}\qquad d\mathrm p(\widetilde X)=X~.\end{equation}
For a vector field $X$ on an open set $U$ of $\G{n+1}$, we will also refer to the  vector field $\widetilde X$ on $\mathrm{p}^{-1}(U)$, defined by the conditions  \eqref{eq:hor lift}, as the \emph{horizontal lift} of the vector field $X$.

\begin{lemma}
\label{Lemma: curvatura media in G 1}
Given two  vector fields $X,Y$ on $\GG$, 
\[
\mathbbm D_XY=d\mathrm p(\widehat D_{\widetilde X}\widetilde Y)
\]
\end{lemma}
\begin{proof}
{By the well-known characterization of the Levi-Civita connection, }it is sufficient to prove that the expression
{$\mathbbm A_XY:=d\mathrm p(\widehat D_{\widetilde X}\widetilde Y)$}
is a {well-posed linear connection which is torsion-free and compatible with the metric of $\GG$. We remark that this is not obvious because, although the metric of $\GG$ is the restriction of the metric $\widehat g_{T^1 \Hyp^{n+1}}$ to $\chi^\perp$, there is no {flat} section of the bundle projection $p:T^1\Hyp^{n+1}\to\G{n+1}$, hence $\G{n+1}$ cannot be seen as an isometrically embedded submanifold of $T^1 \Hyp^{n+1}$.}

First, {observe that} the expression $(\mathbbm A_XY)_{|\ell}=(d_{(x,v)}\mathrm p)(\widehat D_{\widetilde X}\widetilde Y)$ does not depend on the choice of the point $(x,v)\in \mathrm{p}^{-1}(\ell)$. Indeed, given two points $(x_1,v_1)$ and $(x_2,v_2)$ in $\mathrm{p}^{-1}(\ell)$, there exists $t$ such that $(x_2,v_2)=\varphi_t(x_1,v_1)$. By a small adaptation of Lemma \ref{lemma:geodflow isometric}, the geodesic flow $\varphi_t$ acts by isometries of the metric $\widehat g_{T^1\Hyp^{n+1}}$ (see also Remarks \ref{rmk other metric1} and \ref{rmk other metric2}), hence it also preserves the horizontal lifts $\widetilde X$ and $\widetilde Y$ and the Levi-Civita connection $\widehat D$. Hence $d\mathrm p ( \widehat D_{\widetilde X} \widetilde Y)= \mathbbm A_X Y$ is a well-defined vector field on $\G{n+1}$ whose horizontal lift is the projection of $\widehat D_{\widetilde X} \widetilde Y $ to $\chi^\bot$.

{We} check that $\mathbbm A$ is a linear connection{. It} is immediate to check the additivity in $X$ and $Y$. Moreover we have the {$C^\infty$-}linearity in $X$ since:
$$\mathbbm A_{fX} Y= d \mathrm{p} \left( \widehat D_{(f\circ \mathrm{p})\widetilde X} \widetilde Y \right) = d\mathrm p \left( (f\circ \mathrm{p}) \hat D_{\widetilde X} \widetilde Y \right)= f (\mathbbm A_X Y)~,$$
and the Leibnitz rule in $Y$, for:
$$\mathbbm A_X (fY)= d\mathrm p \left( \partial_{\widetilde X} (f\circ \mathrm{p})\ \widetilde Y+ (f\circ \mathrm{p}) \widehat D_{\widetilde X} \widetilde Y\right)= \partial_X f\ Y + f (\mathbbm A_X Y)~.$$
The connection $\mathbbm A$ is torsion-free: 
$$	
    \mathbbm A_X Y - \mathbbm A_Y X =d \mathrm p (\widehat D_ {\widetilde X} \widetilde Y) - d\mathrm  p( \widehat D_{\widetilde Y} \widetilde X)= d\mathrm p ([\widetilde X, \widetilde Y])= [X,Y]~.
$$
Finally, we show that $\mathbbm A$ is compatible with the metric $\GG$:
\begin{align*}
        \GG(\mathbbm A_X Y, Z) + \GG(Y, \mathbbm A_X Z)  &= \widehat g_{T^1 \Hyp^{n+1} } (\widetilde{\mathbbm A_X Y}, \widetilde Z) + \widehat g_{T^1 \Hyp^{n+1} } (\widetilde Y, \widetilde{\mathbbm A_X Z})=\\
        &= \widehat g_{T^1 \Hyp^{n+1} } ( \widehat D_{\widetilde X} \widetilde Y, \widetilde Z) + \widehat g_{T^1 \Hyp^{n+1} } (\widetilde Y, \widehat D_{\widetilde X} \widetilde Z) =\\
        &= \partial_{\widetilde X} ( \widehat g_{T^1 \Hyp^{n+1} } (\widetilde Y, \widetilde Z)  )= \partial_X (\GG(Y,Z))~,    \end{align*}
      where in the first line we used the definition of $\GG$, and in the second line the fact that the horizontal lift of $\mathbbm A_X Y$ is the orthogonal projection (with kernel spanned by $\chi$) of $\widehat D_{\widetilde X}\widetilde Y$.
 \end{proof}

We are now ready to provide the proof of Proposition \ref{Prop: formula H in G}.

\begin{proof}[Proof of Proposition \ref{Prop: formula H in G}]
As already observed after Equation \eqref{eq:formula H in G}, it suffices to prove that 
$\overline{\mathrm H}=-\JJ(dG_\sigma(\overline\nabla f_\sigma))$. So we shall compute the mean curvature vector of $G_\sigma$ in $\G{n+1}$. For this purpose, let $\{e_1,\ldots,e_n\}$ be a local frame on $M$ which is orthonormal with respect to the first fundamental form $\overline \I=G_\sigma^*\GG$. To simplify the notation, let us denote $\epsilon_i:=dG_\sigma(e_i)$.
 Then $\{\JJ\epsilon_1,\ldots,\JJ\epsilon_n\}$ is an orthonormal basis for the normal bundle of $G_\sigma$, on which the metric $\GG$ is negative definite since $G_\sigma$ is Riemannian. The mean curvature vector can be computed as:
$$\overline{\mathrm H}=\frac{1}{n}\sum_{i=1}^n \overline \II(\epsilon_i,\epsilon_i)=-\frac{1}{n}\sum_{i=1}^n\sum_{k=1}^n \GG(\overline \II(\epsilon_i,\epsilon_i),\JJ\epsilon_k)\JJ\epsilon_k=-\frac{1}{n}\sum_{i=1}^n\sum_{k=1}^n \GG(\mathbbm D_{\epsilon_i}\epsilon_i,\JJ\epsilon_k)\JJ\epsilon_k~,$$
where in the last equality we used that $\overline \II(V,W)$ is the normal projection of $\mathbbm D_V W$. 

Let us now apply this expression to a particular {$\overline \I$-}orthonormal frame $\{e_1,\ldots,e_n\}$ obtained in the following way. Pick a local {$\I$-}orthogonal frame on $M$ of eigenvectors for the shape operator $B$ of $\sigma$, and normalize each of them so as to have unit norm for $\overline\I$. Hence each $e_i$ is an eigenvector of $B$, whose corresponding eigenvalue $\lambda_i$ are the principal curvatures of $\sigma$. We claim that, with this choice, $\GG(\mathbbm D_{\epsilon_i}\epsilon_i,\JJ\epsilon_k)=d(\arctanh\lambda_i)(e_k)$. This will conclude the proof, for
$$\overline{\mathrm H}=\frac{1}{n}\JJ\left(\sum_{i=1}^n\sum_{k=1}^n d(\arctanh\lambda_i)(e_k)\epsilon_k\right)=\JJ\left(\sum_{k=1}^n \partial_{e_k}f_\sigma\ dG_\sigma(e_k)\right)=\JJ(dG_\sigma(\overline \nabla f_\sigma)))~.$$

To show the claim, we will first use Lemma \ref{Lemma: curvatura media in G 1} to get 
$$\GG(\mathbbm D_{\epsilon_i}\epsilon_i,\JJ\epsilon_k)=\widehat g_{T^1\Hyp^{n+1}}(\widehat D_{\widetilde\epsilon_i}\widetilde\epsilon_i,J\widetilde \epsilon_k)~,$$
where $\widehat D$ is the Levi-Civita connection of $\widehat g_{T^1\Hyp^{n+1}}$ {and $\widetilde\epsilon_i$ is the horizontal lift of $\epsilon_i$}. 
As in Equation \eqref{eq: differential of sigma tilde}, we can write $$d\zeta_\sigma(e_i)=(d\sigma(e_i),-\lambda_id\sigma(e_i))$$
and the Levi-Civita connection $\widehat D$ is the normal projection with respect to the metric \eqref{eq:metric minkxmink} of the flat connection $\mathrm d$ of $\R^{n+1,1}\times\R^{n+1,1}$. Hence we can compute: 
\begin{align*}
\GG(\mathbbm D_{\epsilon_i}\epsilon_i,\JJ\epsilon_k)&=-\lambda_k\langle \mathrm d_{d\sigma(e_i)}d\sigma(e_i),d\sigma(e_k)\rangle+\lambda_i\langle \mathrm d_{d\sigma(e_i)}d\sigma(e_i),d\sigma(e_k)\rangle+\partial_{e_i} \lambda_i\langle d\sigma(e_i),d\sigma(e_k)\rangle\\
&=(\lambda_i-\lambda_k)\ \I(\nabla_{e_i}e_i,e_k)+(\partial_{e_i} \lambda_i)\ \I(e_i,e_k)~.
\end{align*}
We recall that $g$ denotes the first fundamental form of $\sigma$, and $\nabla$ its Levi-Civita connection, and in the last equality we used that the Levi-Civita connection of $\Hyp^{n+1}$ is the  projection to the hyperboloid in Minkowski space $\R^{n+1,1}$ of the ambient flat connection.

Now, when $i=k$ we obtained the desired result:
$$\GG(\mathbbm D_{\epsilon_i}\epsilon_i,\JJ\epsilon_i)=\frac{\partial_{e_i} \lambda_i}{1-\lambda_i^2}=d(\arctanh\lambda_i)(e_i)$$
since $e_i$ is a unit vector for the metric $\overline\I$, hence using the expression $\overline\I=\I-\III$ from Equation \eqref{eq:fff gauss} its {squared} norm for the metric $\I$ is $(1-\lambda_i^2)^{-1}$. When $i\neq k$, the latter term disappears since $\{e_1,\ldots,e_n\}$ is an orthogonal frame for $g$, and we are thus left with showing that 
$$(\lambda_i-\lambda_k)\ \I(\nabla_{e_i}e_i,e_k)=d(\arctanh\lambda_i)(e_k)~.$$

For this purpose, using the compatibility of $\nabla$ with the metric, namely $\partial_{ e_i} (\I( e_i,  e_k))=\I(\nabla_{ e_i} e_i,  e_k )+\I( e_i,  \nabla_{ e_i} e_k)$, that $\I(e_i,e_k)=0$, and that $\nabla$ is torsion-free, we get:
$$(\lambda_i-\lambda_k)\ \I(\nabla_{ e_i}  e_i,  e_k) = - (\lambda_i-\lambda_k)\ \I( e_i, \nabla_{ e_i}  e_k)
    = \lambda_k\ \I( e_i,  \nabla_{ e_i} e_k) - \lambda_i\ \I( e_i,  \nabla_{ e_k} e_i) -\lambda_i\ \I( e_i, [ e_i, e_k ])~.$$

Now, 
recall that the Codazzi equation for $\sigma$ is $d^\nabla B=0$. 
Applying it to the vector fields $e_i$ and $e_k$, we obtain
$$d^\nabla B( e_i,  e_k)= \nabla_{ e_i} (\lambda_k  e_k) - \nabla_{ e_k} (\lambda_i  e_i)- B([ e_i,  e_k])=0,$$
from which we derive
\begin{equation}
\label{eq: lemma2 curvatura media}
 \lambda_k \nabla_{e_i}e_k - \lambda_i \nabla_{e_k}e_i= (\partial_{ e_k} \lambda_i) e_i -(\partial_{ e_i} \lambda_k) e_k  + B([ e_i,  e_k])~.
\end{equation}
Using Equation \eqref{eq: lemma2 curvatura media} in the previous expression, we finally obtain:

\begin{align*}
    (\lambda_i-\lambda_k)\ \I(\nabla_{ e_i}  e_i,  e_k) =
    & (\partial_{ e_k}\lambda_i)\ \I( e_i,  e_i) - (\partial_{ e_i}\lambda_k)\ \I( e_i,  e_k)+ \I( e_i, B[ e_i,  e_k] )- \I(B( e_i), [ e_i,  e_k])\\
    = &  \frac{\partial_{ e_k} \lambda_i}{1- \lambda_i^2}=d(\arctanh \lambda_i)(e_k)
\end{align*}
where the cancellations from the first to the second line are due to the fact that $B$ is $\I$-self adjoint and that $\I(e_i,e_k)=0$. This concludes the proof.
\end{proof}

\begin{cor}\label{cor:maslov closed}
Given an $n$-manifold $M$, a representation $\rho\colon \pi_1(M) \to \Isom(\Hyp^{n+1})$ and a $\rho$-equivariant Lagrangian and Riemannian {immersion} $G\colon \widetilde M \to \G{n+1}$, the Maslov class $\mu_G$ is a well-defined cohomology class in $H^1_{dR}(M,\R)$.
\end{cor}
\begin{proof}
{By Theorem \ref{prop: riemannian global integrability}, $G$ is the Gauss map of a (in general non equivariant) immersion $\sigma:\widetilde M\to\Hyp^{n+1}$. By Proposition \ref{Prop: formula H in G}, the 1-form $G^*(\Omega (\overline {\mathrm H}, \cdot ) )$ on $\widetilde M$ is exact, and $\rho$-equivariant, hence it induces a closed 1-form on $M$ whose cohomology class is $\mu_G$ as in Definition \ref{defi Maslov class}.}
\end{proof}

\subsection{Holonomy of flat principal bundles}\label{sec:hol flat princ bdles}

An immediate consequence of Proposition \ref{Prop: formula H in G} is that the vanishing of the Maslov class is a necessary condition for a $\rho$-equivariant Lagrangian and Riemannian embedding $G:\widetilde M\to \G{n+1}$ to be $\rho$-integrable (Definition \ref{defi rhointegrable}). Indeed, if $\widetilde\sigma:\widetilde M\to\Hyp^{n+1}$ is a $\rho$-equivariant embedding with $G_{\widetilde\sigma}=G$ (hence necessarily with small principal curvatures), then the function $f_{\widetilde\sigma}$ descends to a well-defined function on $M$, hence by Proposition \ref{Prop: formula H in G} {$G^*(\Omega (\overline {\mathrm H}, \cdot ) )$} is an exact 1-form, i.e. the Maslov class $\mu_{G_{\widetilde\sigma}}$ vanishes. We will now see that this condition is also sufficient, which will be a consequence of a more general result, {Theorem \ref{Teorema hol H}}.

Let $G:\widetilde M\to\G{n+1}$ be a $\rho$-equivariant Lagrangian and Riemannian embedding. We have already used that the $G$-pull-back bundle $\widetilde {\mathrm{p}}_{G}\colon \widetilde P \to \widetilde M$ of $\mathrm p:T^1\Hyp^{n+1}\to\G{n+1}$ is a flat trivial $\R$-principal bundle over $\widetilde M$, {namely, it is isomorphic, as a flat principal bundle, to the trivial bundle $\widetilde M\times \R\to \widetilde M$ with flat sections $\widetilde M\times \{\ast\}$}. Moreover, {$G$ being $\rho$-equivariant,} the fundamental group $\pi_1(M)$ acts freely and properly discontinously on $\widetilde P$, thus inducing a flat $\R$-principal bundle structure $\mathrm p_G\colon P \to M$, where $P$ is the quotient of $\widetilde P$ by the action of $\pi_1(M)$. However the bundle $\mathrm p_G\colon P\to M$ is not trivial in general. The obstruction to triviality is represented by the \emph{holonomy} of the bundle, which can be defined, in our setting, as follows.

\begin{defi}\label{defi holonomy}
Let $P\to M$ be a flat principal $\R$-bundle that is isomorphic to the quotient of the trivial bundle $\widetilde M\times\R\to\widetilde M$ by an equivariant (left) action of $\pi_1(M)$.  The \emph{holonomy representation} is the representation $\hol:\pi_1(M)\to\R$ such that  the action of $\pi_1(M)$ is expressed by:
$$\alpha\cdot(m,s)=(\alpha\cdot m,\hol(\alpha)+ s)~.$$
\end{defi}

\begin{remark}\label{rmk: alternative definition of hol}
	Fix $p\in M$ and $\alpha$ a closed $C^1$ loop based at $p$. Then pick a horizontal lift $\widehat\alpha$ to the total space of $\mathrm p_G$, namely with $\frac{d\widehat \alpha}{dt}$ orthogonal to the fibers, so that $\mathrm p\circ  \widehat \alpha=\alpha$. (The lift is uniquely determined by its initial point in $\mathrm{p}_G^{-1}(p)$.)  It follows from Definition \ref{defi holonomy} that
	\[
	\hol_G(\alpha) \cdot\widehat\alpha(1)=  \widehat \alpha(0) .
	\]
	In the identification $\pi_1(M)=\pi_1(M, [p])$, this allows to give an alternative definition of $\hol_G$ through homotopy classes of closed paths in $M$.
\end{remark}

\begin{remark}
We remark that in general, for flat principal $G$-bundles, the holonomy representation is only defined up to conjugacy, but in out case $G=\R$ is abelian and therefore $\hol$ is uniquely determined by the isomorphism class of the flat principal bundle.

Also observe that, since $\R$ is {abelian}, $\hol_G$ induces a map
from $H_1(M, \Z)$ to $\R$,
where $H_1(M,\Z)$ is the first homology group of $M$ and we are using that there is a canonical isomorphism between $H_1(M, \Z)$ and the abelianization of the fundamental group of $M$ in a point. Equivalently, $\hol_G$ is identified to an element of $H^1(M, \R)$.
\end{remark}

We can interpret the holonomy of the principal bundle $\mathrm p_G$ in terms of the geometry of $\Hyp^{n+1}$. 
{Global flat sections of the trivial bundle  $\widetilde{\mathrm{p}}_{G}\colon \widetilde P \to \widetilde M$ correspond to {Riemannian} embeddings $\zeta:\widetilde M\to T^1\Hyp^{n+1}$ as in Corollaries \ref{cor:lagrangian1} and \ref{cor:lagrangian2}.}
By Theorem \ref{prop: riemannian global integrability}, such a $\zeta$ is the lift to $T^1\Hyp^{n+1}$ of an embedding $\sigma:\widetilde M\to\Hyp^{n+1}$ with small principal curvatures. 

Now, let $\alpha\in\pi_1(M)$. {By equivariance of $G$, namely $G\circ\alpha=\rho(\alpha)\circ G$, it follows that $\mathrm p \circ\zeta\circ \alpha= \mathrm p \circ \rho(\alpha)\circ \zeta$, hence $\rho(\alpha)\circ\zeta\circ\alpha^{-1}:\widetilde M\to T^1\Hyp^{n+1}$ provides another flat section of the pull-back bundle $\widetilde {\mathrm{p}}_{G}$. Therefore} there exists $t_\alpha\in\R$ such that
 \begin{equation}\label{eq:equivariance lifts}
 \varphi_{t_\alpha}\circ\zeta=\rho(\alpha)\circ\zeta\circ\alpha^{-1}~.
 \end{equation}
Then the value $t_\alpha$ is precisely the holonomy of the quotient bundle $\mathrm p_G$, namely the group representation
\[
\hol_G\colon \pi_1(M)\to \R\qquad \hol_G(\alpha)=t_\alpha
\]
A direct consequence of this discussion is the following:

\begin{lemma}\label{lemma:integrable iff trivial bundle}
Given an $n$-manifold $M$ and a representation $\rho\colon \pi_1(M) \to \Isom(\Hyp^{n+1})$, a $\rho$-equivariant Lagrangian and Riemannian embedding $G\colon \widetilde M \to \G{n+1}$ is $\rho$-integrable if and only if the $\R$-principal flat bundle $\mathrm p_G$ is trivial.
\end{lemma}
\begin{proof}
The bundle $\mathrm p_G$ is trivial if and only if its holonomy $\hol_G$ vanishes identically, that is, if and only if $t_\alpha=0$ for every $\alpha\in\pi_1(M)$. By the above construction, this is equivalent to the condition that $\zeta\circ\alpha=\rho(\alpha)\circ\zeta$ for all $\alpha$, which {is equivalent to} $\sigma\circ\alpha=\rho(\alpha)\circ\sigma$, {namely} that $\sigma$ is $\rho$-equivariant.
\end{proof}

We are ready to prove the following.

\begin{theorem}
\label{Teorema hol H}
Given an $n$-manifold $M$, a representation $\rho\colon \pi_1(M) \to \Isom(\Hyp^{n+1})$ and a $\rho$-equivariant Lagrangian and Riemannian embedding $G\colon \widetilde M \to \G{n+1}$, the holonomy of $\mathrm p_G$ is given by
\[
\hol_G(\alpha)=\int_{\alpha} \mu_G.
\]
for all $\alpha \in \pi_1(M)$.
\end{theorem}

Observe that Theorem \ref{teorema hol H baby} follows immediately from Theorem \ref{Teorema hol H} since, by the standard de Rham Theorem, there exists an isomorphism 
\begin{equation*}
    \begin{split}
        H_{dR}^1(M, \R) &\xrightarrow{\sim} H^1(M, \R) \\
        \eta &\mapsto \bigg( \xi \mapsto \int_\xi \eta \bigg),
    \end{split}
\end{equation*}
hence $\hol_G\equiv 0$ if and only if $\mu_G=0$. 

\begin{proof}[Proof of Theorem \ref{Teorema hol H}]
Let $\zeta:\widetilde M\to T^1\Hyp^{n+1}$ be a map {such} that $\mathrm{p}\circ \zeta=G$, so as to induce a global section of the pull-back bundle $\widetilde {\mathrm p}_G$. Then by Equation \eqref{eq:equivariance lifts} the holonomy $t_\alpha=\hol_G(\alpha)$ satisfies $\varphi_{t_\alpha}\circ\zeta\circ\alpha=\rho(\alpha)\circ\zeta$. By Proposition \ref{prop: flusso geod e flusso normale}, this gives the following equivariance relation for $\sigma=\pi\circ\zeta$:
$$(\sigma\circ\alpha)_{t_\alpha}=\rho(\alpha)\circ\sigma~.$$
{Let now $f_\sigma$ denote the mean of the hyperbolic arctangents of the principal curvatures, as in Equation \eqref{eq:aux function}}. Lemma \ref{lemma:evolution fsigma} and the fact that $\rho(\alpha)$ acts isometrically imply:
$$f_{\sigma\circ\alpha}=f_\sigma+t_\alpha~.$$
Now, by Proposition \ref{Prop: formula H in G} and the definition of the Maslov class, we have:
$$\int_\alpha \mu_G=\int_\alpha df_\sigma=f_\sigma(\alpha(p))-f_\sigma(p)=t_\alpha$$
for any point $p\in M$. This concludes the proof.
\end{proof}

\subsection{Minimal Lagrangian immersions}

We prove here two direct corollaries of Theorem \ref{teorema hol H baby}. Let us first recall the definition of minimal Lagrangian (Riemannian) immersions.

\begin{defi}
A {Riemannian} immersion of an $n$-manifold into $\G{n+1}$ is \emph{minimal Lagrangian} if:
\begin{itemize}
\item  its mean curvature vector vanishes identically;
\item it is Lagrangian with respect to the symplectic form $\Omega$.
\end{itemize}
\end{defi}

Our first corollary is essentially a consequence of Theorem \ref{teorema hol H baby}.

\begin{cor} \label{cor:minimal is rho integrable}
Let $M^n$ be a closed {orientable} manifold  and  $\rho:\pi_1(M)\to\Isom(\Hyp^{n+1})$ a representation. If $G:\widetilde M\to\G{n+1}$ is a $\rho$-equivariant Riemannian minimal Lagrangian immersion, then $G$ is the Gauss map of a $\rho$-equivariant embedding $\widetilde\sigma:\widetilde M\to \Hyp^{n+1}$ with small principal curvatures such that
$$f_{\widetilde\sigma}={\frac{1}{n}}\sum_{i=1}^n\arctanh\lambda_i=0~.$$
In particular, $\rho$ is a nearly-Fuchsian representation and $G$ is an embedding.
\end{cor}
We remark that if $n=2$, then the condition $f_{\widetilde\sigma}=0$ is equivalent to $\lambda_1+\lambda_2=0$ since $\arctanh$ is an odd {and injective} function. That is, in this case $\widetilde\sigma$ is a \emph{minimal} embedding in $\Hyp^3$. 
\begin{proof}
Suppose $G$ is a $\rho$-equivariant minimal Lagrangian immersion. Since its mean curvature vector vanishes identically, we have $\mu_G=0$ and therefore $G$ is $\rho$-integrable by Theorem \ref{teorema hol H baby}. That is, there exists a $\rho$-equivariant immersion $\widetilde\sigma:\widetilde M\to\Hyp^{n+1}$ such that $G=G_{\widetilde\sigma}$. By Proposition \ref{prop: small curv sse riemannian}, $\widetilde\sigma$ has small principal curvatures, hence $\rho$ is nearly-Fuchsian. By Proposition \ref{Prop: formula H in G}, we have that $f_{\widetilde\sigma}$ is constant. By Lemma \ref{lemma:evolution fsigma}, up to taking the normal evolution, we can find $\widetilde\sigma$ such that $f_{\widetilde\sigma}$ vanishes identically.

Finally, $\widetilde \sigma$ is complete by cocompactness, and therefore both $\widetilde\sigma$ and $G$ are embeddings by Proposition \ref{prop injectivity} and Proposition \ref{prop:gauss maps diffeo onto image}.  
\end{proof}

The following is a uniqueness result for $\rho$-equivariant minimal Lagrangian immersions.

\begin{corx}\label{cor:uniqueness min lag}
Given a closed {orientable} manifold $M^n$ and a representation $\rho:\pi_1(M)\to\Isom(\Hyp^{n+1})$, there exists at most one $\rho$-equivariant Riemannian minimal Lagrangian immersion $G:\widetilde M\to\G{n+1}$ up to reparametrization. If such a $G$ exists, then $\rho$ is nearly-Fuchsian and $G$ {induces a {minimal Lagrangian} embedding of $M$ in $\mathcal G_\rho$}.
\end{corx}
\begin{proof}
Suppose that $G$ and $G'$ are $\rho$-equivariant minimal Lagrangian immersions. By Corollary \ref{cor:minimal is rho integrable}, there exist $\rho$-equivariant embeddings $\widetilde\sigma,\widetilde\sigma':\widetilde M\to\Hyp^{n+1}$ {with small principal curvatures} such that $G=G_{\widetilde\sigma}$ and $G'=G_{\widetilde\sigma'}$, with $f_{\widetilde\sigma}=f_{\widetilde\sigma'}=0$.  {Moreover, $G$ and $G'$ induce embeddings in  $\mathcal G_\rho$ by Corollary \ref{cor:embedding in Grho}.}

By Remark \ref{rmk:embedding in the quotient}, both $\sigma$ and $\sigma'$ induce embeddings of {$M$ in the nearly-Fuchsian manifold $\faktor{\Hyp^{n+1}}{\rho(\pi_1(M))}$; let us denote with $\Sigma$ and $\Sigma'$ the corresponding images}. We claim that $\Sigma=\Sigma'$, which implies the uniqueness in the statement. 

To see this, consider the signed distance from $\Sigma$, which is a proper function 
$$r:\faktor{\Hyp^{n+1}}{\rho(\pi_1(M))}\to\R~.$$
 Since $\Sigma'$ is closed, $r|_{\Sigma'}$ admits a maximum value $r_{\max}$ achieved at some point $x_{\max}\in\Sigma'$.   This means that at the point $x_{\max}$, $\Sigma'$ is tangent to a hypersurface $\Sigma_{r_{\max}}$ at signed distance $r_{\max}$ from $\Sigma$, and $\Sigma'$ is contained in the side of $\Sigma_{r_{\max}}$ where $r$ is decreasing. This implies that, if $B'$ denotes the shape operator of $\Sigma'$ and $B_{r_{\max}}$ that of $\Sigma_{r_{\max}}$, both computed with respect to the unit normal vector pointing to the side of increasing $r$, then $B_{r_{\max}}-B'$ is positive semi-definite at $x_{\max}$.
 
Let us now denote by $\lambda_1,\ldots,\lambda_n$ the eigenvalues of $B_{r_{\max}}$ and $\lambda'_1,\ldots,\lambda_n'$ those of $B'$. Let us moreover assume that $\lambda_1\leq \ldots\leq\lambda_n$ and similarly for the $\lambda_i'$. By Weyl's monotonicity theorem, $\lambda_i\geq \lambda_i'$ at $x_{\max}$ for $i=1,\ldots,n$. Since $\arctanh$ is a monotone increasing function, this implies that
$$\sum_{i=1}^n\arctanh\lambda_i(x_{\max})\geq \sum_{i=1}^n\arctanh\lambda_i'(x_{\max})~.$$
Since $f_{\widetilde\sigma'}=0$, the right-hand side vanishes. On the other hand, since $f_{\widetilde\sigma}=0$, from Lemma \ref{lemma:evolution fsigma} the left-hand side is identically equal to $-r_{\max}$. Hence $r_{\max}\leq 0$. Repeating the same argument replacing the maximum point of $r$ on $\Sigma'$ by the minimum point, one shows $r_{\min}\geq 0$. Hence $r|_{\Sigma'}$ vanishes identically, which proves that $\Sigma=\Sigma'$ and thus concludes the proof.
\end{proof}

\section{Equivariant integrability: Hamiltonian symplectomorphisms}\label{sec:hamiltonian}

In this section we will provide the second characterization of $\rho$-integrability, in the case of a nearly-Fuchsian representation $\rho:\pi_1(M)\to\Isom(\Hyp^{n+1})$. We first introduce the terminology and state the result (Theorem \ref{thm:second char ham}); then we introduce the so-called \emph{Lagrangian Flux map} which will play a central role in the proof of Theorem \ref{thm:second char ham}.

\subsection{Hamiltonian group and nearly-Fuchsian manifolds}
We will restrict hereafter to the case of nearly-Fuchsian representations $\rho:\pi_1(M)\to\Isom(\Hyp^{n+1})$. 
Let $G:\widetilde M\to\G{n+1}$ be a $\rho$-integrable immersion as in Definition \ref{defi rhointegrable}. Since $\rho$ is nearly-Fuchsian, we showed in Corollary \ref{cor:embedding in Grho} that $G$ induces an embedded submanifold in the para-K\"ahler manifold $\mathcal G_\rho$, defined in Definition \ref{defi quotient Grho}. This motivates the following definition in the spirit of Definition \ref{defi rhointegrable}.

\begin{defi}\label{defi rhointegrable submfd}
Given a closed {orientable} $n$-manifold $M$ and a nearly-Fuchsian representation $\rho:\pi_1(M)\to\Isom(\Hyp^{n+1})$,
 an embedding 
 $M\to \mathcal G_\rho$ is $\rho$-\emph{integrable} if it is induced in the quotient from {a $\rho$-integrable} embedding ${G}\colon \widetilde M\to \G{n+1}$.
Similarly, an embedded submanifold $\mathcal L\subset \mathcal G_\rho$ is $\rho$-\emph{integrable} if it is the image of a $\rho$-integrable embedding.
\end{defi}

Theorem \ref{thm:second char ham} below gives a description of the set of $\rho$-integrable submanifolds $\mathcal L\subset \mathcal G_\rho$ which are induced by immersions $G$ with small principal curvatures. Clearly, {as we have previously shown}, a necessary condition on $\mathcal L$ is that of being Lagrangian and Riemannian. To state the theorem, we need to recall the notion of Hamiltonian symplectomorphism.

\begin{defi}\label{defi Hamc}
Given a symplectic manifold $(\mathcal X,\Omega)$, a compactly supported symplectomorphism $\Phi$ is \emph{Hamiltonian} if there exists a compactly supported smooth function $F_\bullet:\mathcal X\times[0,1]\to\R$
 such that $\Phi=\Phi_1$, where $\Phi_{s_0}$ is the flow at time $s_0$ of the (time-dependent) vector field $X_s$ defined by:
\begin{equation}\label{eq:hamiltonian function}
dF_s=\Omega(X_s,\cdot)~.
\end{equation}
The isotopy  $\Phi_\bullet:\mathcal X\times [0,1]\to \mathcal X$ is called \emph{Hamiltonian isotopy}.
\end{defi} 

\begin{remark}\label{rmk:hamiltonian and cartan}
 If $\Phi_\bullet$ is a Hamiltonian isotopy as in Definition \ref{defi Hamc}, then $\Phi_s$ is a symplectomorphism for every $s\in[0,1]$. Indeed
$$\mathcal L_{X_s}\Omega=\iota_{X_s}d\Omega+d(\iota_{X_s}\Omega)=0$$
as a consequence of Cartan's formula and Equation \eqref{eq:hamiltonian function}, and $\Phi_s$ is clearly Hamiltonian. 
\end{remark}

Compactly supported Hamiltonian symplectomorphisms form a group which we will denote by $\Ham_c(\mathcal X,\Omega)$.

{The aim of this section is to prove the following result.}

\begin{thmx}\label{thm:second char ham}
Let $M$ be a closed {orientable} $n$-manifold, $\rho:\pi_1(M)\to\Isom(\Hyp^{n+1})$ be a nearly-Fuchsian representation and $\mathcal L\subset\mathcal G_\rho$ a Riemannian $\rho$-integrable submanifold. Then a Riemannian submanifold $\mathcal L'$ is $\rho$-integrable if and only if there exists $\Phi\in \Ham_c(\mathcal G_\rho,\Omega)$ such that $\Phi(\mathcal L)=\mathcal L'$.
\end{thmx}

Of course, although not stated in Theorem \ref{thm:second char ham}, both $\mathcal L$ and $\mathcal L'$ are necessarily Lagrangian as a consequence of Corollary \ref{cor:lagrangian1}.

\subsection{The Lagrangian Flux}

We shall now define the Flux map for Lagrangian submanifolds, which was introduced in \cite{solomon}, and relate it to the holonomy of $\R$-principal bundles.

\begin{defi}\label{defi:flux}
Let $(\mathcal X,\Omega)$ be a symplectic manifold and let $\Upsilon_\bullet:M\times[0,1]\to\mathcal X$ be a smooth map such that each $\Upsilon_t$ is a Lagrangian embedding of $M$. Then we define:
$$\Flux(\Upsilon_\bullet)=\int_0^1 \Upsilon_s^*(\Omega(X_s,\cdot))ds\in H^1_{dR}(M,\R)~,$$
where
$$X_{s_0}(\Upsilon_{s_0}(p))=\left.\frac{d}{ds}\right|_{s=s_0}\Upsilon_s(p)\in T_{\Upsilon_{s_0}(p)}\mathcal X~.$$
\end{defi}

Observe that by Cartan's formula the integrand $\Upsilon_s^*(\Omega(X_s,\cdot))$ is a closed 1-form for every $s$, hence $\Flux(\Upsilon_\bullet)$ is well-defined as a cohomology class in $H^1_{dR}(M,\R)$. 

Now, let $\mathcal L$ be a Lagrangian embedded submanifold in $\mathcal G_\rho$, which is induced by a $\rho$-equivariant immersion $G:\widetilde M\to\G{n+1}$. Recall that in Section \ref{sec:hol flat princ bdles} we defined the principal $\R$-bundle $\mathrm p_G$ as the quotient of $\widetilde{\mathrm p}_G=G^*\mathrm p$ by the action of $\pi_1(M)$. Moreover in Theorem \ref{Teorema hol H} we computed the holonomy 
$$\hol_G:\pi_1(M)\to\R$$
of $\mathrm p_G$. 
The key relation between Lagrangian flux and $\hol_{G}$ is stated in the following proposition.

\begin{prop} 
\label{prop: flux = diff holonomy}
Let $M$ be a closed {orientable} $n$-manifold and $\rho:\pi_1(M)\to\Isom(\Hyp^{n+1})$ be a nearly-Fuchsian representation. If $\Upsilon_s$ is as in Definition \ref{defi:flux} and $\Upsilon_0(M)=\mathcal L$, $\Upsilon_1(M)=\mathcal L'$, then
\[
\hol_{\Upsilon_1} (\alpha) - \hol_{\Upsilon_0} (\alpha)=\int_\alpha\Flux (\Upsilon_\bullet)~.
\]
In particular, $\Flux (\Upsilon_\bullet)  (\alpha)$ depends uniquely on the endpoints of $\Upsilon_\bullet$.
\end{prop}

To prove Proposition \ref{prop: flux = diff holonomy}, we will make use of the following expression for the holonomy representation.

\begin{prop}
\label{prop: holonomy as integral of loops}
Let $G\colon \widetilde M\to\G{n+1}$ be a $\rho$-equivariant Lagrangian embedding and $\mathrm p_G$ be the associated $\R$-principal bundle over $M$. If $\alpha:[0,1]\to M$ is a smooth loop and $\overline\alpha$ a smooth loop in the total space of $\mathrm p_G$ such that $\alpha= \mathrm p_G (\overline \alpha)$, then
\[
\hol_G(\alpha)= \int_{\overline \alpha} \omega
\]
where $\omega$ is the principal connection of $\mathrm p_G$.
\end{prop}

\begin{proof}
Say $\alpha(0)=\alpha(1)= x_0$. {Recalling Remark \ref{rmk: alternative definition of hol}},
let $\widehat \alpha$ be the horizontal lift of $\alpha$ starting at $\overline \alpha(0)$. We apply Stokes' theorem. Define {a smooth map $f$ from $[0,1]\times [0,1]$ to the total space of $\mathrm p_G$} so that 
\begin{itemize}
    \item $f(x,0)= \overline \alpha(x)$,
    \item $f(x,1)= \widehat \alpha(x)$, 
    \item {$f(0, y)\equiv \overline \alpha(0)=\overline\alpha(1)$},
    \item {$y\mapsto f(1, y)$  parametrizes the interval from $\overline \alpha(1)$ to $\widehat \alpha(1)$ in $\mathrm p_G^{-1} (x_0)\approx \R$.}
\end{itemize}
{By Stokes' Theorem and the flatness of $\mathrm p_G$}, one gets that  
\begin{align*}
0= \int_{[0,1]\times [0,1]} f^* d\omega&=
\int_{\overline \alpha} \omega + \int_{f(1, \cdot)} \omega - \int_{\widehat \alpha} \omega - \int_{f(0, \cdot)} \omega\\
&=\int_{\overline \alpha} \omega + \int_{f(1, \cdot)} \omega.
\end{align*}

By Remark \ref{rmk: alternative definition of hol},  $\widehat \alpha(1)= (-\hol_G(\alpha)) \cdot \overline\alpha(1)$. 
Since $\omega= g_{T^1 \Hyp^{n+1}}(\chi, \cdot)$ and $f(1, \cdot)$ is contained in {$\mathrm p_G^{-1} (x_0)$}, one gets that
\[
\int_{f(1, \cdot)} \omega=- \hol_G(\alpha)
\]
and the proof follows.
\end{proof}

\begin{proof}[Proof of Proposition \ref{prop: flux = diff holonomy}]
Define 
$\Theta:[0,1]\times S^1\to M$ by
$\Theta(s,t)=\Upsilon_s(\alpha(t))$. Since the bundle $\mathrm p_G$ has contractible fibre, there always exists a smooth global section. In particular, there exists $\overline\Theta$ such that $\Theta=\mathrm p_G\circ \overline\Theta$. 
By Proposition \ref{prop: holonomy as integral of loops}, recalling that $d\omega=\mathrm{p}^* \Omega$, and applying Stokes' Theorem, we obtain:
$$\hol_{\Upsilon_1} (\alpha) - \hol_{\Upsilon_0}(\alpha)=\int_{\overline\Theta(1, \cdot)} \omega - \int_{\overline\Theta(0, \cdot)} \omega=\int_{[0,1]\times S^1}\overline\Theta^*d\omega=\int_{[0,1]\times S^1}\Theta^*\Omega$$
and the last term equals  $\int_\alpha\Flux(\Upsilon_\bullet)$.
 \end{proof}

We conclude this section by proving one (easy) implication of Theorem \ref{thm:second char ham}. As mentioned in the introduction, this implication does not need the hypothesis that $\mathcal L$ and $\mathcal L'$ are Riemannian. 

\begin{proof}[Proof of the ``if'' part of Theorem \ref{thm:second char ham}]
Suppose there exists a Hamiltonian symplectomorphism $\Phi=\Phi_1$, endpoint of a Hamiltonian isotopy $\Phi_\bullet$, such that $\Phi(\mathcal L)=\mathcal L'$. Then define the map $\Upsilon_\bullet:M\times[0,1]\to\mathcal G_\rho$ in such a way that $\Upsilon_0:M\to\mathcal G$ is an embedding with image $\mathcal L$ and
$$\Upsilon_s=\Phi_s\circ \Upsilon_0~.$$
By Remark \ref{rmk:hamiltonian and cartan}, $\Phi_s$ is a (Hamiltonian) symplectomorphism for every $s\in[0,1]$, hence $\Upsilon_s$ is a Lagrangian embedding for all $s$. We claim that $\Flux(\Upsilon_\bullet)$ vanishes in $H^1_{dR}(M,\R)$. Indeed, for every $s$ we have
$$\Upsilon_s^*(\Omega(X_s,\cdot))=\Upsilon_s^*dF_s=df_s$$
by Equation \eqref{eq:hamiltonian function}, where $X_s$ is the vector field generating the Hamiltonian isotopy (and hence $\Upsilon_\bullet$) and $f_s=F_s\circ\Upsilon_s$. Therefore 
$\int_0^1\Upsilon_s^*(\Omega(X_s,\cdot))ds$ is exact, namely $\Flux(\Upsilon_\bullet)=0$.

Using Proposition \ref{prop: flux = diff holonomy}, we have $\hol_{\Upsilon_0}  = \hol_{\Upsilon_1}$. By Lemma \ref{lemma:integrable iff trivial bundle}, this shows that $\mathcal L$ is $\rho$-integrable if and only if $\mathcal L'$ is $\rho$-integrable, and this concludes the proof of the first implication in Theorem \ref{thm:second char ham}.
\end{proof}

\subsection{Conclusion of Theorem \ref{thm:second char ham}} 
 
We are left with the other implication in  Theorem \ref{thm:second char ham}.  Given two Riemannian $\rho$-integrable submanifolds $\mathcal L,\mathcal L'\subset \mathcal G_\rho$, we shall produce $\Phi\in\Ham_c(\mathcal G_\rho,\Omega)$ mapping $\mathcal L$ to $\mathcal L'$. We remark here that the results and methods of \cite{solomon} use stronger topological hypothesis, hence do not apply under our assumptions.

Roughly speaking, the idea is to reduce the problem to finding a deformation in the nearly-Fuchsian manifold $\faktor{\Hyp^{n+1}}{\rho(\pi_1(M))}$ which interpolates between two hypersurfaces of small principal curvatures corresponding to $\mathcal L$ to $\mathcal L'$. For technical reasons, it will be easier to deal with {convex} hypersurfaces  that we defined in Definition \ref{defi:convex immersion}}

\begin{lemma} \label{lemma:convex gauss map diffeo}
Let $M^n$ be a closed oriented manifold, $\rho:\pi_1(M)\to\Isom(\Hyp^{n+1})$ be a nearly-Fuchsian representation and $\widetilde\sigma:\widetilde M\to\Hyp^{n+1}$ be a $\rho$-equivariant embedding. If $\widetilde\sigma$ is convex, then the Gauss map $G_{\widetilde\sigma}^+$ is an {equivariant} diffeomorphism between $\widetilde M$ and the connected component $\Omega_+$ of $\partial\Hyp^{n+1}\setminus\Lambda_\rho$.
\end{lemma}
\begin{proof}
By the same argument as in Section \ref{sec:nearly fuchsian} (see the discussion between Proposition \ref{prop:action free prop disc0} and Proposition \ref{prop:action free prop disc}), $\widetilde\sigma$ extends to a continuous injective map of the visual boundary of $\widetilde M $ with image $\Lambda_\rho$. We can now repeat wordly the argument of Proposition \ref{prop:gauss maps diffeo onto image} to show that, if $B$ is negative semi-definite, then $G_{\widetilde\sigma}^+$ is a diffeomorphism onto its image. To show that {$G_{\widetilde\sigma}^+(\widetilde M)=\Omega_+$}, we repeat instead the proof of Proposition \ref{prop:action free prop disc}. More precisely, one first shows (using tangent horospheres) that every $x\in\Omega_+$ is in the image of $G_{\widetilde\sigma}^+$. Then, by continuity, it suffices to show that every $x\in\Lambda_\rho$ is not on the image of $G_{\widetilde\sigma}^+$. To see this, the last paragraph of the proof of Proposition \ref{prop:action free prop disc} applies unchanged, and when considering tangent $r$-caps we can even take $r=0$, that is, replace $r$-caps by totally geodesic hyperplanes. See Figure \ref{fig:limit2}.
\end{proof}
 
\begin{lemma}\label{lemma:convex interpolation}
Let $M^n$ be a closed oriented manifold and $\rho:\pi_1(M)\to\Isom(\Hyp^{n+1})$ be a nearly-Fuchsian representation. Given two closed hypersurfaces $\Sigma_0$ and $\Sigma_1$ of small principal curvatures in the nearly-Fuchsian manifold $\faktor{\Hyp^{n+1}}{\rho(\pi_1(M))}$, there exists an isotopy $$\upsilon_\bullet\colon M\times[0,1]\to\faktor{\Hyp^{n+1}}{\rho(\pi_1(M))}$$ such that:
\begin{itemize}
\item $\upsilon_s$ is a convex embedding for all $s\in[0,1]$;
\item $\upsilon_0(M)$ is a hypersurface equidistant from $\Sigma_0$;
\item $\upsilon_1(M)$ is a hypersurface equidistant from $\Sigma_1$.
\end{itemize}
\end{lemma} 
\begin{proof}
First of all, let us observe that we can find hypersurfaces equidistant from $\Sigma_0$ and $\Sigma_1$ which are convex. Indeed, by Corollary \ref{cor:equidistant is immersed} and Remark \ref{rmk:embedding in the quotient}, $t$-equidistant hypersurfaces are embedded for all $t\in\R$. Moreover, by compactness, the principal curvatures of $\Sigma_0$ and $\Sigma_1$ are in $(-\epsilon,\epsilon)$ for some $0<\epsilon<1$, and applying Equation \eqref{eq della discordia} we may find  $t_0$ such that the principal curvatures of the $t$-equidistant hypersurfaces are negative for $t\geq t_0$ -- namely, the equidistant hypersurfaces are convex.

Abusing notation, up to taking equidistant hypersurfaces as explained above, we will now assume that $\Sigma_0$ and $\Sigma_1$ are convex, and our goal is to produce $\upsilon_\bullet$ such that $\upsilon_s$ is a convex embedding for all $s\in[0,1]$, $\upsilon_0(M)=\Sigma_0$ and $\upsilon_1(M)=\Sigma_1$. Up to replacing $\Sigma_0$ and $\Sigma_1$ again {with equidistant hypersurfaces}, we {can} also assume that $\Sigma_0\cap\Sigma_1=\emptyset$, that $\Sigma_1$ is in the concave side of $\Sigma_0$, and that the equidistant surfaces from $\Sigma_1$ which intersect $\Sigma_0$ are all convex. We call $\mathcal A$ the region of $\faktor{\Hyp^{n+1}}{\rho(\pi_1(M))}$ bounded by $\Sigma_0$ and containing $\Sigma_1$.

Let us now consider the (signed) distance functions $r_0$ and $r_1$ from $\Sigma_0$ and $\Sigma_1$ respectively, chosen in such a way that both $r_0$ and $r_1$ are positive functions on the concave side of $\Sigma_0$ and $\Sigma_1$ respectively. Again by Corollary \ref{cor:equidistant is immersed} and Remark \ref{rmk:embedding in the quotient}, these functions are smooth and have nonsingular differential everywhere. Let us denote by $\nu_i$ the gradient of $r_i$. {The vector field $\nu_i$} has unit norm and is tangent to the orthogonal foliations of $\Sigma_i$ which have been described in the proof of Proposition \ref{prop:action free prop disc}. (Proposition \ref{prop:action free prop disc} describes the foliation in the universal cover, but it clearly descends to the quotient $\faktor{\Hyp^{n+1}}{\rho(\pi_1(M))}$.)

We claim that both $r_i$'s are convex functions in the region $\mathcal A$,  i.e. that their Hessians are positive semi-definite, as a consequence of the fact that the level sets of $r_i$ in $\mathcal A$ are all convex. Recall that the Riemannian Hessian of a smooth function $f:\mathcal A\to\R$ is the symmetric 2-tensor defined as 
\begin{equation}
\label{eq: def hessiano}
\nabla^2f(X,Y)=\partial_X (\partial_Y f)-\partial_{D_XY}f~,
\end{equation}
where $X,Y$ are local vector fields and $D$ is the ambient Levi-Civita connection as usual. Clearly $\nabla^2r_i(\nu_i,\nu_i)=0$ since $r_i$ is linear along the integral curves of $\nu_i$ and such integral curves, which are the leaves of the orthogonal foliation described above, are geodesics. Moroever, if $X$ is a vector field tangent to the level sets of $r_i$, then $\nabla^2r_i(X,\nu_i)=0$: indeed the first term {in the RHS of Equation \eqref{eq: def hessiano}} vanishes  because $r_i$ is linear along the integral curves, and the second term as well, because $D_{X}\nu_i=-B_i(X)$ is tangential to the level sets of $r_i$ and thus $\partial_{D_{X}\nu_i}r_i=0$.

To conclude that $\nabla^2r_i$ is positive semi-definite, it remains to show that $\nabla^2r_i(X,X)\geq 0$ for all $X$ tangent to the level sets. 
It is more instructive to perform this computation in the general setting of a smooth function  $f:\mathcal A\to\R$. Since the unit normal vector field to the level set of $f$ is $\nu=\frac{Df}{\|Df\|}$,{with $Df$ being the gradient of $f$, for all $X,Y$ vector fields tangent to the fibers, we} get that:
\begin{equation}\label{eq:hessian and II}
\nabla^2f(X,Y)=-\partial_{D_XY}f=-\langle D_XY,\nu\rangle\partial_\nu f=-\|Df\|\II(X,Y)~,
\end{equation}
where in the last step we used that $\partial_\nu f=\langle Df,\nu\rangle=\|Df\|$, and $\II$ denotes the second fundamental form of the level sets of $f$. When $f=r_i$, in the region $\mathcal A$ the level sets of $r_i$ are convex, hence {$\II$ is negative semi-definite and} $\nabla^2r_i(X,X)\geq 0$.

We remark that Equation \eqref{eq:hessian and II} also shows that, if $f$ is a convex function, then its level sets are convex hypersurfaces as long as $Df\neq 0$. We shall now apply this remark to the zero set of the function $f_s=(1-s)r_0+sr_1$ for $s\in[0,1]$. The differential of $f_s$ never vanishes, for $\|Dr_0\|=\|Dr_1\|=1$, hence $Df_s=0$ is only possible for $s=\frac 1 2$ if $Dr_0=-Dr_1$: nevertheless, this cannot happen  since the geodesics with initial vector $Dr_0=\nu_0$ and $Dr_1=\nu_1$ both have final endpoint in $\Omega_+$ and initial endpoint in $\Omega_-$ by (the proof of) Proposition \ref{prop:action free prop disc}. Hence $\{f_s=0\}$ is an embedded hypersurface for all $s$. Observe moreover that
$$\{f_s=0\}=\bigg\{\frac{r_0}{r_0-r_1}=s\bigg\}~.$$
Since $\Sigma_0\cap \Sigma_1=\emptyset$, $r_0-r_1$ never vanishes, and this shows that the hypersurfaces $\{f_s=0\}$ provide a foliation of the region between $\Sigma_0$ and $\Sigma_1$, which is contained in $\mathcal A$. Since both $r_0$ and $r_1$ are convex functions in $\mathcal A$, 
$$\nabla^2f_s(X,X)=(1-s)\nabla^2r_0(X,X)+s\nabla^2r_1(X,X)\geq 0$$
for every $X$, hence $f_s$ is convex. As remarked just after Equation \eqref{eq:hessian and II}, since $\|Df_s\|\neq 0$, $\{f_s=0\}$ is a convex hypersurface.

It is not hard now to produce $\upsilon_\bullet:M\times[0,1]\to\faktor{\Hyp^{n+1}}{\rho(\pi_1(M))}$ such that $\upsilon_s(M)=\{f_s=0\}$. For instance one can flow along the vector field $\frac{DF}{\|DF\|^2}$ where $F=\frac{r_0}{r_0-r_1}$. Alternatively one can apply Lemma \ref{lemma:convex gauss map diffeo} to infer that the Gauss maps $G^+_{\widetilde\sigma}$ in the universal cover  induce diffeomorphisms of each hypersurface $\{f_s=0\}$ with $\faktor{\Omega_+}{\rho(\pi_1(M))}\cong M$, and define $\upsilon_s$ as the inverse map.
\end{proof}

\begin{proof}[Proof of the ``only if'' part of Theorem \ref{thm:second char ham}]
Suppose $\mathcal L$ and $\mathcal L'$ are $\rho$-integrable Riemannian submanifolds in $\mathcal G_\rho$. Then there exists hypersurfaces $\Sigma$ and $\Sigma'$ in $\faktor{\Hyp^{n+1}}{\rho(\pi_1(M))}$ whose Gauss map image induce $\mathcal L$ and $\mathcal L'$ respectively. We now apply Lemma \ref{lemma:convex interpolation} and find $\upsilon_\bullet$ such that $\upsilon_s$ is convex for every $s$, and the images of $\upsilon_0$ and $\upsilon_1$ are equidistant hypersurfaces from $\Sigma$ and $\Sigma'$ respectively. Define 
$\Upsilon_\bullet:M\times[0,1]\to\mathcal G_\rho$ so that $\Upsilon_s$ is the map into $\mathcal G_\rho$ induced  by the Gauss map of the lifts on the universal cover $\widetilde\upsilon_s:\widetilde M\to\Hyp^{n+1}$. As a consequence of Lemma \ref{lemma:convex gauss map diffeo} the Gauss map of each $\widetilde\upsilon_s$ is an embedding with image in $\Omega_+\times\partial\Hyp^{n+1}\setminus\Delta$ in $\G{n+1}$. 
Repeating the same argument of Remark \ref{rmk:embedding in the quotient}, $\Upsilon_s:M\to\mathcal G_\rho$ is an embedding for every $s$. As a particular case, by Lemma \ref{prop:gauss map invariant normal evo}, $\Upsilon_0(M)=\mathcal L$ and $\Upsilon_1(M)=\mathcal L'$.

By construction for every $s\in[0,1]$ the image of $\Upsilon_s$ is a $\rho$-integrable embedded submanifold in $\mathcal G_\rho$. Let us denote by $\mathcal L_s$ such submanifold. It follows (Lemma \ref{lemma:integrable iff trivial bundle}) that {$\hol_{\Upsilon_s}$} is trivial for all $s$.
By Proposition \ref{prop: flux = diff holonomy}, together with the definition of $\Flux$, we have that
$$\int_0^{s}\Upsilon_r^*(\Omega(X_r,\cdot))dr=0\in H^1_{dR}(M,\R)$$
for all $s$, where $X_s$ is the vector field generating $\Upsilon_s$. Hence necessarily the cohomology class of $\Upsilon_s^*(\Omega(X_s,\cdot))$  in $H^1_{dR}(M,\R)$ is trivial for all $s$. We can therefore find a smooth function $f_\bullet:M\times[0,1]\to\R$ such that $df_s=\Upsilon_s^*(\Omega(X_s,\cdot))$ for all $s$. Pushing forward $f_s$ by means of $\Upsilon_s$, we have defined smooth functions on $\mathcal L_s$ whose differential equals $\Omega(X_s,\cdot)$. Let us extend them to $F_\bullet:\mathcal G_\rho\to\R$ so that $F_\bullet$ is compactly supported and $F_s\circ \Upsilon_s=f_s$. 

Let $\widehat X_s$ be the symplectic gradient of $F_s$, namely
$$dF_s=\Omega(\widehat X_s,\cdot)~,$$
and let $\Phi_s$ be the flow generated by $\widehat X_s$. From
$d(F_s\circ \Upsilon_s)=df_s$
we see that 
$$\Omega(\widehat X_s,d\Upsilon_s(V))=\Omega(X_s,d\Upsilon_s(V))$$
for all $V\in T_p M$. This implies that $\Omega(\widehat X_s-X_s,\cdot)$ vanishes identically along the Lagrangian submanifold $\mathcal L_s$. By non-degeneracy of the symplectic form $\Omega$, $\widehat X_s-X_s$ is tangential to  $\mathcal L_s$. Therefore 
$\Phi_s\circ\Upsilon_0$ and $\Upsilon_s$ differ by pre-composition with a diffeomorphism $\phi_s$ of $M$ (which is indeed obtained as the flow on $M$ of the vector field $\Upsilon_s^*(\widehat X_s-X_s)$). This shows that $\Phi_s(\mathcal L)=\mathcal L_s$. In particular, $\Phi=\Phi_1$ is the desired compactly supported Hamiltonian symplectomorphism of $\mathcal G_\rho$ such that $\Phi(\mathcal L)=\mathcal L'$. 
\end{proof}

\appendix

\section{Evolution by geometric flows}\label{app:geometric flows}

The aim of this Appendix is to provide {a relationship between certain geometric flows for hypersurfaces in $\Hyp^{n+1}$ and their induced flows in $T^1 \Hyp^{n+1}$ and in $\G{n+1}$.}

Let $M=M^n$ be an oriented manifold. Let $\sigma \colon M \times (-\varepsilon,\varepsilon) \to \Hyp^{n+1}$ be a smooth map such that $\sigma_t= \sigma(\cdot, t)$ is an immersion with small principal curvatures for all $t$, and let $\nu=\nu(x,t)$ be the normal vector field.

\begin{prop} 
\label{Prop: flows on G and T1Hn}
Let $f\colon M\times(-\varepsilon, \varepsilon)\to \R$ be a smooth map such that 
\[
\frac{d}{dt} \sigma_t  = f_t \nu_t,
\]
{and let $\zeta_t:=\zeta_{\sigma_t}:M\to T^1\Hyp^{n+1}$ be the lift to $T^1\Hyp^{n+1}$, $G_t:=G_{\sigma_t}:M\to\G{n+1}$ be the Gauss map.} Then,
\begin{align}
        \label{eq flusso in T1}
        \frac {d}{dt} \zeta_t &= - d\zeta_t ( B_t ( \overline \nabla^t f_t) ) - J  ( d\zeta_t (\overline \nabla^t f_t)) + f_t \chi \\
        \label{eq flusso in G}
        \frac {d}{dt} G_{t} &= - d G_{t} ( B_t ( \overline \nabla^t f_t) ) - \mathbb J   (d G_{t} (\overline \nabla^t f_t) )
\end{align}
where $\overline \nabla^t f_t$ is the gradient of $f_t$ with respect to the first fundamental form $\overline\I_t=G_t^*\GG$ {and $B_t$ is the shape operator of $\sigma_t$.}
\end{prop}

{As a preliminary step to prove Proposition \ref{Prop: flows on G and T1Hn}, we  compute} the variation in time of the normal vector field. Recalling that $D$ denotes the Levi-Civita connection on $\Hyp^{n+1}$, we show:

\begin{equation}\label{eq appendix 1}
D_{\frac{d\sigma_t}{dt}  } {\nu_t}= - d\sigma( \nabla^t f_t)~,
\end{equation}
where now $\nabla^t$ denotes the gradient with respect to the first fundamental form $\I_t$ of $\sigma_t$.
On the one hand, by metric compatibility,
\[
\inner{D_{\frac{d\sigma_t}{dt} }\nu_t, \nu_t }= \frac 1 2 \partial_{\frac{d\sigma_t}{dt} } \inner{\nu_t,\nu_t}= 0
\]
hence $D_{\frac{d\sigma_t}{dt} }\nu_t$ is tangent to the hypersurface.

On the other hand, let $X$ be any vector field over $M$. Since $X$ and $\frac{\partial}{\partial t}$ commute on $M\times(-\varepsilon,\varepsilon)$, we have that:
\begin{align*}
    \inner{D_{\frac{d\sigma_t}{dt} } \nu_t, d\sigma_t(X)}&= \partial_{\frac{d\sigma_t}{dt} } \inner{\nu_t, d\sigma_t (X)} - \inner{\nu_t, D_{ \frac{d\sigma_t}{dt}  } (d\sigma_t (X) )}\\
    &= 0 - \inner{\nu_t, D_{ \frac{d\sigma_t}{dt}  } (d\sigma_t (X) )}\\
    &=-\inner{\nu_t, D_{d\sigma_t(X)} (f_t \nu_t) }\\
    &=- X( f_t) -f_t \inner{\nu_t, D_{d\sigma_t(X)} \nu_t}\\
    &=-X(f_t) =- \I_t(\nabla^t f_t, X)~.
\end{align*}
This shows Equation \eqref{eq appendix 1}. As a result, in the hyperboloid model \eqref{eq:modelT} we have:
\begin{equation}\label{eq appendix 2}
\frac{d}{dt} \zeta_t = \bigg(\frac d {dt} \sigma_t, D_{\frac{d \sigma_t}{dt}} \nu\bigg)= (f_t \nu_t, - d\sigma_t (\nabla^t f_t))
\end{equation}

\begin{proof}[Proof of Proposition \ref{Prop: flows on G and T1Hn}]
Let $e_{t,1}, \dots, e_{t,n}$ be a local $\overline{\I}_t$-orthonormal frame diagonalizing $B_t$, so $B_t(e_{t,k})= \lambda_{t,k} e_{t,k}$, and recall that $\chi$ denotes as usual the infinitesimal generator of the geodesic flow on $T^1 \Hyp^{n+1}$.
By definition of $g_{T^1 \Hyp^{n+1}}$ and of $J$ 
\begin{equation}
\label{eq: base T1 Hn}
\big(d\zeta_{t}( e_{t,1}), \dots, d\zeta_{t}(e_{t,n}), \chi, J d\zeta_{t}( e_{t,1}), \dots, Jd \zeta_{t}( e_{t,n})\big)
\end{equation}
defines at each point of the image an orthonormal basis for the tangent space of $T^1 \Hyp^{n+1}$, with the former $n+1$ vectors having norm $1$ and the latter $n$ vectors having norm $-1$.

We prove Equation \eqref{eq flusso in T1}, then Equation \eqref{eq flusso in G} follows after observing that
\[
\frac {d}{dt} G_{t} = (dG_{t})\bigg(\frac \partial {\partial t}\bigg)= (d\mathrm p \circ d\zeta_t)\bigg(\frac {\partial} {\partial t} \bigg)= d\mathrm p \bigg(\frac{d}{dt} \zeta_t\bigg).
\]
We show that LHS and RHS of \eqref{eq flusso in T1} have the same coordinates with respect to the basis \eqref{eq: base T1 Hn}. By  Equations \eqref{eq: differential of sigma tilde} and \eqref{eq appendix 2},
\begin{align*}
    g_{T^1 \Hyp^{n+1} } \bigg( \frac d {dt} \zeta_t, J d\zeta_t ( e_{t,k}) \bigg) &= f_t \inner{\nu_t, - d\sigma_t (B_t (e_{t,k})) } - \inner{- d\sigma_t (\nabla^t f_t), d\sigma_t (e_{t,k})} \\
    &= \inner{  d\sigma_t (\nabla^t f_t), d\sigma_t (e_{t,k})}= \partial_{e_{t,k}} f_t \\
    & 
    = \overline \I_t (\overline {\nabla}^t f_t,  e_{t,k})
    = g_{T^1 \Hyp^{n+1}} (- J d\zeta_t (\overline {\nabla}^t f_t), J d\zeta_t (e_{t,k}) )~.
\end{align*}
Similarly, recalling that $B_t$ is self-adjoint with respect to both $\I_t$ and $\overline \I_t$, one has 
\begin{align*}
    g_{T^1 \Hyp^{n+1}} \bigg( \frac d {dt}\zeta_t, d\zeta_t (e_{t,k}) \bigg)&= \inner{f_t \nu_t, d\sigma_t (e_{t,k})} - \inner{- d\sigma_t(\nabla^t f_t), - d\sigma_t (B_t ( e_{t,k} )) }  \\
    &= - \inner{d\sigma_t (\nabla^t f_t),  d\sigma_t (B_t(e_{t,k}))}\\
    &= - \I_t (\nabla^t f_t,  B_t(e_{t,k}))
    = - \overline{\I}_t (\overline{\nabla}^t f_t,  B_t(e_{t,k}))\\
    &= -\overline{\I}_t (B_t( \overline {\nabla}^t f_t), e_{t,k})
    = g_{T^1 \Hyp^{n+1}} (- d\zeta_t (B_t (\overline {\nabla}^t f_t) ), d\zeta_t (e_{t,k}) )~.
\end{align*}
Finally, 
\begin{align*}
    g_{T^1 \Hyp^{n+1}} \bigg(\frac d {dt} \zeta_t, \chi\bigg)= f_t\inner{\nu_t,\nu_t}= f_t  = g_{T^1 \Hyp^{n+1}}(f_t\chi, \chi)
\end{align*}
and the proof follows.
\end{proof}

An interesting corollary of Proposition \ref{Prop: flows on G and T1Hn} involves mean curvature {flow}. Directly by Proposition \ref{Prop: formula H in G}, one has the following.

\begin{cor}
The  flow in $\Hyp^{n+1}$ defined by
\[
\frac d {dt} \sigma_t = \frac 1 n \sum_{k=1}^n \arctanh(\lambda_{t,k}),
\]
on hypersurfaces of small principal curvatures, induces in $\G{n+1}$ the mean curvature flow up to a horizontal factor, namely
\[
\frac d {dt} G_{t} =\mathrm{\overline H_t} +   B_t (\mathbb J (\mathrm{\overline H_t})) .
\]
\end{cor}

\bibliographystyle{alpha}
\bibliographystyle{ieeetr}
\bibliography{eesbiblio.bib}

\end{document}